\documentclass[12pt]{amsart}

\usepackage[top=1in, bottom=1in, left=1in, right=1in]{geometry}
\usepackage{times}
\usepackage{tikz}
\usetikzlibrary{shapes,shadows,arrows,decorations.markings}

\usepackage{tkz-berge}

\tikzstyle{circ}=[draw,circle, text centered]

\usepackage{amssymb,amsmath,amsthm,mathtools,mathrsfs}
\usepackage{enumitem}
\usepackage{bbm}

\newcommand{\one}{\mathbbm{1}}

\setlist[1]{itemsep=0.5em, topsep=0.5em}

\usepackage{color,epsfig,xcolor}
\definecolor{red}{rgb}{1,0,0}
\definecolor{orange}{rgb}{0.7,0.3,0}
\definecolor{blue}{rgb}{0,.3,.7}
\definecolor{green}{rgb}{0,.6,.4}

\definecolor{dblue}{rgb}{0,.3,.7}

\PassOptionsToPackage{hyphens}{url}\usepackage[colorlinks=true, linkcolor=dblue, citecolor=teal, urlcolor=gray]{hyperref}   

\renewcommand{\le}{\leqslant}
\renewcommand{\leq}{\leqslant}
\renewcommand{\ge}{\geqslant}
\renewcommand{\geq}{\geqslant}{}
\newcommand{\ep}{\varepsilon}

\newcommand{\e}{\operatorname{e}}

\numberwithin{equation}{section}

\renewcommand{\tilde}{\widetilde}

\theoremstyle{plain}
\newtheorem{maintheorem}{Theorem}
\newtheorem{lemma}{Lemma}[section]
\newtheorem{theorem}[lemma]{Theorem}
\newtheorem{corollary}[lemma]{Corollary}
\newtheorem{proposition}[lemma]{Proposition}

\theoremstyle{remark}
\newtheorem{rem}{Remark}[section]
\newtheorem*{rem*}{Remark}
\newtheorem*{rems*}{Remarks}

\theoremstyle{definition}
\newtheorem{dfn}[lemma]{Definition}

\newcommand{\N}{\mathbb{N}}
\newcommand{\Z}{\mathbb{Z}}
\newcommand{\Q}{\mathbb{Q}}
\newcommand{\R}{\mathbb{R}}

\renewcommand{\P}{\mathbb{P}}

\newcommand{\CA}{\mathcal{A}}
\newcommand{\CB}{\mathcal{B}}
\newcommand{\CC}{\mathcal{C}}
\newcommand{\CD}{\mathcal{D}}
\newcommand{\CE}{\mathcal{E}}

\newcommand{\CK}{\mathcal{K}}
\newcommand{\CL}{\mathcal{L}}
\newcommand{\CM}{\mathcal{M}}
\newcommand{\CN}{\mathcal{N}}

\newcommand{\CP}{\mathcal{P}}
\newcommand{\CQ}{\mathcal{Q}}
\newcommand{\CR}{\mathcal{R}}
\newcommand{\CS}{\mathcal{S}}
\newcommand{\CT}{\mathcal{T}}

\newcommand{\CV}{\mathcal{V}}
\newcommand{\CW}{\mathcal{W}}

\newcommand{\nn}{\nonumber \\}
\newcommand{\ds}{\displaystyle}

\newcommand{\dee}{\,\mathrm{d}}

\newcommand{\fl}[1]{\left\lfloor#1\right\rfloor}

\newcommand{\eps}{\varepsilon}
\renewcommand{\phi}{\varphi}
\renewcommand{\rho}{\varrho}

\newcommand{\bs}\boldsymbol{}

\renewcommand{\bar}[1]{\overline{#1}}

\renewcommand{\mod}[1]{\,(\mathrm{mod}\,#1)}

\newcommand{\bg}{\big}

\DeclareMathOperator{\meas}{meas}

\definecolor{red}{rgb}{1,0,0}

\definecolor{orange}{rgb}{0.7,0.3,0}

\definecolor{blue}{rgb}{.2,.6,.75}

\definecolor{green}{rgb}{.4,.7,.4}

\newcommand{\lbr}{[}
\newcommand{\rbr}{]}

\begin{document}

\title{Erd\H os's integer dilation approximation problem and GCD graphs}

\author[D. Koukoulopoulos]{Dimitris Koukoulopoulos}
\address{D\'epartement de math\'ematiques et de statistique\\
Universit\'e de Montr\'eal\\
CP 6128 succ. Centre-Ville\\
Montr\'eal, QC H3C 3J7\\
Canada}
\email{dimitris.koukoulopoulos@umontreal.ca}

\author[Y. Lamzouri]{Youness Lamzouri}
\address{
Universit\'e de Lorraine, CNRS, IECL,  and  Institut Universitaire de France,
F-54000 Nancy, France}
\email{youness.lamzouri@univ-lorraine.fr}

\author[J. D. Lichtman]{Jared Duker Lichtman}
\address{Department of Mathematics, Stanford University, Stanford, CA, USA}
\email{jared.d.lichtman@gmail.com}


\date{\today}

\begin{abstract}
Let $\mathcal{A}\subset\mathbb{R}_{\geqslant1}$ be a countable set such that $\limsup_{x\to\infty}\frac{1}{\log x}\sum_{\alpha\in\mathcal{A}\cap[1,x]}\frac{1}{\alpha}>0$. We prove that, for every $\varepsilon>0$, there exist infinitely many pairs $(\alpha, \beta)\in \CA^2$ such that $\alpha\neq \beta$ and $|n\alpha-\beta| <\varepsilon$ for some positive integer $n$. This resolves a problem of Erd\H os from 1948. A critical role in the proof is played by the machinery of GCD graphs, which were introduced by the first author and by James Maynard in their work on the Duffin--Schaeffer conjecture in Diophantine approximation.
\end{abstract}

\maketitle

\section{Introduction} 
Given a set $\mathcal{A}\subset\mathbb{R}_{>0}$ and a positive number $\eps$, let us consider the following approximation problem of a Diophantine nature: are there distinct $\alpha,\beta\in\CA$ and an integer $n$ such that 
\begin{equation}
	\label{eq:solutions}
|n\alpha-\beta|<\varepsilon\,?
\end{equation}
Equivalently, we seek to find a fraction $\beta/\alpha\neq 1$ whose numerator and denominator both lie in $\CA$, and which has the property that  $\|\beta/\alpha\|<\eps/\alpha$, where $\|\cdot\|$ denotes the distance to the nearest integer. 

Evidently, if $\CA$ has an accumulation point in $\R$, then \eqref{eq:solutions} has infinitely many solutions with distinct $\alpha,\beta\in\CA$ and with $n=1$. So, we will be assuming throughout that $\CA$ has no accumulation points, that is to say it is a {\it discrete set}. In particular, $\CA$ is a countable set.

In 1948, Erd\H os \cite[Page 692]{Erd48} asked whether it is possible to find solutions to \eqref{eq:solutions} if $\CA$ is ``large enough''. Motivated by his work \cite{Erd35} and that of Behrend \cite{Be} on integer primitive sequences (see Definition \ref{dfn:primitive} below and the discussion following it), he proposed that this might indeed be possible if $\CA$ satisfies one of the following conditions:
\begin{equation}\label{Eq:Condition1Erdos}
\sum_{\alpha\in\CA,\,\alpha\ge2}\frac{1}{\alpha\log\alpha}=\infty
\end{equation}
or
\begin{equation}\label{Eq:Condition2Erdos}
\limsup_{x\to\infty}\frac{1}{\log x} \sum_{\alpha\in \CA\cap[1,x]}\frac{1}{\alpha} > 0.
\end{equation}
We shall refer to this question as the {\it integer dilation approximation problem}.  This problem was mentioned\footnote{In \cite{Erd97}, which was published in 1997, shortly after Paul Erd\H os passed away, he wrote ``I offer \$500 for settling this annoying diophantine problem.''} again in \cite{Erd61, Erd63, Erd73, Erd75, Erd77, Erd80, Erd92, Erd97, ES78, ESS}. In particular, Erd\H os, S\'ark\"ozy and Szemer\'edi \cite{ESS} asked in 1968 whether one can prove  the existence of infinitely many solutions to \eqref{eq:solutions} under the stronger assumption that 
\[
\liminf_{x\to\infty}\frac{1}{x} \sum_{\alpha\in \CA\cap[1,x]}1 > 0.
\]

In this paper, we resolve Erd\H os's integer dilation approximation problem when $\CA$ satisfies the second condition  \eqref{Eq:Condition2Erdos}.

\begin{maintheorem}\label{thm:solutions}
Let $\mathcal{A}\subset\mathbb{R}_{>0}$ be a discrete set such that
\[
\limsup_{x\to\infty} \frac{1}{\log x} \sum_{\alpha\in\CA\cap[1,x]} \frac{1}{\alpha}>0 .
\]
Then, for every $\eps>0$, there exists a pair $(\alpha, \beta)\in \CA^2$ such that $\alpha\neq \beta$ and $|n\alpha-\beta| <\varepsilon$ for some positive integer $n$. 
\end{maintheorem}

\begin{rem*}
	In fact, the above theorem implies trivially that, for every $\eps>0$, there must exist infinitely many pairs $(\alpha,\beta)$ as above. Indeed, once we locate the first one, say $(\alpha_1,\beta_1)$, we apply the theorem to $\CA'\coloneqq\CA\setminus\{\alpha_1,\beta_1\}$ to locate a second suitable pair $(\alpha_2,\beta_2)$. Continuing this way, we may find an infinite number of such pairs.
\end{rem*}

It is worth noting that Erd\H os originally stated his problem in its contrapositive form; the above formulation of the problem appeared first in a paper of Erd\H os and S\'ark\"ozy \cite{ES78}, and subsequently in Haight's 1988 work \cite{Ha88}, who proved Theorem \ref{thm:solutions} in the special case when the ratios $\alpha/\beta$ with distinct $\alpha,\beta\in\CA$ are all irrational. As a matter of fact, under this assumption, Haight proved the stronger estimate
\begin{equation}\label{eq:Haight}
	\lim_{x\to\infty} \frac{1}{x} \sum_{a\in \CA\cap[1,x]}1=0.
\end{equation}
Erd\H os's motivation for stating the problem in the contrapositive form comes from its connection to {\it primitive sets} of integers.  Indeed, if $\CA\subset\N$ and $\eps\in(0,1]$, then the only way to have a solution to \eqref{eq:solutions} is if $\beta=n\alpha$, that is to say, if $\alpha$ divides $\beta$. Let us now recall the definition of a primitive set:

\begin{dfn}[primitive set]\label{dfn:primitive}
We say that a set $\CA\subset\N$ is {\it primitive} if $a\nmid b$ for all distinct $a,b\in\CA$. 
\end{dfn}

Early on, many experts, including Chowla, Davenport, and Erd\H{o}s, believed that primitive sets all have natural density equal to zero. But to their great surprise, in 1934 Besicovitch \cite{Bes34} constructed primitive sets with upper natural density $\frac{1}{2}-\eps$, for any $\eps>0$. (This is sharply different from the situation in Haight's estimate \eqref{eq:Haight}.) On the other hand, in 1935 Behrend \cite{Be} and Erd\H{o}s \cite{Erd35} proved that primitive sets of integers all have logarithmic density zero, and hence lower natural density zero. In fact, Erd\H{o}s \cite{Erd35} showed the stronger result that 
	\begin{equation}
		\label{eq:Erdos}
		\sum_{a\in \CA} \frac{1}{a\log a} <\infty,
	\end{equation}
for any primitive set of integers $\CA$, while Behrend's result \cite{Be} states that 
	\begin{equation}
		\label{eq:Behrend}
	\sum_{a\in \CA\cap[1,x]}\frac{1}{a}\ll \frac{\log x}{\sqrt{\log\log x}}, 
	\end{equation}
for such sets. More recently, Ahlswede, Khachatrian and S\'ark\"ozy \cite{AKS04} strengthened Behrend's estimate by showing that
	\begin{equation}
		\label{eq:AKS bound}
	\sum_{a\in \CA\cap[y,yx]}\frac{1}{a}\ll \frac{\log x}{\sqrt{\log\log x}}
	\end{equation}
uniformly for all primitive sets $\CA$ and all $x\ge3$ and $y\ge1$.

\subsection{Notation}\label{sec:notation}
Given $\CS\subseteq\R$ and $x\in\R$, we let $\mathcal S_{\le x} \coloneqq \mathcal S\cap(-\infty,x]$ and $\mathcal S_{> x} \coloneqq \mathcal S\cap(x,\infty)$. 

For a positive integer $n$, we denote by $P^{+}(n)$ and $P^{-}(n)$ its largest and the smallest prime factors respectively, with the standard conventions that $P^{+}(1)=1$ and $P^{-}(1)=\infty$. In addition, we write $\omega(n)$ for the number of distinct prime factors of $n$.

Given $k\in\N$, we let $\tau_k$ denote the $k$-th divisor function. 

We write $f(x) \lesssim g(x)$ (respectively $f(x)\gtrsim g(x)$) if $f(x)\leq (1+o(1))g(x)$ (respectively $f(x)\geq (1+o(1))g(x)$) as $x\to \infty$.

Lastly, we write $\meas$ to denote the Lebesgue measure on $\R$. In addition, given $T>0$ and a Lebesgue-measurable set $\CS$, we let 
\begin{equation}\label{Eq:DefinitionNormalizedMeasure}
	\P_T(\CS)\coloneqq \frac{\meas(\CS\cap[0,T])}{T}.
\end{equation}

\subsection{Ideas from the proof} 
The connection of Erd\H os's problem to primitive sets of integers plays a crucial role in our proof, as does Haight's work. In fact, one may view our proof as an instance of the {\it structure versus randomness} philosophy. 

Let us consider a discrete set $\CA\subset\R_{\ge1}$ and a number $\eps\in(0,1]$ such that there are no solutions to the inequality $|n\alpha-\beta|<\eps$ with distinct $\alpha,\beta\in\CA$ and with $n\in\N$. In order to establish Theorem \ref{thm:solutions}, we must prove that 
\begin{equation}
	\label{eq:log density zero}
	\sum_{\alpha\in\CA\cap[1,x]} \frac{1}{\alpha}=o(\log x) \quad(x\to\infty).
\end{equation}
We have two extreme cases:
\begin{itemize}
	\item {\it Structured sets:} these are primitive sets or small ``perturbations'' of them. By this, we mean that $\CA \subseteq \gamma\Q_{\ge1}\coloneqq \{\gamma\rho:\rho\in\Q_{\ge1}\}$ for some $\gamma\in\R_{\ge1}$, and that the set of denominators 
	\[
	\CQ = \big\{q\in\N: \exists a\in\N\ \mbox{such that $\gcd(a,q)=1$ and $\gamma a/q\in\CA$} \big\}
	\]
	is sparse. For each given $q\in\CQ$, the set $\{a\in\N: \gcd(a,q)=1,\ a/q\in\CA\}$ is primitive, so we may apply  Behrend's estimate \eqref{eq:Behrend} to it. If $\CQ$ is sparse enough, then we deduce that \eqref{eq:log density zero} holds.
	
	\item {\it Random sets:} these are sets $\CA$ for which all ratios $\alpha/\beta$ are irrational, or perhaps rational numbers of large height\footnote{The correct notion of ``large height'' is with respect to the size of $\alpha$ and $\beta$. The crucial quantity to consider is $[\alpha,\beta] \coloneqq H(\alpha/\beta)/\max\{\alpha,\beta\}$, where $H(\cdot)$ is the height function; see  Definition \ref{dfn:bracket} below.}. We can expect to be able to handle such sets by a suitable variant of Haight's proof and to prove that \eqref{eq:log density zero} holds for them too.
\end{itemize}
Roughly speaking, the strategy of the proof is to show that either $\CA$ consists almost $100\%$ of a random set, or that a positive proportion of $\CA$ is structured.\footnote{An astute reader might have noticed that the notions of structured and random sets are not fully orthogonal to each other, because in the latter case we allow the ratios $\alpha/\beta$ to be rationals of large height. Consider for example the case of a set $\CA=\bigcup_{j\ge1}\{a/q_j:a\in\CS_j,\ \gcd(a,q_j)=1\}$, where $(q_j)_{j=1}^\infty$ is a sparse sequence of prime numbers and $\CS_j\subset\Z\cap(q_jx_{j-1},q_jx_j]$ is primitive for each $j$, with the sequence $(x_j)_{j=1}^\infty$ growing at some appropriate rate. This is a structured set, but it could also be a random set, unless there are elements in some $\CS_j$ with a large GCD. This example is very useful to keep in mind for the discussion later on.}

Let us now provide some more concrete context. As in Haight's work, the starting observation is that, for any choice $\CA'$ of a subset of $\CA$, the union
\begin{equation}\label{eq:starting point}
\bigcup_{\alpha\in\CA'} \bigcup_{n\in\N} (n\alpha-\eps/2,n\alpha+\eps/2) 
\end{equation}
avoids all intervals $(\alpha-\eps/2,\alpha+\eps/2)$ with $\alpha\in\CA\setminus\CA'$. Indeed, this is a simple consequence of our assumption that there are no solutions to the inequality $|n\alpha-\beta|<\eps$ with distinct $\alpha,\beta\in\CA$ and with $n\in\N$. Thus, if we could show that the union in \eqref{eq:starting point} covers $\sim 100\%$ of the positive real numbers, we would deduce that $\#(\CA\setminus\CA')\cap[0,T]=o(T)$ as $T\to\infty$. 

To prove that the union in \eqref{eq:starting point} has large measure, we use the {\it second moment method} (cf.~Section \ref{sec:second moment}). Suppose we are given certain events $E_1,E_2,\dots$ in a probability space and we need to show that, with high probability, at least one of them occurs. The second moment method says that this  is indeed true if the sum of the probabilities of the events $E_i$ is large and the events $E_i$ are \emph{negatively correlated on average} (see Definition \ref{dfn:uncorrelated} and equation \eqref{Eq:UncorrelatedAverage} below).  

Our probability space is the interval space $[0,T]$ with $T$ a parameter tending fast to infinity, and the probability measure is given by $\P_T$, which we defined in \eqref{Eq:DefinitionNormalizedMeasure}. To construct the events $E_i$, we fix a judicious choice of distinct elements $\alpha_1,\dots,\alpha_J$ of $\CA$; these will form the special set $\CA'$. Haight considered the events
\[
H_i = \bigcup_{n\in\N} (n\alpha_i-\eps/2,n\alpha_i+\eps/2) 
\]
and proved that if  $\alpha_i/\alpha_j$ is irrational, then $\P_T(H_i\cap H_j)\sim \P_T(H_i)\P_T(H_j)$ as $T\to\infty$. However, the events $H_i$ could be highly correlated when $\alpha_i/\alpha_j$ is a rational number of small height (see Section \ref{sec:problem_sets} below). For this reason, we borrow a key idea from the proof of Erd\H os's estimate \eqref{eq:Erdos} and we take 
\[
E_i = \bigcup_{n\in\N,\ P^-(n)>\alpha_i}  (n\alpha_i-\eps/2,n\alpha_i+\eps/2).
\]
For these events too, Haight's argument and inclusion-exclusion allows us to show that if $\alpha_i/\alpha_j$ is irrational, then $\P_T(E_i\cap E_j)\sim \P_T(E_i)\P_T(E_j)$ as $T\to\infty$.  In addition, we may use an elementary argument to calculate $\P_T(E_i\cap E_j)$ when $\alpha_i/\alpha_j$ is a rational number. It turns out that $\P_T(E_i\cap E_j)\lesssim \P_T(E_i)\P_T(E_j)$, unless $\alpha_i/\alpha_j$ has small height. 

By the above discussion, we see that the only potentially problematic case is when there is a positive proportion of ratios $\alpha_i/\alpha_j$ of small height. Our goal is then to prove that $\CA'$ contains a large structured subset to which we can apply Behrend's estimate \eqref{eq:Behrend}. In order to explain how to locate this subset, let us assume for simplicity that $\CA\subset\Q_{\ge1}$ (without necessarily knowing that the set of denominators is very sparse, meaning that $\CA$ is structured). If we write $\alpha_i=a/q$ and $\alpha_j=b/r$, it turns out that the corresponding events $E_i$ and $E_j$ are negatively correlated, unless the product $\gcd(a,b)\gcd(q,r)$ is large. This is a generalization of the set-up that occurred in the work of the first author and of James Maynard on the Duffin--Schaeffer conjecture \cite{DS}. In that paper, a potential counterexample to the Duffin--Schaeffer conjecture was a set of integers with many pairwise GCDs being large. Here, we work with a set of rational numbers and we must account for the product of the GCDs of the numerators and the denominators, but the methods of \cite{DS} can be adapted to this more general setting, without any serious difficulty. Hence, using the machinery of GCD graphs of \cite{DS}, we show that the only way we can have lots of pairs $(a/q,b/r)$ with large product $\gcd(a,b)\gcd(q,r)$ is if there exist fixed integers $A$, $B$, $Q$ and $R$ such that, for lots of pairs $(a/q,b/r)$, we have $a=Aa'$, $b=Bb'$, $q=Qq'$ and $r=Rr'$, as well as $\gcd(a,b)=\gcd(A,B)$ and $\gcd(q,r)=\gcd(Q,R)$. In particular, $\gcd(A,B)\gcd(Q,R)$ must be large.

However, when we use GCD graphs to produce the fixed integers $A,B,Q,R$, we incur the loss of the Euler factors $$
\prod_{p|\gcd(A,B)}\Big(1-\frac{1}{p}\Big)^{-2}\prod_{p|\gcd(Q,R)}\Big(1-\frac{1}{p}\Big)^{-2}.
$$
 In the context of the Duffin--Schaeffer conjecture, each integer $n$ is weighted by $\phi(n)/n$. These weights were used in \cite{DS} to  counterbalance the lost Euler factors from the method of GCD graphs. Here, however, we do not possess such convenient weights, so we have to gain the missing Euler factors from certain coprimality considerations that allow us to sieve for them. Indeed, note that $\gcd(a,q)=\gcd(b,r)=1$, and thus 
 \[
 \gcd(a,Q)=\gcd(b,R)=\gcd(q,A)=\gcd(r,B)=1. 
 \]
 Hence, in principle, we should be able to win the needed Euler factors by sieving $a',b',q'$ and $r'$ by $Q,R,A$ and $B$, respectively. However, this will only be true if the range of summation of $a',b',q',r'$ is ``long enough'' with respect to the integers we are sieving with. It turns out that we can arrange for the range of $a'$ and $b'$ to be long enough, and thus to gain the factors $\frac{\phi(Q)}{Q}\cdot \frac{\phi(R)}{R}$. However, the range of $q'$ and $r'$ might not be long enough. In fact, it is possible that $q'=r'=1$,  a case that will occur if all elements of $\CA$ have the same denominator. It thus seems we have reached an impasse. 

To get around this issue, we observe that the range of summation of $a'$ and $b'$ is sufficiently long so that if we somehow knew that $\gcd(a',A)=\gcd(b',B)=1$, then we would be able to gain the factors $\frac{\phi(A)}{A}$ and $\frac{\phi(B)}{B}$ from the $a'$ and $b'$ summations, respectively. This would be the case, for example, if we knew that all numerators $a$ were square-free (if $a$ is square-free and $a=Aa'$, we automatically have $\gcd(a',A)=1$). But this assumption is too restrictive to prove Theorem \ref{thm:solutions} in full generality. Alternatively, we would be able to gain the needed Euler factors if we knew that, for every prime $p|A$, the $p$-adic valuation of $A$ matches that of $a$ (and similarly for the primes $p|B$). However, the method of GCD graphs does not guarantee this exact divisibility. Indeed, we have the required exact divisibility in the ``first iterative step'' of the method of GCD graphs \cite[Proposition 8.1]{DS}, but not necessarily in the ``second iterative step''  \cite[Proposition 8.2]{DS}.

To get around this second obstacle, we use a variant of the method of GCD graphs that borrows a key idea from the recent paper of Hauke, Vazquez and Walker \cite{HSW}, in which they build on the work of Green--Walker \cite{GW} to give an alternative proof of the Duffin--Schaeffer conjecture. In this variant, we only perform the second iterative step to fully determine the divisors $Q$ and $R$ of the denominators, but we do not fully determine fixed divisors $A$ and $B$ for the numerators. Instead, after we perform fully the first iterative step for the numerators, and both the first and second iterative steps for the denominators, we arrive at a situation where we have fully determined $Q$ and $R$, and we have few possibilities for $A$ and $B$ with the benefit that, for each given possibility of $A$ and $B$, we have the required exact divisibility. This gives us the advantage that we can gain the factors $\frac{\phi(A)}{A}$ and $\frac{\phi(B)}{B}$ from the $a'$ and $b'$ summations. But it comes at the expense of having to execute a new summation over all possible values of $A$ and $B$. It turns out that this summation is too large (unlike in \cite{HSW}, where the savings $\frac{\phi(A)}{A}\cdot \frac{\phi(B)}{B}$ is enough to counterbalance the loss from the summation over $A$ and $B$). Hence, we have reached a third important obstacle. 

In order to get around this last obstacle, we use the extra savings provided to us from Behrend's estimate \eqref{eq:Behrend} (in fact, we need a generalization of it - see Theorem \ref{thm:behrend} below). In this last step, a small miracle occurs: the savings from Behrend's estimate is precisely what we need to balance the loss from the extra summation over all possible values of $A$ and $B$. This feature of our proof is new and did not appear in the work on the Duffin--Schaeffer conjecture \cite{HSW,DS,DS-quantitative}. It is also one of the main reasons why our theorem is ``soft'' and we do not prove quantitative estimates like \eqref{eq:Behrend}.

\begin{rem*} A different reason why we cannot prove \eqref{eq:Behrend} is that we work with the events $E_i$, whose measure is too small. There is some hope though that our method can show \eqref{eq:Erdos}, at least when $\CA\subset\{a/q \in\Q_{\ge1} :a\ \text{square-free}\}$. The reason is that under this assumption we can avoid using the Hauke--Vazquez--Walker variation described above. Note however that there is a second important feature of our proof, stemming from the construction of the sets $\CA'_j$ in Section \ref{sec:construction_setA'}. We use this construction to guarantee that the events $E_i$ and $E_j$ can only be positively correlated when $\log\alpha_{i}\asymp \log\alpha_{j}$. This is an important technical feature of the proof that greatly simplifies many details in Section \ref{sec:strategy}. But it also appears to be essential, especially in Sections \ref{sec:reduction-to-rationals} and \ref{sec:proof-of-moments-bound}.
\end{rem*}


\section{Strategy and key ingredients of the proof of Theorem \ref{thm:solutions}}\label{sec:strategy}

\subsection{Initial maneuvers}\label{sec:initial maneuvers}
Let $\CA\subset\R_{\ge2}$ be a discrete set such that 
\begin{equation}
	\label{eq:positive upper log density}
	\limsup_{x\to\infty} \frac{1}{\log x}\sum_{\alpha\in\CA\cap[1,x]}\frac{1}{\alpha} >0,
\end{equation}
and let $\eps\in(0,1]$. We seek to find distinct $\alpha,\beta\in\CA$ and an integer $n$ such that $|n\alpha-\beta|<\eps$. By dividing this inequality by $\eps$, and by replacing $\CA$ by the set $\{\alpha/\eps: \alpha\in\CA\}$, which also satisfies \eqref{eq:positive upper log density}, we may assume that $\eps=1$. 

In conclusion, we have reduced Theorem \ref{thm:solutions} to the following case: we are given a discrete set $\CA\subset\R_{\ge2}$ satisfying \eqref{eq:positive upper log density}, and we wish to find distinct $\alpha,\beta\in\CA$ and a positive integer $n$ such that $|n\alpha-\beta|<1$. Let us assume for contradiction that 
\begin{equation}
	\label{eq:No solutions}
	|n\alpha-\beta|\ge1 \quad\mbox{for all $\alpha,\beta\in\CA$ with $\alpha\neq\beta$, and for all $n\in\N$}.
\end{equation}
In particular, $\CA$ is $1$-spaced and satisfies the ``primitivity condition''
\begin{equation}
	\label{eq:primitivity conditon}
	\alpha/\beta\notin\N\quad\mbox{for all distinct $\alpha,\beta\in\CA$}.
\end{equation}

We will then show that 
\begin{equation}
	\label{eq:zero natural density}
	\lim_{T\to\infty} \frac{1}{T}\sum_{\alpha\in\CA\cap[1,T]}1=0.
\end{equation}
Clearly, if we can indeed prove this, we will contradict \eqref{eq:positive upper log density}. We will have thus completed the proof of Theorem \ref{thm:solutions}.

Now for each $\alpha\in\CA$, let us define the sets
\[
\CM_\alpha \coloneqq \bigcup_{n\in\N} \Big(n\alpha-\tfrac{1}{2},\, n\alpha+\tfrac{1}{2}\Big) .
\]
As Haight also observed, for any $\CA'\subset\CA$, condition \eqref{eq:No solutions} implies that
\begin{equation}\label{Eq:IntersectionNeighborhood1}
\bigg(\bigcup_{\alpha \in \CA\setminus \CA'} \Big(\alpha-\tfrac{1}{2},\, \alpha+\tfrac{1}{2}\Big) \bigg)\cap \bigg(\bigcup_{\alpha\in\CA'}\CM_\alpha \bigg) =\emptyset. 
\end{equation}
Moreover, we have that
\begin{equation}\label{Eq:IntersectionNeighborhood2}
\meas\bigg( [0,T]\cap \bigcup_{\alpha \in \CA\setminus \CA'} \Big(\alpha-\tfrac{1}{2},\, \alpha+\tfrac{1}{2}\Big) \bigg) =\#\Big(\big(\CA\setminus \CA')\cap[0,T]\Big) +O(1).
\end{equation}
Assume now that, for every given $\ep>0$, we can show the existence of a finite set $\CA'=\CA'(\ep)\subset \CA$, and $T_0= T_0(\CA', \ep)$ such that
\begin{equation}
		\label{eq:large measure}
			\P_T\bigg(\bigcup_{\alpha\in\CA'}\CM_\alpha \bigg) \ge 1-\ep\quad\text{for all}\ T\ge T_0.
	\end{equation}
We then deduce that $\#(\mathcal{A}\cap[1,T])\leq \ep T +O_{\ep}(1)$ for all $T\geq T_0$  by \eqref{Eq:IntersectionNeighborhood1} and \eqref{Eq:IntersectionNeighborhood2}, which implies \eqref{eq:zero natural density} since $\ep$ can be taken arbitrarily small.


\subsection{The second moment method}\label{sec:second moment}
In order to prove \eqref{eq:large measure}, we use the second moment method as in Haight's paper \cite{Ha88}. To this end we record the following classical lemma which follows from an easy application of the Cauchy-Schwarz inequality (see for example \cite[Lemma 2.3]{Har}).

\begin{lemma}\label{Lem:SecondMomentMethod}
Let $T$ be a positive real number, let $k\ge1$ be an integer, and let $E_1, \dots, E_k$ be measurable subsets of $[0, T]$. Then, we have 
\begin{equation}\label{Eq:PaleyZygmund}
	\P_T\bigg( \bigcup_{j=1}^k E_j \bigg) \geq \frac{\big(\sum_{j=1}^k  \P_T\big(E_j\big)\big)^2}{\sum_{i=1}^k\sum_{j=1}^k \P_T\big(E_i\cap E_j\big)}.
\end{equation}
\end{lemma}

In view of the above lemma, if we are given measurable sets $E_1, \dots, E_k\subset[0, T]$ and we want to show that
\begin{equation}\label{Eq:FullMeasureUnion}
\P_T\bigg( \bigcup_{j=1}^k E_j \bigg)\geq (1+o(1)), \text{ as } T\to \infty,
\end{equation}
then it suffices to prove that $E_1, \dots, E_k$ verify the following two conditions:
 \begin{itemize} 
 \item[1.] $\sum_{j=1}^k \P_T\big(E_j) \to \infty$ as $T\to \infty$;
 \item[2.] $\P_T\big(E_i\cap E_j\big)\lesssim \P_T\big(E_i)\P_T\big(E_j)$ whenever $1\leq i<j\leq k$.
 \end{itemize}
Motivated by the second condition, and in order to keep the exposition concise, we make the following definition.
 
\begin{dfn}[Negatively correlated sets]\label{dfn:uncorrelated}
Let $E_1, E_2$ be measurable subsets of $\R_{\ge0}$. We say that $E_1$ and $E_2$ are \emph{negatively correlated}\footnote{If we view $[0, T]$ as a probability space equipped with the probability measure $\P_T$, then the condition $\P_T\big(E_1\cap E_2\big)<\P_T\big(E_1)\P_T\big(E_2)$ is equivalent to saying the random variables $\one_{E_1}$ and $\one_{E_2}$ have negative correlation. With this in mind, we ought to have defined $\P_T\big(E_1\cap E_2\big)\lesssim \P_T\big(E_1)\P_T\big(E_2)$ to mean that $E_1$ and $E_2$ are ``asymptotically non-positively correlated'', but this would be rather cumbersome. For this reason, we resort to a less precise but cleaner terminology.}  if $\P_T\big(E_1\cap E_2\big)\lesssim \P_T\big(E_1)\P_T\big(E_2)$ as $T\to\infty$.
\end{dfn}
 
 In order to obtain \eqref{Eq:FullMeasureUnion}, one can of course weaken the second condition above to only require that the sets $E_1, \dots, E_k$ are \emph{pairwise negatively correlated on average}, namely that 
 \begin{equation}\label{Eq:UncorrelatedAverage}
 \mathop{\sum\sum}_{1\leq i\neq j \leq k} \P_T\big(E_i\cap E_j\big)\lesssim \Big(\sum_{1\leq j\leq k} \P_T\big(E_j\big)\Big)^2.
 \end{equation}

Note that 
$ 
\P_T\big(\CM_\alpha\big) \sim 1/\alpha,
$ and we might assume that $\sum_{\alpha\in \CA}1/\alpha=\infty$ since otherwise \eqref{eq:zero natural density} holds trivially. Therefore, if we can show that the sets $\{\CM_\alpha\}_{\alpha \in \CA'}$ are pairwise negatively correlated on average, for some judicious choice of $\CA'$ such that $\sum_{\alpha\in \CA'}1/\alpha$ is large, then we would deduce \eqref{eq:large measure}. 

Haight \cite[Lemma 1]{Ha88} proved that
\begin{equation}\label{Eq:HaightIntersection}
\P_T\big(\CM_\alpha\cap\CM_\beta\big)\sim\P_T\big(\CM_\alpha)\P_T\big(\CM_\beta\big)
\end{equation}
  if $\alpha/\beta$ is irrational. This implies that $\CM_\alpha$ and $\CM_\beta$ are asymptotically uncorrelated, which is even stronger than saying they are negatively correlated. Using \eqref{Eq:HaightIntersection}, Haight deduced that \eqref{eq:zero natural density} holds if $\alpha/\beta\notin \mathbb{Q}$ for all distinct elements $\alpha, \beta\in \CA$.  


\subsection{The problem with the sets $\CM_\alpha$}\label{sec:problem_sets}
Given the discussion in the previous section, a natural approach to prove \eqref{eq:large measure} is to apply Lemma \ref{Lem:SecondMomentMethod} with the events $E_i$ being of the form $\CM_\alpha$ with $\alpha$ in our specially chosen set $\CA'$. However, this does not work in full generality. For instance, if $\CA$ is a set of integers, then the events $\CM_\alpha$ and $\CM_\beta$ are positively correlated if $\gcd(\alpha, \beta)>1$. Indeed if $\alpha, \beta\in\N$, then
\[
\P_T\big(\CM_\alpha\cap\CM_\beta\big)
	=\frac{1}{T}\#\big\{(m,n)\in\N^2: m\alpha=n\beta,\ m\alpha\leq T\big\}
	\sim \gcd(\alpha, \beta)\P_T\big(\CM_\alpha)\P_T\big(\CM_\beta\big).
\]
As a matter of fact, one can construct easy examples of sets of integers $\CA$ such that  $\{\CM_\alpha\}_{\alpha\in\CA}$ are positively correlated on average, so there is no hope for the second moment method to work in this case.\footnote{An easy example is to take $\CA\subset\bigcup_{j\ge1} \{a\in\N\cap[\sqrt{x_j},x_j] :  \omega(a)=\fl{\log\log x_j}\}$ for a sparse sequence $x_j\to\infty$. We may find such an example such that the factor $\gcd(a,b)$ has expected value equal to a power of $\log x_j$ when $a$ and $b$ range over $\CA\cap[\sqrt{x_j},x_j]$ and are weighted by $\P_T(\CM_a)\sim1/a$ and $\P_T(\CM_b)\sim1/b$, respectively.} In fact, if it were to work, we would have an analogous result to Haight in the integer case, but we know from Besicovitch's result \cite{Bes34} that there exist integer primitive sets with positive upper natural density. 

By \eqref{Eq:HaightIntersection}, we only need to focus on the case where $\alpha/\beta\in \mathbb{Q}$. In this case let us write $\alpha/\beta=s/t$ where $s, t$ are positive coprime integers. Then we observe that 
\begin{equation}\label{Eq:LowBoundIntersectionDiagonal}
	\begin{split}
\P_T\big(\CM_\alpha\cap\CM_\beta\big)
	&\geq \frac{1}{T}\#\big\{(m,n)\in\N^2:  m\alpha=n\beta,\ m\alpha\leq T\big\} \\
	&\sim \frac{1}{\alpha t}
	\sim \frac{\beta}{t} \cdot \P_T\big(\CM_\alpha)\P_T\big(\CM_\beta\big).
	\end{split}
\end{equation} 
Hence, if the quantity $t/\beta= s/\alpha$ is small, the sets $\CM_\alpha$, $\CM_\beta$ will be positively correlated. It turns out that this quantity plays a central role in our proof of Theorem \ref{thm:solutions}. We shall call it the \emph{bracket} of $\alpha$ and $\beta$. To define this notion properly for all positive real numbers, we first define the \emph{height} of a real number.

\begin{dfn}[Height of a real number]\label{dfn:height}
		Let $\alpha>0$ be a real number. If $\alpha\notin\Q$, we define 
	\[
	H(\alpha)\coloneqq \infty.
	\]
	Otherwise, if $\alpha\in\Q$, then we write $\alpha=a/q$ with $a,q\in\N$ and $\gcd(a,q)=1$, and we define 
	\[
	H(\alpha)\coloneqq \max\{a,q\}. 
	\]
	We call $H(\alpha)$ the \emph{height} of the number $\alpha$. 
\end{dfn}

\begin{dfn}[Bracket of two real numbers]\label{dfn:bracket}
	Given two real numbers $\alpha,\beta>0$, we define
		\[
		[\alpha,\beta]\coloneqq \frac{H(\alpha/\beta)}{\max\{\alpha,\beta\}}.
		\]
	\end{dfn}
We have the following basic properties of the bracket of two real numbers.

	\begin{lemma}\label{lem:R for rational ratio}
		Let $\alpha,\beta>0$. 
		\begin{enumerate}
			\item We have $[\alpha,\beta]=[\beta,\alpha]$.
			\item If $\alpha/\beta=s/t \in \Q$, where $s/t$ is a reduced fraction, then 
				\[
				[\alpha,\beta]= \frac{s}{\alpha}=\frac{t}{\beta} .
				\]
			\item If $\alpha=a/q$ and $\beta=b/r$, where $a/q$ and $b/r$ are both reduced fractions, then
			\[
			[\alpha,\beta] = \frac{qr}{\gcd(q,r)\gcd(a,b)}.
			\]
		\end{enumerate}
	\end{lemma}
	
	\begin{proof}(a) This follows readily by the definition of $[\alpha,\beta]$.
		
		\medskip 
		
		(b) Using part (a), we may assume that $\alpha\ge\beta$. Then, we have that $\alpha/\beta=s/t\ge1$, and thus $H(\alpha/\beta)=s$. We thus find that $[\alpha,\beta]=s/\alpha$. Lastly, we have that $s/\alpha=t/\beta$ by assumption.
		
		\medskip
		
		(c) If we let $a_1=a/\gcd(a,b)$, $b_1=b/\gcd(a,b)$, $q_1=q/\gcd(q,r)$ and $r_1=r/\gcd(q,r)$, then we have $\alpha/\beta=\frac{a_1r_1}{b_1q_1}$, and $\gcd(a_1r_1,b_1q_1)=1$. Hence, $[\alpha,\beta]=\frac{a_1r_1}{\alpha}$ by part (b). Inserting the definition of $a_1$ and $r_1$ in this expression, and using that $\alpha=a/q$, completes the proof of this last part of the lemma.
	\end{proof}
	
It follows from \eqref{Eq:LowBoundIntersectionDiagonal} that $\CM_{\alpha}$ and $\CM_{\beta}$ are positively correlated if $[\alpha,\beta]<1$. To overcome this problem so that the second moment method works, we replace $\CM_{\alpha}$ by \emph{smaller} sets, of the form $\bigcup_{n\in \mathcal{L}_{\alpha}} (n\alpha-1/2,n\alpha+1/2)$ for some $\mathcal{L}_{\alpha}\subset \mathbb{N}$. Using an idea of Erd\H{o}s \cite{Erd35} from his proof of \eqref{eq:Erdos}, it turns out that a judicious choice for $\mathcal{L}_{\alpha}$ is to take $\{n\in \mathbb{N} : P^{-}(n)>\alpha\}$, meaning the set of $\alpha$-rough numbers. Hence, for each $\alpha\ge1$, we define the sets
\[
\CN_\alpha\coloneqq \bigcup_{\substack{n\in\N \\ P^-(n)>\alpha }} 
\Big(n\alpha-\tfrac{1}{2},\, n\alpha+\tfrac{1}{2}\Big),
\]
for which we have
\begin{equation}
\label{eq:measure of N_alpha}
\P_T(\CN_\alpha) \sim \frac{1}{\alpha} \prod_{p\le\alpha}\bigg(1-\frac{1}{p}\bigg)
\end{equation}
as $T\to\infty$. Since $\CN_\alpha\subset\CM_\alpha$ for all $\alpha\ge1$, relation \eqref{eq:large measure} will follow as long as we can show that
\begin{equation}
	\label{eq:large measure for N_alpha}
	\P_T\bigg(\bigcup_{\alpha\in\CA'}\CN_\alpha \bigg) \ge 1-\eps
\end{equation}
for an appropriate choice of $\CA'$. We will use Lemma \ref{Lem:SecondMomentMethod} to prove \eqref{eq:large measure for N_alpha}. To this end, we must understand the correlations of the sets $\CN_\alpha$.


\subsection{Correlation estimates}
We first show that if $[\alpha, \beta]\leq 1$, then the sets $\CN_\alpha$, $\CN_\beta$ are not only negatively correlated, they are in fact disjoint!
\begin{lemma}[No overlap when $\lbr\alpha,\beta\rbr\le1$]
	\label{lem:R<1}
	Let $\alpha>\beta\geq 1$ be two real numbers such that $\alpha/\beta\notin\N$ and $[\alpha,\beta]\le 1$. Then 
	\[
	\CN_\alpha\cap\CN_\beta= \emptyset. 
	\]
\end{lemma}
To prove this result we first need to take care of  the diagonal solutions $m\alpha=n\beta$, when $m$ is $\alpha$-rough and $n$ is $\beta$-rough.

\begin{lemma}[No diagonal solutions]
	\label{lem:erdos for rationals}
	Let $\alpha>\beta\geq 1$ be two real numbers such that $\alpha/\beta\notin\N$ and $[\alpha,\beta]\le \alpha/\beta$. Then there are no solutions to the equation $m\alpha=n\beta$ with $m$ and $n$ natural numbers such that $P^-(m)>\alpha$.
\end{lemma}
\begin{proof} 
We have $\alpha/\beta\in\Q$; otherwise, $[\alpha,\beta]=\infty$. Write $\alpha/\beta=s/t$ in reduced form. Suppose for contradiction that $m\alpha=n\beta$ for some $m,n\in\N$ with $P^-(m)>\alpha$. Equivalently, $ms=nt$. Since $\gcd(s,t)=1$, we must have that $t|m$,  in particular $P^-(t)>\alpha$. On the other hand, note that $t/\beta=[\alpha,\beta]\le \alpha/\beta$, and thus $t\le\alpha$. But $t\le\alpha$ and $P^-(t)>\alpha$ forces $t=1$, which gives  $\alpha/\beta=s\in\N$, a contradiction.
\end{proof}

\begin{proof}[Proof of Lemma \ref{lem:R<1}] As in Lemma \ref{lem:erdos for rationals}, we may assume that $\alpha/\beta\in\Q$. Let us write $\alpha/\beta=s/t$ in reduced form. If $\CN_\alpha\cap\CN_\beta$ were non-empty, then there would exist $m,n\in\N$ such that $P^-(m)>\alpha$, $P^-(n)>\beta$ and $ |m\alpha-n\beta|<1$. Since $[\alpha,\beta]\le1<\alpha/\beta$, Lemma \ref{lem:erdos for rationals} forces $m\alpha\neq n\beta$. Thus $ms-nt$ is a non-zero integer, and so 
	\[
	1>|m\alpha-n\beta|=\beta|ms/t-n|=\frac{\beta}{t} \cdot |ms-nt| = \frac{|ms-nt|}{[\alpha,\beta]} \ge |ms-nt| \ge 1
	\] 
	using $[\alpha,\beta]\le1$. This gives a contradiction, and completes the proof.
\end{proof}
	
In view of Lemma \ref{lem:R<1}, we only need to investigate the correlations of $\CN_{\alpha}$ and $\CN_{\beta}$ when $[\alpha, \beta]>1$. Using inclusion-exclusion and Haight's result \eqref{Eq:HaightIntersection} we first show that $\CN_{\alpha}$ and $\CN_{\beta}$ are negatively correlated when $[\alpha, \beta]=\infty$, that is to say, when $\alpha/\beta$ is irrational.

\begin{lemma} 	\label{lem:overlap estimate for irrational ratios}
	Let $\alpha,\beta\in\R_{\ge1}$ be such that $\alpha/\beta$ is irrational. Then, as $T\to\infty$, we have
\[
\P_T\big(\CN_\alpha\cap\CN_\beta\big) \sim   \P_T\big(\CN_\alpha\big)\P_T\big(\CN_\beta\big).
\]
\end{lemma}

\begin{proof}
First, we note that
\begin{align}\label{Eq:ExpressionOverlapTwoSets}
\P_T\big(\CN_\alpha\cap\CN_\beta\big) 
& = \frac{1}{T}\int_0^T \one_{t\in \CN_\alpha}\cdot\one_{t\in\CN_\beta}\dee t
= \frac{1}{T}\sum_{\substack{n\leq T/\beta \\ P^{-}(n)>\beta}} \ \sum_{\substack{m\leq T/\alpha \\ P^{-}(m)>\alpha}}\int_0^T \one_{|t-m\alpha|<1} \cdot \one_{|t-n\beta|<1}\dee t \nonumber\\
&=\frac{1}{T}\sum_{\substack{n\leq T/\beta \\ P^{-}(n)>\beta}} \ \sum_{\substack{m\leq T/\alpha \\ P^{-}(m)>\alpha}} \left(1-|n\beta-m\alpha|\right)^{+},
\end{align}
where $x^{+} \coloneqq  \max(0, x)$. Given $z\in\R_{\ge1}$, let $\CP_{z}= \prod_{p\leq z} p$. Then the condition $P^-(m)>\alpha$ is equivalent to $\gcd(m,\CP_\alpha)=1$, and so by M\"obius inversion we have
\begin{align}
\P_T\big(\CN_\alpha\cap\CN_\beta\big) 
 &=  \sum_{d_1\mid \mathcal{P}_{\alpha}} (-1)^{\omega(d_1)}\sum_{d_2\mid \mathcal{P}_{\beta}} (-1)^{\omega(d_2)} \ \frac{1}{T}\sum_{n\leq T/(d_2\beta)} \ \sum_{m\leq T/(d_1\alpha)} \left(1-|nd_2\beta-md_1\alpha|\right)^{+} \nn
 &=  \sum_{d_1\mid \mathcal{P}_{\alpha}} (-1)^{\omega(d_1)} \sum_{d_2\mid \mathcal{P}_{\beta}} (-1)^{\omega(d_2)} \P_T\big(\CM_{d_1\alpha}\cap\CM_{d_2\beta}\big).
 \label{Eq:InclusionExclusionIrrational}
\end{align}
Furthermore, note that $\alpha/\beta\notin\mathbb{Q}$ if and only if $d_1\alpha/(d_2\beta)\notin \mathbb{Q}$ for any pair of positive integers $d_1, d_2$. Therefore, by \eqref{Eq:HaightIntersection}, \eqref{eq:measure of N_alpha}  and \eqref{Eq:InclusionExclusionIrrational} we get 
\begin{align*}
\P_T\big(\CN_\alpha\cap\CN_\beta\big) &\sim \sum_{d_1\mid \mathcal{P}_{\alpha}} (-1)^{\omega(d_1)} \sum_{d_2\mid \mathcal{P}_{\beta}} (-1)^{\omega(d_2)} 
\cdot \frac{1}{d_1\alpha d_2\beta}\\
&= \frac{1}{\alpha\beta}\prod_{p\leq \alpha} \left(1-\frac1p\right)\prod_{p\leq \beta} \left(1-\frac1p\right)\sim  \P_T\big(\CN_\alpha\big)\P_T\big(\CN_\beta\big)
\end{align*}
when $T\to\infty$, as desired.
\end{proof}

It now remains to consider the case $1<[\alpha, \beta]<\infty$, which we will split further in two cases. If $[\alpha, \beta]$ is large, we prove the following lemma, which we shall use to show that for such pairs $(\alpha, \beta)$, the sets $\CN_\alpha$ and $\CN_\beta$ are negatively correlated on average. Indeed, as Lemma \ref{lem:contribution of diagonal error terms} below shows, the error term $\log(2\beta)/(\alpha/\beta)+1/\beta$ is negligible on average.

\begin{lemma}
	\label{lem:overlap estimate for ratios with large R}
		Let $\alpha>\beta\ge2$ be two real numbers with $\alpha/\beta\in\Q \setminus \N$ and $[\alpha,\beta]\ge\min\{\alpha^2,10^\beta\}$. Then as $T\to\infty$,
	\[
	\frac{\P_T\big(\CN_\alpha\cap \CN_\beta\big)}{\P_T\big(\CN_\alpha\big) \P_T\big(\CN_\beta\big)}
	 \ \lesssim \ 1 +O\bigg(\frac{\log(2\beta)}{\alpha/\beta}+\frac{1}{\beta}\bigg) .
	\]
\end{lemma} 

We shall construct our set $\CA'$ so that we can apply Lemma \ref{lem:overlap estimate for ratios with large R} to the vast majority of pairs $(\alpha, \beta)\in \CA'\times \CA'$. The remaining pairs are very few, but their treatment constitutes the hardest part of the proof. In particular, it is in this part that we need to use the machinery of GCD graphs from \cite{DS}. On the other hand, because the pairs $(\alpha,\beta)$ with 
$1<[\alpha,\beta]\leq \min\{\alpha^2,10^\beta\}$ will be arranged to be sparse, it suffices to obtain less precise upper bounds for $\P_T\big(\CN_\alpha\cap \CN_\beta\big)$. These are given in terms of the prime factors dividing the product of the numerator and denominator of the fraction $\alpha/\beta$. To this end, we introduce the following useful concepts.

\begin{dfn}[The prime numbers of a rational number]\label{dfn:primes of a rational}
		Given a prime $p$ and a rational number $\rho>0$, we write $p\in\rho$ if we may write $\rho=a/q$ with $\gcd(a,q)=1$ and $p|aq$. 
	\end{dfn}
	
	\begin{dfn}[Two arithmetic functions on the rationals]\label{dfn:omega and L}
		Given $\rho\in\Q_{>0}$ and $z\ge1$, we define
		\[
		\omega(\rho;z)\coloneqq \#\{p\in  \rho \text{ such that } p\leq z\}
		\quad\text{and}\quad
		L(\rho;z)\coloneqq \sum_{\substack{p\in \rho  \\ p>z}} \frac{1}{p} .
		\]
	\end{dfn}

Then we prove the following result:

\begin{lemma}
	\label{lem:overlap estimate for rational ratios of typical rationals}
	Let $\alpha>\beta\ge2$ be real numbers with $\alpha/\beta\in\Q\setminus\N$ and $R\coloneqq[\alpha,\beta]>1$. 
	\begin{enumerate}
		\item If $L(\alpha/\beta;R) \le 1$, then as $T\to\infty$,
		\[
		\frac{\P_T\big(\CN_\alpha\cap \CN_\beta\big)}{\P_T\big(\CN_\alpha\big) \P_T\big(\CN_\beta\big)}
		\ \ll \ 1 +  \frac{\log \beta}{\alpha/\beta}. 
		\]

		\item If $L(\alpha/\beta;R)>1$ and we let
		\[
		z=\max\big\{2^i : i\in\Z_{\ge0},\ L(\alpha/\beta;2^i)>1 \big\},
		\] 
		then $R< 2z$ and, as $T\to\infty$, we have
		\[
		\frac{\P_T\big(\CN_\alpha\cap \CN_\beta\big)}{\P_T\big(\CN_\alpha\big) \P_T\big(\CN_\beta\big)}
		\ll \frac{\log(2z)}{\log(2R)} + \frac{\log \beta}{\alpha/\beta}. 
		\]
			\end{enumerate}
\end{lemma}
Both Lemmas \ref{lem:overlap estimate for ratios with large R} and \ref{lem:overlap estimate for rational ratios of typical rationals} will be consequences of Lemma \ref{lem:overlap estimate for rational ratios} which we shall establish in Section \ref{sec:rational} using an elementary sieving argument.
 
\subsection{Constructing the set $\CA'$}\label{sec:construction_setA'} Next, we explain how to construct a finite set $\CA'$ such that \eqref{eq:large measure for N_alpha} holds. 
To simplify our notation, for each $\alpha>0$ we let
\begin{equation}\label{eq:Kappa_Weight}
\kappa(\alpha)\coloneqq \frac{1}{\alpha} \prod_{p\le \alpha}\bigg(1-\frac{1}{p}\bigg) \asymp \frac{1}{\alpha\log(\alpha+2)},
\end{equation}
and
\[
\lambda(\alpha)\coloneqq \frac{1}{\alpha}.
\]
In addition, for any finite sets $\CB\subset\R_{>0}$ and $\CE\subset\R_{>0}^2$, and for $\mu:\R_{>0}\to \R_{\ge0}$, we define
\begin{equation}\label{eq:Sets_Weights}
	\mu(\CB) \coloneqq \sum_{\beta\in\CB} \mu(\beta), \qquad \text{and}\qquad
	\mu(\CE) \coloneqq \sum_{(\alpha,\beta)\in\CE} \mu(\alpha) \mu(\beta).
\end{equation}
By \eqref{eq:positive upper log density}, there exists some constant $c\in(0,1/10)$ such that 
\[	
\lambda\big(\CA\cap[1,x]\big)=\sum_{\alpha\in\CA\cap[1,x]}\frac{1}{\alpha}\ge 4c\log x
\]
infinitely often as $x\to\infty$. Recall that $\CA$ is $1$-spaced by \eqref{eq:No solutions}. This implies that $\sum_{\alpha\in\CA\cap[1,x^c]}1/\alpha\le 2c\log x$
for all $x$ sufficiently large. Hence, we find that
\begin{equation}
	\label{eq:lb on log sums}
\lambda\big(\CA\cap[x^c,x]\big)\ge 2c\log x
\end{equation}
infinitely often as $x\to\infty$.

Next, we construct a convenient sequence $x_1<x_2<\dots$ such that $x_j>x_{j-1}^{1/c}$ for all $j\ge2$, and certain sets
\[
\CA_j'\subseteq \CA_j\coloneqq \CA\cap[x_j^c,x_j]\quad j=1,2,\dots
\]
such that
\begin{equation}
	\label{eq:lb on log sums 2}
		\lambda(\CA_j')\ge c\log x_j. 
\end{equation}
The reason we introduce the sets $\CA_j'$ is to guarantee Lemma \ref{lem:independence A_i and A_j}(b) below. In turn, this will allow us to apply Lemma \ref{lem:overlap estimate for ratios with large R} for almost all pairs $(\alpha,\beta)\in\CA'\times\CA'$. 

In order to achieve the construction of $\CA'$, let us consider the following equivalence relation on $\CA$: we write $\alpha\equiv\beta$ if, and only, if $\alpha/\beta\in\Q$. We may then partition $\CA$ into equivalence classes. For each $\alpha\in \CA$, let us denote by $\gamma_{\alpha}$ the smallest element of $\CA$ that lies in the same equivalence class as $\alpha$. Moreover, let 
\[
\Gamma \coloneqq \{\gamma_\alpha : \alpha\in\CA \} ,
\]
which is the set of minimal representatives of the equivalent classes of $\CA$.

Let us now construct the $x_j$'s. Firstly, we take $x_1$ to be the smallest integer $\ge e^{1/c}$ for which \eqref{eq:lb on log sums} holds, and we let $\CA_1'=\CA_1=\CA\cap[x_1^c,x_1]$. In particular, \eqref{eq:lb on log sums 2} holds.

Next, assume we have constructed $x_1,\dots,x_j$ and $\CA_1',\dots,\CA_j'$ such that $x_i^c>x_{i-1}$ and \eqref{eq:lb on log sums 2} holds for $i=2,\dots,j$. Consider the quantities 
\[
Q_j=\max\Big\{10^\alpha \alpha H(\alpha/\gamma_{\alpha}) : \alpha\in \CA_1\cup\cdots \cup \CA_j \Big\} \quad \text{and}\quad 
S_j =  Q_j^2\sum_{\gamma\in\Gamma\cap[1,x_j]}\frac{1}{\gamma}.
\]
In addition, let $\CB_j$ be the set of $\alpha\in\CA$ for which there exists $\gamma\in\Gamma\cap[1,x_j]$ 
 and a reduced fraction $a/q$  such that $q\le Q_j$ and $\alpha=\gamma a/q$. For each given $\gamma$ and $q$, the set $\{a\in\N: \gamma a/q\in \CA\}$ is primitive. We may thus apply \eqref{eq:AKS bound} to show that
\begin{align*}
\sum_{\substack{q\le a\le qx \\ \gamma a/q\in \CA}} \frac{1}{a} \ll \frac{\log x}{\sqrt{\log\log x}}.
\end{align*}
Consequently,
\[
\begin{split}
\lambda\big(\CB_j\cap[1,x]\big)
	 \le \sum_{\gamma\in\Gamma\cap[1,x_j]}\frac{1}{\gamma}
 \sum_{q\le Q_j} q \sum_{\substack{q\le a\le qx \\ \gamma a/q\in \CA}} \frac{1}{a}
	& \ll \sum_{\gamma\in\Gamma\cap[1,x_j]}\frac{1}{\gamma} 
	 		\sum_{q\le Q_j} q  \cdot \frac{\log x}{\sqrt{\log\log x}}  \\
 	&\le  S_j \cdot \frac{\log x}{\sqrt{\log\log x}}.
\end{split}
\]
Therefore, there exists a constant $y_j$ such
\[
\lambda\big(\CB_j\cap[1,x]\big) \le c\log x\quad\text{for all}\ x\ge y_j.
\]
We then take $x_{j+1}$ to be the smallest integer $>\max\{y_j,x_j^{1/c}\}$ such that \eqref{eq:lb on log sums} holds. Hence, if we set 
\[
\CA_{j+1}'\coloneqq \CA_{j+1}\setminus \CB_j, 
\]
then \eqref{eq:lb on log sums 2} holds. Fix $J$ large and denote 
\[
\CA'  \coloneqq \bigcup_{j=1}^J \CA_j'. 
\]

This completes the construction of the set $\CA'$. We show that it satisfies the following properties. 

\begin{lemma}\label{lem:independence A_i and A_j}
Let $\CA'$ be as above. 
\begin{enumerate}
\item We have
\[
	cJ \ll \sum_{\alpha\in \CA'} \P_T(\CN_\alpha) \ll J \log(1/c).
\]
\item If $\alpha, \beta \in\CA'$ are such that $\alpha>\beta$ and  $[\alpha,\beta] \leq 10^{\beta}$, then there exists $j\in\{1,2,\dots,J\}$ such that 
$\alpha, \beta  \in\CA'_j$.
\end{enumerate}
\end{lemma}

\begin{proof} (a) By \eqref{eq:measure of N_alpha} and Mertens' theorem we have 
\begin{equation}\label{eq:Asymp_Measure_SetsA'}
\sum_{\alpha\in \CA'} \P_T(\CN_\alpha)\asymp \sum_{j=1}^J \sum_{\alpha\in \CA'_j} \frac{1}{\alpha\log \alpha}. 
\end{equation}
Furthermore, by \eqref{eq:lb on log sums 2} and the fact that $\CA$ is $1$-spaced, we get
\begin{equation}\label{eq:Asymp_Measure_SetsA'j}
c\ll \sum_{\alpha\in \CA'_j} \frac{1}{\alpha\log \alpha}\le \sum_{\alpha\in \CA_j} \frac{1}{\alpha\log \alpha} \ll \log(1/c).
\end{equation}
Inserting this estimate in \eqref{eq:Asymp_Measure_SetsA'} completes the proof of part (a). 

\medskip

(b) 
Since $\alpha>\beta$ there exists $1\leq j\leq i\leq J$ such that $\alpha \in \CA'_i$ and $\beta\in \CA'_j$.  Assume for contradiction that $i>j$. Since $[\alpha,\beta] \leq 10^{\beta}$, we must have $\alpha/\beta\in\Q$ and hence there exists $\gamma\in\Gamma$ such that $\alpha=\gamma a/q$ and $\beta= \gamma b/r$ with $a/q$ and $b/r$ reduced fractions $\ge1$. In particular, $\gamma\le \beta\le x_j\le x_{i-1}$. 
Since $\alpha=\gamma a/q\in\CA\setminus \CB_{i-1}$, we must have that $q>Q_{i-1}\ge 10^\beta \beta H(\beta/\gamma)$. 
	
Now, note that $\frac{\alpha}{\beta}=\frac{a_1r_1}{b_1q_1}$, where $a_1=a/\gcd(a,b)$, $r_1=r/\gcd(q,r)$, $b_1=b/\gcd(a,b)$ and $q_1=q/\gcd(q,r)$. We have $\gcd(a_1r_1,b_1q_1)=1$. Hence, 
\[
[\alpha,\beta]=\frac{b_1q_1}{\beta}	
	\ge \frac{q}{\beta r} 
	\ge \frac{q}{\beta H(\beta/\gamma)} 
	> 10^\beta,
\]
which gives a contradiction. This completes the proof of part (b). 
\end{proof}

\subsection{The last key result towards Theorem \ref{thm:solutions}} 
Let us now explain how to put together all of the above ingredients. This will reduce Theorem \ref{thm:solutions} to the key Proposition \ref{prop:key proposition} below.

Recall that Theorem \ref{thm:solutions} follows from \eqref{eq:large measure}, which in turn can be deduced from \eqref{eq:large measure for N_alpha}. To show this last relation, we apply Lemma \ref{Lem:SecondMomentMethod} with the events $E_i$ being the sets $\CN_\alpha$ with $\alpha\in\CA'$. We have
\[
\sum_{\alpha\in\CA'} \P_T(\CN_\alpha)\sim \kappa(\CA')\asymp_c J
\]
by Lemma \ref{lem:independence A_i and A_j}(a), where $\kappa(\CA')$ is defined in \eqref{eq:Kappa_Weight} and \eqref{eq:Sets_Weights}. Thus, it suffices to show, as $T\to\infty$,
\begin{align}\label{eq:CS}
\mathop{\sum\sum}_{\substack{\alpha,\beta\in\CA' \\ \alpha\neq\beta}} \P_T\big(\CN_\alpha\cap\CN_\beta\big) \ \lesssim \ \kappa(\CA')^2 + O_c\big(J\big),
\end{align}
in which case we may conclude
\begin{align*}
\P_T\bigg(\bigcup_{\alpha\in\CA'} \CN_\alpha\bigg) \gtrsim \frac{\kappa(\CA')^2}{\kappa(\CA')^2 + O_c\big(\kappa(\CA')\big)}= 1-O_c\bigg(\frac{1}{J}\bigg),
\end{align*}
by Lemma \ref{Lem:SecondMomentMethod}. 
This will complete the proof of \eqref{eq:large measure for N_alpha} (and thus of Theorem \ref{thm:solutions}) if we take $J$ to be suitably large.

Combining the estimates in Lemmas \ref{lem:R<1}, \ref{lem:overlap estimate for irrational ratios}, \ref{lem:overlap estimate for ratios with large R} and \ref{lem:overlap estimate for rational ratios of typical rationals}, we obtain that
\begin{align}\label{eq:bilinear}
\mathop{\sum\sum}_{\alpha,\beta\in\CA',\ \alpha>\beta} \P_T\big(\CN_\alpha\cap\CN_\beta\big)
\lesssim \mathop{\sum\sum}_{\alpha,\beta\in\CA',\ \alpha>\beta} &\P_T\big(\CN_\alpha)\P_T\big(\CN_\beta\big) 
 + O\Big(E_1+E_2+E_3\Big)
\end{align}
as $T\to\infty$, where
\begin{align*}
E_1&\coloneqq  \mathop{\sum\sum}_{\alpha,\beta\in\CA',\ \alpha>\beta} \kappa(\alpha)\kappa(\beta) \cdot \bigg(\frac{1}{\beta}+\frac{\log \beta}{\alpha/\beta}\bigg) ,\\
E_2 &\coloneqq   \mathop{\sum\sum}_{\substack{\alpha,\beta\in\CA',\ \alpha>\beta \\ 1< [\alpha,\beta]<\min\{\alpha^2,10^\beta\}}} \kappa(\alpha)\kappa(\beta) ,\\
E_3 &\coloneqq \sum_{i\ge0}  \mathop{\sum\sum}_{\substack{\alpha,\beta\in\CA',\ \alpha>\beta \\ 1<[\alpha,\beta]<\min\{\alpha^2,10^\beta\} \\ 2^{i+1}>[\alpha,\beta],\, L(\alpha/\beta;2^i)>1}} \kappa(\alpha)\kappa(\beta)\cdot \frac{i+1}{\log(2[\alpha,\beta])}.
\end{align*}
We first show that the contribution of $E_1$ is negligible. This follows from the following basic lemma.
\begin{lemma}
	\label{lem:contribution of diagonal error terms}
	If $\CC\subset\R_{\ge2}$ is $1$-spaced, then
	\[
	\mathop{\sum\sum}_{\alpha,\beta\in\CC,\ \alpha>\beta} \frac{1}{\alpha\log \alpha} \cdot \frac{1}{\beta\log \beta}  \cdot \bigg(\frac{1}{\beta}+\frac{\log \beta}{\alpha/\beta}\bigg) \ll \kappa(\CC).	
	\]
\end{lemma}

\begin{proof}
	Indeed, the sum in question simplifies as
	\[
	\sum_{\alpha\in \CC} \frac{1}{\alpha\log\alpha} 
		\mathop{\sum}_{\substack{\beta\in\CC \\ \beta<\alpha}}\frac{1}{\beta^2\log \beta}+
	\sum_{\beta\in \CC}\mathop{\sum}_{\substack{\alpha\in\CC \\ \alpha>\beta}} \frac{1}{\alpha^2\log \alpha} \ll 
	\sum_{\alpha\in \CC} \frac{1}{\alpha\log\alpha} 
		+	\sum_{\beta\in\CC} \frac{1}{\beta\log \beta} \asymp \kappa(\CC) ,
	\]
	since $\CC$ is $1$-spaced. This completes the proof. 
\end{proof}

This result, together with Lemma \ref{lem:independence A_i and A_j}(a), implies that
\begin{align}\label{eq:E1}
E_1\ll \kappa(\CA') \asymp_c J.
\end{align}
To bound $E_2$, note that Lemma \ref{lem:independence A_i and A_j}(b) and the condition $[\alpha,\beta]\le 10^\beta$ in its range of summation imply that there exists $j\in\{1,2,\dots,J\}$ such that $\alpha,\beta\in\CA_j'$. Hence,
\[
E_2\le \sum_{j=1}^J \kappa(\CA_j')^2 \ll_c J 
\]
by \eqref{eq:Asymp_Measure_SetsA'j}.

It remains to bound $E_3$. As in the case of $E_2$, there  must exist some $j\in\{1,2,\dots,J\}$ such that $\alpha,\beta\in\CA_j'\subseteq\CA_j$. In particular, the condition $[\alpha,\beta]<\alpha^2$ and the definition of $[\alpha,\beta]$ imply that $H(\alpha/\beta)<x_j^3$. Since $\log\alpha\asymp_c \log x_j$ for each $\alpha\in\CA_j$, we also have that $\kappa(\alpha)\ll_c\lambda(\alpha)/\log x_j$ (and similarly for $\beta$).  
Thus we conclude that
\begin{equation}
	\label{eq:E3 reduction to key prop}
E_3 \ll_c \sum_{j=1}^J \frac{1}{(\log x_j)^2} \sum_{i=0}^\infty \sum_{r=0}^i \frac{i+1}{r+1} \mathop{\sum\sum}_{ \substack{(\alpha,\beta)\in \CA_j^2  \\ H(\alpha/\beta)< x_j^3,\\  2^r<[\alpha,\beta]\le 2^{r+1},\\ L(\alpha/\beta;2^i)>1}} \lambda(\alpha)\lambda(\beta).
\end{equation}
We claim that we have the following key bound:

\begin{proposition}
	\label{prop:key proposition}
	Let $x\ge3$, let $y,z\ge1$, and let $\CB\subseteq[1,x]$ be a $1$-spaced set such that $\alpha/\beta\notin\N$ for all distinct $\alpha,\beta\in\CB$. 
	We have the uniform estimate
	\[
	\lambda\Big(\big\{(\alpha,\beta)\in\CB\times\CB :  H(\alpha/\beta)\le x^3,\ y<[\alpha,\beta]\le 2y,\ L(\alpha/\beta;z)>1\big\}\Big)  \ll y e^{-z} (\log x)^2 .
	\]
\end{proposition}

In view of \eqref{eq:primitivity conditon} and \eqref{eq:E3 reduction to key prop}, we may apply the above proposition with $\CB=\CA_j$, $x=x_j$, $y=2^r$ and $z=2^i$ to find that
\[
E_3 \ll_c \sum_{j=1}^J \sum_{i=0}^\infty \sum_{r=0}^i \frac{i+1}{r+1} \cdot 2^r e^{-2^i} \asymp J \sum_{i=0}^\infty  2^i e^{-2^i} \asymp J .
\]
Therefore, to complete the proof of \eqref{eq:large measure for N_alpha} (and hence of Theorem \ref{thm:solutions}), it suffices to prove Lemmas  \ref{lem:overlap estimate for ratios with large R} and \ref{lem:overlap estimate for rational ratios of typical rationals}, as well as Proposition \ref{prop:key proposition}. 

\medskip 

The remainder of the paper is organized as follows: 
\begin{itemize}
	\item We prove Lemmas  \ref{lem:overlap estimate for ratios with large R} and \ref{lem:overlap estimate for rational ratios of typical rationals} in Section \ref{sec:rational}.
	\item We prove a generalization of Behrend's estimate \eqref{eq:Behrend} in Section \ref{sec:behrend}.
	\item In Sections \ref{sec:bipartite-graphs} and \ref{sec:reduction-to-rationals}, we reduce Proposition \ref{prop:key proposition} to the case when $\CA$ is a set of rational numbers. This is the context of Proposition \ref{prop:key proposition for rationals}.
	\item We introduce the language of GCD graphs and various results about them in Section \ref{sec:GCDgraphs}. Then, in Section \ref{sec:proof-of-moments-bound} we show how to deduce from them Proposition \ref{prop:key proposition for rationals}. The final sections of the paper are devoted to proving the needed results about GCD graphs.
\end{itemize}


\section{The overlap estimate for rational ratios}\label{sec:rational}

The purpose of this section is to establish the following result and deduce Lemmas \ref{lem:overlap estimate for ratios with large R} and \ref{lem:overlap estimate for rational ratios of typical rationals} from it. 

\begin{lemma}\label{lem:overlap estimate for rational ratios}
	Let $\alpha>\beta\ge2$ be two real numbers with $\alpha/\beta\in\Q\setminus \N$, let $y\ge100$ be a parameter, and assume that $R\coloneqq [\alpha,\beta]>1$.  Then, as $T\to\infty$, we have
		\begin{align}
			\label{eq:intersection inequality asymptotic}
			\frac{\P_T\big(\CN_\alpha\cap \CN_\beta\big)}{\P_T\big(\CN_\alpha\big) \P_T\big(\CN_\beta\big)}
			\ 
			& \lesssim 3^{L(\alpha/\beta;y)}   \big( 1+O(\eta)\big)
			+ O\bigg(\frac{\log(2\beta)}{\alpha/\beta}\bigg) 
		\end{align}
		with $\eta=\min\big\{2^\beta R^{-1/2} , 3^{\omega(\alpha/\beta;y)} (1+\log R)/R\big\}$, as well as
		\begin{equation}
			\label{eq:intersection inequality crude}
			\frac{\P_T\big(\CN_\alpha\cap \CN_\beta\big)}{\P_T\big(\CN_\alpha\big) \P_T\big(\CN_\beta\big)}
			\ \ll \ e^{L(\alpha/\beta;R)}+ \frac{\log(2\beta)}{\alpha/\beta} .
		\end{equation}

\end{lemma}

\subsection{Auxiliary results} We shall make frequent use of the following basic estimate about averages of non-negative, divisor-bounded multiplicative functions.

\begin{lemma}[Bounds on multiplicative functions]\label{lem:mult-fnc-bound}
	Let $k\in\N$ and write $\tau_k$ for the $k$-th divisor function. If $f$ is a multiplicative function such that $0\le f\le \tau_k$, then
	\[
	\sum_{n\le x} f(n) \ll_k \frac{x}{\log x} \cdot \exp\bigg(\sum_{p\le x}\frac{f(p)}{p}\bigg)
	\asymp_k x \cdot \exp\bigg(\sum_{p\le x}\frac{f(p)-1}{p}\bigg) \qquad(x\ge2). 
	\]
	Moreover, 
	\[
	\sum_{n\le x} \frac{f(n)}{n} \asymp_k \exp\bigg(\sum_{p\le x}\frac{f(p)}{p}\bigg)\qquad(x\ge2). 
	\]
\end{lemma}

\begin{proof}
	The first part of the lemma follows readily from \cite[Theorem 14.2, p. 145]{dk-book}.	For the second part, see \cite[Exercise 14.5, p. 154]{dk-book}. 
\end{proof}

We also need the following elementary estimate.

\begin{lemma}\label{lem:fejer}
	Let $t\ge1$ be a real number and $q\in\N$.
	\begin{enumerate}
		\item We have
		\[
		\sum_{1\le |j|\le t} (1-|j|/t)  =t-1+O(1/t).
		\]
		\item We have
		\[
		\sum_{\substack{1\le |j|\le t \\ (q,j)=1}}  (1-|j|/t)  =t\prod_{p|q}\bigg(1-\frac{1}{p}\bigg) + O\big(2^{\omega(q)}\big) .
		\]
		\item We have 
				\[
				\sum_{\substack{1\le |j|\le t \\ (j,q)=1}}	\prod_{\substack{p|j \\ p>2}} \frac{p-1}{p-2} \ll t\prod_{\substack{p|q \\ p\le t}}\bigg(1-\frac{1}{p}\bigg) .
				\]
	\end{enumerate}
\end{lemma}

\begin{proof}(a) Using the Euler--Maclaurin summation formula \cite[Theorem 1.10]{dk-book}, we find that
	\begin{align*}
		\sum_{1\le |j|\le t} (1-|j|/t) 
		 = 2\sum_{0<j\le t}(1-j/t) = 2\int_0^t (1-y/t)\dee y - \frac{2}{t} \int_0^t \{y\} \dee y 
	\end{align*}
	We have $\int_0^t (1-y/t)\dee y=t/2$ and $\int_0^t \{y\}\dee y=t/2+O(1)$, with the last estimate following by periodicity. This completes the proof.
	
	\medskip
	
	(b) By M\"obius inversion and part (a), we have
	\[
	\begin{split}
		\sum_{\substack{1\le |j|\le t \\ (q,j)=1}} (1-|j|/t)
		&= \sum_{1\le |j|\le t} (1-|j|/t) \sum_{d\mid (q,j)}\mu(d) 
		\overset{j=di}{=}\sum_{d\mid q} \mu(d) \sum_{1\le |i|\le t/d} \bigg(1 - \frac{|i|}{t/d}\bigg) \\
        &= \sum_{d\mid q}\mu(d)\big(t/d+O(1)\big)= t\prod_{p\mid q}(1-1/p)+ O(2^{\omega(q)}). 
	\end{split}
	\]
	\medskip
	
	(c) This follows from Lemma \ref{lem:mult-fnc-bound} above applied to $f(n)=\one_{(n,q)=1}\prod_{p|n,\, p>2}\frac{p-1}{p-2}$, which is indeed a non-negative and divisor-bounded multiplicative function.
\end{proof}

\subsection{Proof of Lemma \ref{lem:overlap estimate for rational ratios}} 
    Recall by \eqref{Eq:ExpressionOverlapTwoSets} that
    \begin{align}\label{eq:TPTNaNb}
    \P\big(\CN_\alpha\cap \CN_\beta\big) 
	= \frac{1}{T}\sum_{\substack{P^-(n)>\beta \\ n\beta \le T}}\sum_{P^-(m)>\alpha} \big(1- |m\alpha-n\beta|\big)^+ +O\bigg(\frac{1}{T}\bigg).
    \end{align}
	Here we dropped the condition that $m\alpha\le T$, since for all but $O(1)$ values of $n$, it is a consequence of the condition $ |m\alpha-n\beta|<1$ that is implicit in the above double sum. In addition, note that $|m-nt/s|=|m-n\beta/\alpha|< 1/\alpha\le 1/2$ for $m$ and $n$ as above. Hence, given a $\beta$-rough natural number $n$, there is at most one choice for $m$, which is given by the nearest integer to $nt/s$. This choice is admissible if, and only if, $\|nt/s\|<1/\alpha$ and the resulting $m$ is $\alpha$-rough, in which case we have 
	\[
	|m\alpha-n\beta| = \alpha|m-nt/s| = \alpha \|nt/s\|.
	\]
	We must find a convenient way to express all this data.
	
	Since $\gcd(s,t)=1$, there exist integers $u,v$ such that 
	\begin{equation}
		\label{eq:ell_0-k_0}
		ut=1+vs.
	\end{equation}
    We must then have $\gcd(u,s)=1$, and so we may define the residue class $n\bar u \pmod{s}$. Choosing the unique representative $j$ of this class lying in $(-s/2,s/2]$, there exists a unique way to write
	\[
	n=ks+uj\qquad\text{with}\quad -s/2<j\le s/2,\ j,k\in\Z. 
	\]
	We then find that
	\[
	0<n\beta\le T\qquad\iff\qquad -\frac{uj}{s}<k\le \frac{T}{\beta s}-\frac{uj}{s}.
	\]
	In addition, 
	\begin{equation}
		\label{eq:nt/s}
	\frac{nt}{s} = kt+ \frac{ujt}{s} = kt+vj+\frac{j}{s} .
		\end{equation}
		
		Now, using \eqref{eq:nt/s}, we find readily that $\|nt/s\|=|j|/s$ and so
	\begin{equation}
		\label{eq:nt/s 2}
	|m\alpha-n\beta|=\alpha\|nt/s\| = \alpha|j|/s =|j|/R ,
	\end{equation}
	where we used that $\alpha>\beta$ to find that $R=[\alpha,\beta]=H(\alpha/\beta)/\max\{\alpha,\beta\}=s/\alpha$. Since $R=s/\alpha$, we also find that
	\[
	\frac{T}{\beta s} = \frac{T}{\alpha\beta R} .
	\]
	Next, observe that \eqref{eq:nt/s 2} renders the condition $|m\alpha-n\beta|<1$ equivalent to having $|j|<R$. Since $R = s/\alpha$ and $\alpha>2$, the condition $|j|<R$ is stronger than our prior condition $-s/2<j\le s/2$. In particular, we must have $|j|<s/2$, and thus relation \eqref{eq:nt/s} implies that the nearest integer to $nt/s$ equals $kt+vj=m$. 
	Lastly, note that if $j=0$, then  $m=kt$ and $n=ks$. Consequently, $k$ and $t$ must be $\alpha$-rough, $s$ must be $\beta$-rough, and we must also have $m\alpha=n\beta$. Using Lemma \ref{lem:erdos for rationals}, we deduce further that $R=[\alpha,\beta]>\alpha/\beta$ in the case when $j=0$.
	
	In summary, we see that \eqref{eq:TPTNaNb} becomes
	\begin{equation}
		\label{eq:overlap reduction to Sj}
	\P\big(\CN_\alpha\cap \CN_\beta \big)
	= \frac{1}{T}\sum_{0\le|j|\le R} \bigg(1-\frac{|j|}{R}\bigg)  \cdot S_j  \ + O\Big(\frac{1}{T}\Big),
		\end{equation}
	where
	\[
	S_j \coloneqq  \sum_{\substack{-uj/s<k\le \frac{T}{\alpha\beta R}-uj/s \\ P^-(kt+vj)>\alpha,\ P^-(ks+uj)>\beta}}1
	\]
	for $j\neq0$ and
	\[
	S_0 \coloneqq \one_{P^-(s)>\beta}\one_{P^-(t)>\alpha}\one_{R>\alpha/\beta}
	\sum_{\substack{k\le \frac{T}{\alpha\beta R} \\ P^-(k)>\alpha}} 1.
	\]
	
	For $S_0$, we use directly the fundamental lemma of sieve methods to find that
	\[
	\sum_{\substack{k\le \frac{T}{\alpha\beta R} \\ P^-(k)>\alpha}} 1
	\sim \frac{T}{\alpha\beta R} \prod_{p\le \alpha}\bigg(1-\frac{1}{p}\bigg) 
	\sim \frac{T}{R}\,\P_T\big(\CN_\alpha\big)\P_T\big(\CN_\beta\big)\prod_{p\le \beta}\bigg(1-\frac{1}{p}\bigg)^{-1} .
	\]
	Using Mertens' theorem and that we must have $R>\alpha/\beta$ for $S_0$ to be non-zero, we conclude that 
	\begin{equation}
		\label{eq:overlap S0 bound}
	S_0 \ll T\,\P_T\big(\CN_\alpha\big)\P_T\big(\CN_\beta\big)\frac{\log(2\beta)}{\alpha/\beta}.
		\end{equation}
	
	Let us now estimate $S_j$ when $j\neq0$. Note that if there exists a prime $p\le \beta$ that divides $j$ and $st$, then $S_j=0$. Similarly, if there exists a prime $p\in(\beta,\alpha]$ that divides $j$ and $t$, then $S_j=0$. Lastly, note that if $2\nmid jst$, then $ks+uj\equiv k+u\mod2$ and $kt+vj\equiv k+v\mod 2$. In addition, we have that $u\equiv 1+v\mod 2$ by \eqref{eq:ell_0-k_0}. This implies that $2|(k+u)(k+v)$ for all $k$. In particular, there are no $k$ with $P^-(ks+uj)>\beta$ and $P^-(kt+vj)>\alpha$, so that $S_j=0$ again in this case. 
	
	In conclusion, if we let 
	\[
	Q_1\coloneqq \prod_{p\le \beta,\  p|st}p
	\quad\text{and}\quad 
	Q_2\coloneqq \prod_{\beta<p\le\alpha,\  p|t}p,
	\]
	then we have that 
	\begin{equation}
		\label{eq:overlap Sj conditions}
	\mbox{$S_j=0$ when $\gcd(j,Q_1Q_2)>1$ or when $2\nmid jQ_1$}.
\end{equation}
	
	Finally, we estimate $S_j$ when $\gcd(j,Q_1Q_2)=1$ and $2|jQ_1$. 
	For any prime $p\le \beta$, we have
	\[
	\sum_{\substack{-uj/s<k\le \frac{T}{\alpha\beta R}-uj/s \\ p|(ks+uj)(kt+vj) }}1
	=\frac{T}{\alpha\beta R}\cdot \frac{\nu_1(p)}{p}+O(\nu_1(p)),
	\]
	where
	\[
	\nu_1(p)\coloneqq  \#\big\{ k\in \Z/p\Z : (ks+uj)(kt+vj) \equiv0\mod p \big\}. 
	\]
	If $p|s$, then $p\nmid j$ by our assumption that $\gcd(j,Q_1)=1$. In addition, $p\nmid ut$ by \eqref{eq:ell_0-k_0}. So $\nu_1(p)=1$ for such primes. Similarly, if $p|t$, then $\nu_1(p)=1$. Assume now that $p\nmid st$. Then, $\nu_1(p)=2$, unless 
	\[
	ju \bar{s}\equiv j v\bar{t}\mod p\quad\iff\quad j ut\equiv j vs\mod p \quad\iff\quad j \equiv0\mod p,
	\]
	where we used \eqref{eq:ell_0-k_0}. 
	In conclusion, 
	\[
	\nu_1(p) = \begin{cases} 1&\text{if}\ p|jQ_1  ,\\
		2&\text{otherwise}.
	\end{cases}
	\]
	In particular, $\nu_1(2)=1$ because we have assumed that $2|jQ_1$.
	
	Similarly, for any prime $p\in(\beta,\alpha]$, we have
	\[
	\sum_{\substack{-uj/s<k\le\frac{T}{\alpha\beta R} -uj/s  \\ p|kt+vj}}1
	= \frac{T}{\alpha\beta R} \cdot \frac{\nu_2(p)}{p}+O(\nu_2(p)),
	\]
	where
	\[
	\nu_2(p)\coloneqq  \#\big\{ k\in \Z/p\Z : kt+vj \equiv0\mod p \big\}. 
	\]
	If $p|t$, then $p\nmid j$ by our assumption that $\gcd(j,Q_2)=1$. In addition, we have $p\nmid v$ by \eqref{eq:ell_0-k_0}. So $\nu_2(p)=0$ for such primes. On the other hand, if $p\nmid t$, then $\nu_2(p)=1$. In conclusion, 
	\[
	\nu_2(p) = \begin{cases} 0&\text{if}\ p|Q_2  ,\\
		1&\text{otherwise}.
	\end{cases}
	\]
Using the fundamental lemma of sieve methods \cite[Theorem 18.11]{dk-book} and the fact that $\nu_1(2)=1$, we find that 
	\begin{align}
		S_j &\sim \frac{T}{\alpha\beta R} \prod_{p\le \beta}\bigg(1-\frac{\nu_1(p)}{p}\bigg) \prod_{\beta<p\le \alpha}\bigg(1-\frac{\nu_2(p)}{p}\bigg) \nn
		&= \frac{T}{2\alpha\beta R} \prod_{\substack{3\le p\le \beta \\ p|jQ_1 }}\bigg(1-\frac{1}{p}\bigg)
		\prod_{\substack{3\le p\le \beta \\ p\nmid jQ_1}}		\bigg(1-\frac{2}{p}\bigg)
		\prod_{\substack{\beta<p\le \alpha \\ p\nmid Q_2}}\bigg(1-\frac{1}{p}\bigg)	\nn
		&=\frac{T}{2\alpha\beta R} \prod_{p|jQ_1,\, p>2}\frac{p-1}{p-2} \prod_{p|Q_2}\frac{p}{p-1}  \prod_{3\le p\le \beta}\bigg(1-\frac{2}{p}\bigg)	
		\prod_{\beta<p\le \alpha}\bigg(1-\frac{1}{p}\bigg)
			\label{eq:overlap Sj}
	\end{align}
	as $T\to\infty$. 
	
	Lastly, note that $\P_T\big(\CN_\alpha\big) \P_T\big(\CN_\beta\big)\sim \frac{1}{4\alpha\beta} \prod_{3\le p\le \beta}(1-1/p)^2\prod_{\beta<p\le \alpha}(1-1/p)$. Hence, combining \eqref{eq:overlap reduction to Sj}, \eqref{eq:overlap S0 bound}, \eqref{eq:overlap Sj conditions} and \eqref{eq:overlap Sj}, we deduce that
	\begin{equation}
		\label{eq:intersection asymptotic - almost there}
		\begin{split}
		\frac{\P_T\big(\CN_\alpha\cap \CN_\beta\big)}{\P_T\big(\CN_\alpha\big) \P_T\big(\CN_\beta\big)}
		&\sim \frac{2CS}{R} \prod_{p|Q_1,\, p>2}  \frac{p-1}{p-2}   \prod_{p|Q_2}\frac{p}{p-1}   +O\bigg(\frac{\log(2\beta)}{\alpha/\beta}\bigg), 
			\end{split}
		\end{equation}
	where
	\[
	C\coloneqq \prod_{3\le p\le \beta} \frac{1-2/p}{(1-1/p)^2}
	 \qquad\text{and}\qquad 
	S\coloneqq  \sum_{\substack{1\le |j|\le R \\ \gcd(j,Q_1Q_2)=1,\ 2|jQ_1}} \bigg(1-\frac{|j|}{R}\bigg)
	\prod_{\substack{3\le p\le \beta \\ p|j}} \frac{p-1}{p-2} .
	\]
	
		To prove \eqref{eq:intersection inequality crude}, note that \eqref{eq:intersection asymptotic - almost there} implies that 
	\[
	\lim_{T\to\infty}\frac{\P_T\big(\CN_\alpha\cap \CN_\beta\big)}{\P_T\big(\CN_\alpha\big) \P_T\big(\CN_\beta\big)}
	\ll \frac{S}{R}\cdot \prod_{p|Q_1Q_2}\bigg(1+\frac{1}{p}\bigg) + \frac{\log(2\beta)}{\alpha/\beta} .
	\]
	To bound $S$, we note that $0\le 1-|j|/R\le 1$ for $|j|\le R$ and then apply Lemma \ref{lem:fejer}(c). This completes the proof of \eqref{eq:intersection inequality crude}. 
	
	Let us now prove \eqref{eq:intersection inequality asymptotic}. We wish to estimate $S$. Given $y\ge100$ as in the statement of Lemma \ref{lem:overlap estimate for rational ratios}, let $Q=\prod_{p|Q_1Q_2,\, p\le y}p$. In addition, note that
	\[
	\prod_{\substack{3\le p\le \beta \\ p|j}} \frac{p-1}{p-2} = \sum_{d|j,\ d|P}f(d)
	\quad\text{with}\quad P\coloneqq\prod_{3\le p\le\beta}p\quad\text{and}\quad f(d)\coloneqq \prod_{\substack{p|d \\ p>2 }}\frac{1}{p-2}.
	\]
	Hence,
	\begin{align}
		S&\le  \sum_{\substack{1\le |j|\le R \\ \gcd(j,Q)=1,\ 2|jQ_1}} \bigg(1-\frac{|j|}{R}\bigg)
		\prod_{\substack{3\le p\le \beta \\ p|j}} \frac{p-1}{p-2} \nn
		&\overset{j=di}{=} \sum_{\substack{d|P,\, d\le R \\ \gcd(d,Q)=1}}f(d)
		\sum_{\substack{1\le |i|\le R/d \\ \gcd(i,Q)=1,\ 2|diQ_1}} \bigg(1-\frac{|i|}{R/d}\bigg). 
		\label{eq:overlap sum over j}
	\end{align}
	Fix $d|P$ such that $d\le R$. We claim that 
	\begin{equation}
		\label{eq:overlap sum over i}
		\sum_{\substack{1\le |i|\le R/d \\ \gcd(i,Q)=1,\ 2|diQ_1}} \bigg(1-\frac{|i|}{R/d}\bigg)
		= \frac{R}{2d} \prod_{p|Q,\,p>2}\bigg(1-\frac{1}{p}\bigg)
		+O(2^{\omega(Q)}) .
	\end{equation}
	Indeed, if $2|Q_1$ (and thus $2|Q$), then the condition $2|diQ_1$ holds for all $i$. Hence, using Lemma \ref{lem:fejer}(b) with $t=R/d$ establishes \eqref{eq:overlap sum over i}, since $2|Q$ here. On the other hand, if $2\nmid Q_1$, then the condition $2|diQ_1$ implies that $2|i$ (note here that $d$ is odd because it divides $P$). Writing $i=2i'$, and applying Lemma \ref{lem:fejer}(b) with $t=R/(2d)$ proves \eqref{eq:overlap sum over i} in this case too. 
	
	Combining \eqref{eq:overlap sum over j} and \eqref{eq:overlap sum over i}, we conclude that 
	\[
	S\le \sum_{\substack{d|P,\, d\le R \\ \gcd(d,Q)=1}} f(d)\cdot \bigg( 
	\frac{R}{2d} \prod_{p|Q,\, p>2} \bigg(1-\frac{1}{p}\bigg) 
	+ O(2^{\omega(Q)})\bigg). 
	\]
	In addition, note that
	\[
		\sum_{\substack{d|P,\, d\le R \\ \gcd(d,Q)=1}}\frac{f(d)}{d} \le 	\sum_{\substack{d|P \\ \gcd(d,Q)=1}}\frac{f(d)}{d} 
		= \prod_{\substack{3\le p\le \beta \\ p\nmid Q}} \bigg(1+\frac{f(p)}{p}\bigg) 	
		= \frac{1}{C} \prod_{\substack{p|Q \\ 3\le p\le \beta}}\frac{p(p-2)}{(p-1)^2},
	\]
	since $1+f(p)/p= (p-1)^2/(p(p-2))=(1-2/p)^{-1}(1-1/p)^2$ for $p>2$.
	Moreover, letting $g(n)=\one_{n|P} \cdot nf(n)$, we deduce by the second part of Lemma \ref{lem:mult-fnc-bound} that
	\[
	\sum_{d|P,\, d\le R}f(d)= \sum_{d\le R}\frac{g(d)}{d} \ll \log(2R).
	\]
	Consequently, 
	\begin{align*}
		S&\le \frac{R}{2C}
		\prod_{\substack{p|Q \\ 3\le p\le \beta}}\frac{p-2}{p-1}
		\prod_{\substack{p|Q \\ p>\beta}}\frac{p-1}{p} 
		+ O\big(2^{\omega(Q)}\log(2R)\big) \\
		&=  \bigg(\frac{R}{2C}+O\big(3^{\omega(Q)}\log(2R) \big) \bigg) 	\prod_{\substack{p|Q \\ 3\le p\le \beta}}\frac{p-2}{p-1}
			\prod_{\substack{p|Q \\ p>\beta}}\frac{p-1}{p} .
	\end{align*}
	Since $\{p|Q:3\le p\le \beta\}=\{p|Q_1:3\le p\le y\}$ and $\{p|Q: p> \beta\}=\{p|Q_2: p\le y\}$, we conclude that
	\[
	S\le \bigg(\frac{R}{2C}+O\big(3^{\omega(Q)}\log(2R) \big) \bigg) 
		\prod_{\substack{p|Q_1 \\ 3\le p\le y}}\frac{p-2}{p-1} \prod_{\substack{p|Q_2 \\ p\le y}}\frac{p-1}{p}  .
	\]
	Furthermore, note that $\frac{p}{p-1}\leq \frac{p-1}{p-2}\le 3^{1/p}$ for $p\ge y\ge 100$. Therefore,
	\begin{equation}
\label{eq:S almost there}
	\begin{split}
		S&\le \bigg(\frac{R}{2C}+O\big(3^{\omega(Q)}\log(2R) \big) \bigg)   3^{\sum_{p|st,\, p>y}\frac{1}{p}} 
		\prod_{\substack{p|Q_1 \\ p>2}} \frac{p-2}{p-1}
		\prod_{p|Q_2 }\frac{p-1}{p} .
	\end{split}
\end{equation}
	Finally, note that $\omega(Q)\le \omega(st;y)=\omega(\alpha/\beta;y)$, as well as $\omega(Q)\le \#\{p\le\beta\}+\omega(t)\le 0.5\beta+ 0.4\log t+O(1)$. Since $t=R\beta$, we conclude that
	\[
	3^{\omega(Q)}\log(2R) 
	\ll \min\big\{3^{\omega(\alpha/\beta;y)} \log(2R),  2^\beta R^{1/2}\big\} \le R\cdot \eta.
	\]
	Inserting the above estimate into \eqref{eq:S almost there}, and then combining the resulting inequality with \eqref{eq:intersection asymptotic - almost there} completes the proof of \eqref{eq:intersection inequality asymptotic}, and hence of Lemma \ref{lem:overlap estimate for rational ratios}.\qed

\subsection{Deduction of Lemmas \ref{lem:overlap estimate for ratios with large R} and \ref{lem:overlap estimate for rational ratios of typical rationals} from Lemma \ref{lem:overlap estimate for rational ratios}}

\begin{proof}[Proof of Lemma \ref{lem:overlap estimate for ratios with large R}]
	We use the estimate \eqref{eq:intersection inequality asymptotic} in Lemma \ref{lem:overlap estimate for rational ratios} with $y=\infty$. Let $R=[\alpha,\beta]$, and let us first consider the case when $R\ge \alpha^2$. Write $\alpha/\beta = s/t$ for coprime integers $s,t$. Since $\alpha>\beta$, we have $t<s=R\alpha \le R^{3/2}$, and thus
    $\omega(\alpha/\beta; \infty)=\omega(st) \ll \log R/\log\log R$. Hence, 
    \[
 	3^{\omega(\alpha/\beta;\infty)} \frac{1+\log R}{R}\ll R^{-1/2}\le \frac{1}{\alpha} \le \frac{1}{\beta}.
    \]
    Hence, we have established the claimed estimate in the case when $\alpha/\beta\in\Q$ and $R\ge \alpha^2$.
    
    Finally, let us consider the case when $\alpha/\beta\in\Q$ and $R\ge 10^\beta$. We then note that 
    \[
   	\frac{2^\beta}{R^{1/2}} \ll \frac{1}{\beta}.
    \]
	Hence, Lemma \ref{lem:overlap estimate for rational ratios}, applied with $y=\infty$, completes the proof once again. 
\end{proof}

\begin{proof}[Proof of Lemma \ref{lem:overlap estimate for rational ratios of typical rationals}] 
We apply estimate \eqref{eq:intersection inequality crude} in Lemma \ref{lem:overlap estimate for rational ratios} to find that
	\begin{equation}
		\label{eq:intersection inequality crude encore}
	\frac{\P_T\big(\CN_\alpha\cap \CN_\beta\big)}{\P_T\big(\CN_\alpha\big) \P_T\big(\CN_\beta\big)} 
		\ll e^{L(\alpha/\beta;R)}+ \frac{\log \beta}{\alpha/\beta}
	\end{equation}
	as $T\to\infty$. Clearly, this proves part (a) of the lemma.
	
Let us now prove part (b), in which case $L(\alpha/\beta;R)>1$. By the definition of $z$, we have $L(\alpha/\beta;z)>1\ge L(\alpha/\beta;2z)$. Since $L(\alpha/\beta;R)>1$ here, we must have $R<2z$ as claimed. In addition, we have
\[
L(\alpha/\beta;R)\le \sum_{R<p\le 2z}\frac{1}{p}+ L(\alpha/\beta;2z)
	\le \log\bigg(\frac{\log(2z)}{\log(2R)}\bigg)+ O(1) 
\]
by Mertens' theorem. Inserting this bound into \eqref{eq:intersection inequality crude encore} completes the proof.
\end{proof}


\section{A refinement of Behrend's theorem}\label{sec:behrend}

In order to prove Proposition \ref{prop:key proposition}, we shall need the following generalization of a result due to Behrend:

\begin{theorem}\label{thm:behrend} Let $z\ge y\ge2$, let $\CA\subset\N$ be a primitive set, and let $f$ be a multiplicative function such that $0\le f\le\tau_k$ for some fixed integer $k$. If we let $L=\sum_{p\le y} \frac{f(p)}{p}$, then we have the uniform estimate
	\[
	\sum_{a\in\CA\cap[z/y,z] } \frac{f(a)}{a} \ll_k \frac{\log y}{\sqrt{L+1}} \cdot \exp\bigg(\sum_{p\le z}\frac{f(p)-1}{p}\bigg). 
	\]
\end{theorem}

\begin{proof} Set $\CA'=\CA\cap[z/y,z]$, and define 
	\[
	r(n) \coloneqq \sum_{\substack{a\in\CA',\, a|n \\ n/a\ \text{square-free}}} f(a)f(n/a)
	\qquad\text{and}\qquad
	R\coloneqq \sum_{z/y\le n\le zy } \frac{r(n)}{n}.
	\]
	Notice that
	\begin{equation}
		\label{eq:R lb}
	R \ge \sum_{a\in\CA'} \frac{f(a)}{a} \sum_{\substack{m\le y \\  m\ \text{square-free} }} \frac{f(m)}{m} \gg_k \sum_{a\in\CA'} \frac{f(a)}{a} \cdot e^L,
	\end{equation}
	by the second part of Lemma \ref{lem:mult-fnc-bound}. Next, we give an upper bound on $R$. 
	
	Let $g$ be the multiplicative function defined by $g(p^k)=\max\{f(p^k),f(p)f(p^{k-1})\}$, so that $g(p)=f(p)$. Note that $f(a)f(n/a)\le g(n)$ whenever $a|n$ and $n/a$ is square-free. Hence, 
	\begin{equation}
		\label{eq:from r to r tilde}
	r(n) \le g(n) \#\{a\in\CA' : a|n,\ n/a\ \text{square-free} \} \eqqcolon g(n)\tilde{r}(n).
	\end{equation}
	If $n\le zy$ and $a\in \CA'$, then $n/a\le y^2$ because $a\ge z/y$. In particular, $n/a$ divides the $y^2$-smooth part of the radical of $n$. In addition, if $a$ and $b$ are distinct divisors of $n$ from the set $\CA'$, then $n/a$ does not divide $n/b$ by our assumption that $\CA$ is primitive. Hence, the integers counted by $\tilde{r}(n)$ are in 1-1 correspondence with a set $\CD$ consisting of integers that divide $ \prod_{p\le y^2,\ p|n}p$ and that they do not divide each other. In particular, Sperner's theorem \cite{Sp28}  implies that 
	\[
	\tilde{r}(n) = |\CD| \le 2^{\omega(n;y^2)}/\sqrt{\max\{1,\omega(n;y^2)\}}.
	\]
	Hence, 
	\[
	\tilde{r}(n)\le 
	\begin{cases} 
		2^{\omega(n;y^2)}/\sqrt{L+1}&\text{if}\ \omega(n;y^2)\ge L+1,\\
		2^{L+1} &\text{otherwise}.
	\end{cases}
	\]
	Together with \eqref{eq:from r to r tilde}, this implies that
	\[
	R \le \frac{1}{\sqrt{L+1}} \sum_{z/y\le n\le zy } \frac{g(n)2^{\omega(n;y^2)}}{n} 
	+ 2^{L+1}  \sum_{z/y\le n\le zy}  \frac{g(n)}{n}.
	\]
	Using the first part of Lemma \ref{lem:mult-fnc-bound} and partial summation, we find that
	\[
	\begin{split}
		\sum_{z/y\le n\le zy } \frac{g(n)2^{\omega(n;y^2)}}{n} 
		&\ll_k \frac{\log y}{\log z}  \exp\bigg( \sum_{p\le y^2 } \frac{2f(p)}{p} + \sum_{y^2<p\le zy} \frac{f(p)}{p}\bigg) \\
		&\asymp_k e^{2L} \exp\bigg(  \sum_{y<p\le z} \frac{f(p)-1}{p}\bigg) ,
	\end{split}
	\]
	as well as
	\[
	\begin{split}
		2^{L+1}\sum_{z/y\le n\le zy } \frac{g(n)}{n} 
		&\ll_k  (2e)^L  \exp\bigg( \sum_{y<p\le z} \frac{f(p)-1}{p}\bigg) .
	\end{split}
	\]
	In conclusion, 
	\[
	R \ll_k  \frac{e^{2L}}{\sqrt{L+1}} \exp\bigg(  \sum_{y<p\le z} \frac{f(p)-1}{p}\bigg) .
	\]
	Comparing the above estimate with \eqref{eq:R lb}, and noticing that $L=\log\log y+\sum_{p\le y}\frac{f(p)-1}{p}$ by Mertens' theorem completes the proof.
\end{proof}

\begin{corollary}\label{cor:behrend} Fix $C\ge1$. Let $z\ge y\ge3$ and let $q\in\N\cap[1,y^C]$. 
	In addition, let $\CA\subset\N$ be a primitive set, and let $f$ be a multiplicative function such that $1\le f\le\tau_k$ for some fixed integer $k$. Then
	\[
	\sum_{\substack{a\in\CA\cap[z/y,z] \\ \gcd(a,q)=1}} \frac{f(a)}{a} \ll_{k,C} \frac{\phi(q)}{q}\cdot 
	\frac{\log y} {\sqrt{\log\log y}} \cdot \exp\bigg(\sum_{p\le z}\frac{f(p)-1}{p}\bigg). 
	\]
\end{corollary}

\begin{proof} First of all, note that
\begin{equation}
	\label{eq:behrend-missing primes}
	\sum_{\substack{p|q,\ p>y}} \frac{f(p)}{p}=O_{k,C}(1).
\end{equation}
Indeed, if $p|q$, then we have $p\le q\le y^C$, and thus the above estimate follows by Mertens' theorem and the fact that $0\le f(p)\le k$.
	
We are now ready to prove the corollary, which we will accomplish by applying Theorem \ref{thm:behrend} to the multiplicative function $g(a)\coloneqq f(a)\cdot \one_{\gcd(a,q)=1}$. Using \eqref{eq:behrend-missing primes} and our assumptions that $z\ge y$ and that $f\ge1$, we have
	\[
	\sum_{p\le z}\frac{g(p)}{p} 
	=  \sum_{p\le z}\frac{f(p)}{p} - \sum_{p|q}\frac{f(p)}{p}  +O_{k,C}(1) \le \sum_{p\le z}\frac{f(p)}{p} - \sum_{p|q}\frac{1}{p}  +O_{k,C}(1).
	\]
Thus, Theorem \ref{thm:behrend} implies that
\begin{equation}
	\label{eq:L for corollary 1}
\sum_{\substack{a\in\CA\cap[z/y,z] \\ \gcd(a,q)=1}} \frac{f(a)}{a} \ll_{k,C} \frac{\phi(q)}{q}\cdot 
\frac{\log y} {\sqrt{L+1}} \cdot \exp\bigg(\sum_{p\le z}\frac{f(p)-1}{p}\bigg)
\end{equation}
with $L=\sum_{p\le y}g(p)/p$. Using \eqref{eq:behrend-missing primes} once again, as well as our assumption that $1\le f\le\tau_k$, we find that
\[
	L = \sum_{p\le y}\frac{f(p)}{p} - \sum_{p\le y,\, p|q}\frac{f(p)}{p} 
		\ge \sum_{p\le y}\frac{1}{p} - \sum_{p|q}\frac{k}{p}  +O_{k,C}(1) .
\]
Moreover, $\sum_{p|q}1/p\le \log\log\log y+O_C(1)$ because $q\le y^C$. As a consequence, we have $L+1\gg_{k,C} \log\log y$.  Together with \eqref{eq:L for corollary 1}, this completes the proof of the corollary.
\end{proof}

\section{Weighted bipartite graphs}\label{sec:bipartite-graphs}

\begin{dfn}[weight of a set] Let $\mu:\CS\to\R_{\ge0}$ be a weight function defined on a set $\CS$. Given finite sets $\CT\subseteq\CS$ and $\CE\subseteq\CS^2$, we define their {\it weight} to be 
	\[
\mu(\CT) \coloneqq  \sum_{t\in\CT}\mu(t)
\quad\text{and}\quad 
\mu(\CE) \coloneqq \sum_{(s,t)\in\CE} \mu(s)\mu(t).
\]
\end{dfn}

\begin{dfn}[weighted bipartite graph]\label{dfn:bipartite-graph}
	Let $G$ be a quadruple $(\mu,\CV,\CW,\CE)$ such that
	\begin{enumerate}
		\item $\mu:\R_{>0}\to\R_{>0}$ is a positive weight function;
		\item $\CV$ and $\CW$ are finite sets of positive real numbers;
		\item $\CE\subseteq\CV\times \CW$, that is to say $(\CV,\CW,\CE)$ is a bipartite graph.
	\end{enumerate}			
	We then call $G$ a {\it weighted bipartite graph} with \emph{sets of vertices} $(\CV,\CW)$ and \emph{set of edges} $\CE$. 
	
	In addition, we say that $G$ is {\it non-trivial} if $\CE\neq\emptyset$ or, equivalently, if $\mu(\CE)>0$. (We must then also have that $\mu(\CV),\mu(\CW)>0$.)
\end{dfn}

\begin{dfn}[edge density and $\theta$-weight]
	Let $G = (\mu,\mathcal{V},\mathcal{W},\mathcal{E})$ be a weighted bipartite graph. 
	\begin{enumerate}[label=(\alph*)] 
		\item If $G$ is non-trivial, the (weighted) \emph{edge density} of $G$ is defined by
		\[
		\delta(G) = \frac{\mu(\mathcal{E})}{\mu(\mathcal{V}) \mu(\mathcal{W})}.
		\]
		Otherwise, we define $\delta(G)=0$.
		
		\item Let $\theta\ge1$. The \emph{$\theta$-weight} of $G$ is defined by
		\[
		\mu^{(\theta)}(G):=\delta(G)^{\theta}\mu(\mathcal{V})\mu(\mathcal{W}) .
		\]
			\end{enumerate}
\end{dfn}

\begin{rems*}
(a) Since the edge density of a graph is always $\leq 1$, we have
\begin{equation}\label{Eq:RelationEdgeWeight}
\mu^{(\theta')}(G) \le 	\mu^{(\theta)}(G) \leq \mu(\mathcal{E})
\quad\text{for all}\ \theta'\ge\theta\geq 1.
\end{equation}

\medskip 

(b) If $G$ is non-trivial, then we have that
	\begin{equation}
		\label{eq:theta-weight rewrite}
		\mu^{(\theta)}(G) = \frac{\mu(\CE)^\theta}{\mu(\CV)^{\theta-1}\mu(\CW)^{\theta-1}}. 
	\end{equation}
\end{rems*}

\begin{dfn}[Subgraph]\label{dfn:subgraph} 
	Consider two weighted bipartite graphs $G=(\mu,\CV,\CW,\CE)$ and $G'=(\mu',\CV',\CW',\CE')$. We say that $G'$ is a \emph{subgraph} of $G$ if
		\[
		\mu'=\mu,\quad \CV'\subseteq\CV,\quad \CW'\subseteq\CW,\quad \CE'\subseteq\CE.
		\]
We say that $G'$ is a {\it non-trivial subgraph} of $G$ if $\mu(\CE')>0$, that is to say if $G'$ is non-trivial as a weighted bipartite graph.
\end{dfn}

\begin{dfn}[Maximal graph]\label{maximal graph} 
	Let $G = (\mu,\mathcal{V},\mathcal{W},\mathcal{E})$ be a weighted bipartite graph. We say that $G$ is $\theta$-\emph{maximal} if for every subgraph $G'=(\mu,\CV',\CW',\CE')$ of $G$, we have that $\mu^{(\theta)}(G)\ge \mu^{(\theta)}(G')$. 
\end{dfn}

\begin{rems*}
(a) Hauke--Vazquez--Walker use a related notion in their work \cite{HSW} -- see the definition of $\CE'$ in Section 3 of their paper. They employed it to prove a result that is very similar to Lemma \ref{lem:max->connected}. The notion of maximal graphs also appears in \cite[Definition 5]{DS-quantitative}.

\medskip

(b) For any given graph $G= (\mu,\mathcal{V},\mathcal{W},\mathcal{E})$, since both the sets $\mathcal{V}$ and $\mathcal{W}$ are finite, there exists a maximal subgraph $G'$ of $G$. To see this, just list all  subgraphs of $G$, and pick out one with the largest possible $\theta$-weight. We will use this fact several times in the later sections without mentioning it again.
\end{rems*}

\begin{lemma}
	\label{lem:weight increase implies weight increase for larger values of theta}
	Let $\theta\ge1$, let $G = (\mu,\mathcal{V},\mathcal{W},\mathcal{E})$ be a non-trivial weighted bipartite graph and let $G'$ be a subgraph of $G$ such that $\mu^{(\theta)}(G')>\mu^{(\theta)}(G)$. We then have that
	\[
	\delta(G')>\delta(G) 
	\quad\text{and}\quad 
	\frac{\mu^{(\theta')}(G')}{\mu^{(\theta')}(G)} \ge	\frac{\mu^{(\theta)}(G')}{\mu^{(\theta)}(G)} >1 \quad\text{for all}\ \theta'\ge\theta.
	\]
\end{lemma}

\begin{proof} We have $\mu^{(\theta)}(G)>0$ because $G$ is non-trivial. Thus $\mu^{(\theta)}(G')>\mu^{(\theta)}(G)>0$, which implies that the graph $G'$ is also non-trivial. 
	
	Now, let $G'=(\mu,\CV',\CW',\CE')$. Since $\CV'\subseteq\CV$ and $\CW'\subseteq\CW$, we have
	\[
	\mu^{(\theta)}(G')\le \delta(G')^\theta \mu(\CV)\mu(\CW).
	\]
	Together with our assumption that $\mu^{(\theta)}(G')>\mu^{(\theta)}(G)$, this implies that $\delta(G')>\delta(G)$, as needed. 
	
	Lastly, note that 
	\[
	\frac{\mu^{(\theta')}(G')/\mu^{(\theta')}(G)}{\mu^{(\theta)}(G') / \mu^{(\theta)}(G)}
	= \bigg(\frac{\delta(G')}{\delta(G)}\bigg)^{\theta'-\theta} \ge 1
	\]
	for all $\theta'\ge\theta$, since we have already established that $\delta(G')/\delta(G)\ge1$. 
	This completes the proof of the lemma.
\end{proof}

\begin{rem*} 
In virtue of Lemma \ref{lem:weight increase implies weight increase for larger values of theta}, if $G$ is $\theta$-maximal for some $\theta\ge1$, then it is $\theta'$-maximal for all $\theta'\in[1,\theta]$. Indeed, it this were false, then there would exist some $\theta'\in[1,\theta]$ and a subgraph $G'$ of $G$ such that $\mu^{(\theta')}(G')>\mu^{(\theta')}(G)$. But then Lemma \ref{lem:weight increase implies weight increase for larger values of theta} (with the roles of $\theta'$ and $\theta$ reversed) would imply that $\mu^{(\theta)}(G')>\mu^{(\theta)}(G)$, which contradicts the maximality of $G$.
\end{rem*}

\begin{dfn}[Induced set of edges and subgraph]
Let $G=(\mu,\CV,\CW,\CE)$ be a weighted bipartite graph. If $\CV'\subseteq\CV$ and $\CW'\subseteq\CW$, we define
\[
\CE(\CV',\CW'):=\CE\cap(\CV'\times\CW')
\]
and call it the set of edges {\it induced} by $\CV'$ and $\CW'$. Similarly, we call $(\mu,\CV',\CW',\CE(\CV',\CW'))$ the subgraph of $G$ {\it induced} by $\CV'$ and $\CW'$.
\end{dfn}

\begin{dfn}[Neighbourhood sets and connectivity]\label{dfn:neighborhood sets} Let $G=(\mu,\CV,\CW,\CE)$ be a weighted bipartite graph. 
	\begin{enumerate}
		\item 	We define the \emph{neighbourhood sets} by
		\[
		\Gamma_G(v):=\{w\in\CW:\,(v,w)\in\CE\}\quad\text{for any}\ v\in\CV,
		\]
		and 
		\[
		\Gamma_G(w):=\{v\in\CV:\,(v,w)\in\CE\}
		\quad\text{for any}\ w\in \CW.
		\]			
		\item We say that $G$ is $\eta$-connected if for all $v\in\CV$ and all $w\in\CW$ we have 
		\[
		\mu(\Gamma_G(v))\geqslant \eta\cdot \mu(\CW)
		\quad\text{and}\quad 
		\mu(\Gamma_G(w))\geqslant \eta\cdot \mu(\CV).
		\]
	\end{enumerate}
\end{dfn}	

The notion of maximality implies good connectivity for all vertices.

\begin{lemma}[Maximality implies connectivity]\label{lem:max->connected}
	Let $G=(\mu,\CV,\CW,\CE)$ be a weighted bipartite graph with edge density $\delta>0$. If $G$ is $\theta$-maximal for some $\theta>1$, then $G$ is $(\frac{\theta-1}{\theta}\cdot\delta)$-connected.
\end{lemma}

\begin{proof} We use a simple modification of the argument leading to \cite[Lemma 10.1]{DS} and to \cite[Lemma 10.1]{DS-quantitative}.
	Assume for contradiction that there exists some $v\in\CV$ with $\mu(\Gamma_{G}(v))<\frac{\theta}{\theta-1}\delta\mu(\CW)$, and set $G'=(\mu,\CV\setminus\{v\},\CW,\CE')$ with $\CE'=\CE(\CV\setminus\{v\},\CW)$. We claim that $\mu^{(\theta)}(G')>\mu^{(\theta)}(G)$. Equivalently, we need to show that
	\[
	\bigg(\frac{\mu(\CE')}{\mu(\CE)}\bigg)^\theta > \bigg(\frac{\mu(\CV\setminus\{v\})}{\mu(\CV)}\bigg)^{\theta-1}.
	\]
	Note that $\mu(\CE')>\delta\mu(\CV)\mu(\CW)-\frac{\theta-1}{\theta}\delta \mu(v)\mu(\CW)$, because we have assumed that $\mu(\Gamma_{G}(v))<\frac{\theta-1}{\theta}\delta\mu(\CW)$ and because $\mu(v)>0$ (see part (a) of Definition \ref{dfn:bipartite-graph}). Hence,  $\mu(\CE')/\mu(\CE)>1-\frac{\theta-1}{\theta}\mu(v)/\mu(\CV)$. Here $\frac{\theta}{\theta-1}>1$, and thus $(\mu(\CE')/\mu(\CE))^{\frac{\theta}{\theta-1}}>1-\mu(v)/\mu(\CV)=\mu(V\setminus \{v\})/\mu(\CV)$. This proves that $\mu^{(\theta)}(G')>\mu^{(\theta)}(G)$, in contradiction to the $\theta$-maximality of $G$. We must thus have  $\mu(\Gamma_{G}(v))\ge \frac{\theta-1}{\theta}\delta\mu(\CW)$, as needed.
	
	Similarly, we may disprove the existence of $w\in\CW$ with $\mu(\Gamma_G(w))<\frac{\theta-1}{\theta} \delta\mu(\CV)$. This completes the proof of the lemma.
\end{proof}

Next, we have the following important technical lemma, which is based on some ideas in \cite{DS} -- see relations (7.4) and (7.5) there and the discussion around them.

\begin{lemma}[Subgraph where all vertices have a common neighbor, I]\label{lem:common neighbor 1}
	Let $\theta>2$, let $G=(\mu,\CV,\CW,\CE)$ be a $\theta$-maximal, non-trivial weighted bipartite graph, and let $w_0\in\CW$.  There exists 
	a subgraph $G'=(\mu,\CV',\CW',\CE')$ of $G$ such that:
	\begin{enumerate}
		\item $(v,w_0)\in \CE'$ for all $v\in\CV'$;
		\item for all $w\in\CW'$, there exists $v\in\CV'$ such that $(v,w)\in\CE'$;
		\item $\mu^{(\theta)}(G)\le (1-1/\theta)^{-\theta} \mu^{(\theta-1)}(G')$.
	\end{enumerate}
\end{lemma}

\begin{proof} Let $\delta\coloneqq \delta(G)$. Since $G$ is non-trivial, we must have $\delta>0$. Moreover, since $G$ is maximal, Lemma \ref{lem:max->connected} implies that it is $(\frac{\theta-1}{\theta}\cdot\delta)$-connected. Let $G_1=(\mu,\CV_1,\CW_1,\CE_1)$ with $\CV_1=\Gamma_{G}(w_0)$, $\CW_1=\CW$, and $\CE_1=\CE(\CV_1,\CW_1)$. We have that
\[
\begin{split}
\mu(\CE_1) 
	=\sum_{v\in \Gamma_G(w_0)}\mu(v) \mu(\Gamma_{G}(v))
	&\ge \sum_{v\in \Gamma_{G}(w_0)}\mu(v) \cdot(1-1/\theta)\delta\cdot  \mu(\CW) \\	
	&=  (1-1/\theta)\delta  \cdot \mu(\CV_1)\mu(\CW_1),
\end{split}
\]
where we used that $G$ is $(\frac{\theta-1}{\theta}\cdot\delta)$-connected. In particular, we have $\delta_1\coloneqq\delta(G_1)\ge (1-1/\theta)\delta$. In addition, since $\mu(\CV_1)=\mu(\Gamma_{G}(w_0))\ge (1-1/\theta)\delta \cdot \mu(\CV)$ and $\CW_1=\CW$, we also find that
\[
\delta^\theta \mu(\CV)\mu(\CW)  \le (1-1/\theta)^{-1} \delta^{\theta-1}\mu(\CV_1)\mu(\CW_1) \le (1-1/\theta)^{-\theta} \delta_1^{\theta-1}\mu(\CV_1)\mu(\CW_1).
\]
We have thus proven that $\mu^{(\theta)}(G)\le (1-1/\theta)^{-\theta}\mu^{(\theta-1)}(G_1)$. To complete the proof, let $G'=(\mu,\CV',\CW',\CE')$ be a $(\theta-1)$-maximal subgraph of $G_1$. We then have that 
\[
\mu^{(\theta)}(G)\le  (1-1/\theta)^{-\theta}\mu^{(\theta-1)}(G_1)\le (1-1/\theta)^{-\theta} \mu^{(\theta-1)}(G').
\]
In addition, Lemma \ref{lem:max->connected} implies that $\mu(\Gamma_{G'}(w))>0$ for all $w\in\CW'$; in particular, $\Gamma_{G'}(w)\neq\emptyset$. Lastly, we have that $(v,w_0)\in\CE'$ for all $v\in\CV'$ because $G'$ is a subgraph of $G_1$. This completes the proof. 
\end{proof}

The symmetric version of Lemma \ref{lem:common neighbor 1} obviously also holds:

\begin{lemma}[Subgraph where all vertices have a common neighbor, II]\label{lem:common neighbor 2}
Let $\theta>2$, let $G=(\mu,\CV,\CW,\CE)$ be a $\theta$-maximal, non-trivial weighted bipartite graph, and let $v_0\in\CV$.  There exists 
	a subgraph $G'=(\mu,\CV',\CW',\CE')$ of $G$ such that:
	\begin{enumerate}
		\item $(v_0,w)\in \CE'$ for all $w\in\CW'$;
		\item for all $v\in\CV'$, there exists $w\in\CW'$ such that $(v,w)\in\CE'$;
		\item $\mu^{(\theta)}(G)\le (1-1/\theta)^{-\theta} \mu^{(\theta-1)}(G')$.
	\end{enumerate}
\end{lemma}

\begin{lemma}[Few edges between small sets]\label{lem:SmallSetEdges}
	Let	$G=(\mu,\CV,\CW,\CE)$ be a $\theta$-maximal weighted bipartite graph with edge density $\delta>0$, and let $\eta\in(0,1]$. Then, for all sets $\CA\subseteq\CV$ and $\CB\subseteq\CW$ such that $\mu(\CA)\leqslant \eta\cdot  \mu(\CV)$ and  $\mu(\CB)\leqslant\eta \cdot \mu(\CW)$, we have $\mu(\CE(\CA,\CB))\leqslant \eta^{2-2/\theta} \cdot \mu(\CE)$.
\end{lemma}

\begin{proof}The lemma follows by a simple modification of the proof of \cite[Lemma 11.5]{DS}: let $\eta\in[0,1]$, let $\CA$ and $\CB$ be as in the statement of the lemma, and let $G'$ be the subgraph of $G$ induced by $\CA$ and $\CB$. Then $\mu^{(\theta)}(G')\le \mu^{(\theta)}(G)$ by the maximality of $G$. On the other hand,  \eqref{eq:theta-weight rewrite} implies that $\mu^{(\theta)}(G')/\mu^{(\theta)}(G)\ge (\mu(\CE(\CA,\CB))/\mu(\CE))^\theta \eta^{2-2\theta}$. Comparing the two inequalities completes the proof.
\end{proof}

\section{Reduction of Proposition \ref{prop:key proposition} to the case of rational sets}\label{sec:reduction-to-rationals}

In this section, we use the results of Section \ref{sec:bipartite-graphs} to reduce Proposition \ref{prop:key proposition} to the case of rational sets satisfying the following notion of primitivity:

\begin{dfn}[Set of primitive numerators]
	\label{dfn:numerator primitive}
	Let $\CR\subset\Q_{>0}$. We say that $\CR$ is a {\it set of primitive numerators} if, for each $q\in\N$, the set $\{a\in\N: a/q\in\CR,\ \gcd(a,q)=1\}$ is primitive.
\end{dfn}

With this definition, we have the following special case of Proposition \ref{prop:key proposition}:

\begin{proposition}\label{prop:key proposition for rationals} 
		Let $x\ge3$ and $y,z\ge1$, let $\CR,\CS\subset\{\alpha\in\Q_{\ge1}: H(\alpha)\le x\}$ be two sets of primitive numerators, let $\lambda:\R_{>0}\to\R_{>0}$ be the weight function defined by $\lambda(\alpha)=1/\alpha$, and let
	\[
	\CE \subseteq \Big\{(\rho,\sigma) \in  \CR\times \CS :  y<[\rho,\sigma]\le 2y,\ L(\rho/\sigma;z)>1\Big\} .
	\]
	Consider the weighted bipartite graph $G=(\lambda,\CR,\CS,\CE)$. Then its $3$-weight $\lambda^{(3)}(G)$ satisfies
	\[
	\lambda^{(3)}(G) \ll (y\log x)^2 e^{-4z} .
	\]
\end{proposition}

\begin{proof}[Deduction of Proposition \ref{prop:key proposition} from Proposition \ref{prop:key proposition for rationals}]
	Let $x\ge3$ and $y,z\ge1$, and let $\CB\subset[1,x]$ be $1$-spaced and such that $\alpha/\beta\notin\N$ for all distinct $\alpha,\beta\in\CB$. In addition, set 
	\[
	\CE = \big\{(\alpha,\beta)\in\CB\times\CB :  H(\alpha/\beta)\le x^3,\ y<[\alpha,\beta]\le 2y,\ L(\alpha/\beta;z)>1\big\}
	\]
	and consider the weighted bipartite graph $G=(\lambda,\CB,\CB,\CE)$. Let us denote its density by $\delta$. We may assume that $\delta>0$; otherwise Proposition \ref{prop:key proposition} holds trivially. 
	
	We claim that it suffices to prove that 
		\begin{equation}
		\label{eq:reduction of key prop}
		\lambda^{(4)}(G) = \frac{\lambda(\CE)^4}{\lambda(\CB)^6} \ll (y\log x)^2 e^{-4z}.
	\end{equation}
	Indeed, we have $\lambda(\CB)\ll\log x$ because $\CB$ is $1$-spaced. Hence, \eqref{eq:reduction of key prop} would imply that
	\[
	\lambda(\CE)^4 \ll y^2(\log x)^8 e^{-4z},
	\]
	whence Proposition \ref{prop:key proposition} follows. 
	
	To prove \eqref{eq:reduction of key prop}, let $G_1=(\lambda,\CV_1,\CW_1,\CE_1)$ be a $4$-maximal subgraph of $G$, and let 
	\[
	\gamma=\min(\CV_1\cup\CW_1) .
	\]
	Let us assume that $\gamma\in\CW_1$; the case when $\gamma\in\CV_1$ is treated very similarly. We then apply Lemma \ref{lem:common neighbor 1} with $\theta=4$, $w_0=\gamma$ and $G_{\text{Lemma}\ \ref{lem:common neighbor 1}}=G_1$. Hence, there exists a subgraph $G_2=(\lambda,\CV_2,\CW_2,\CE_2)$ of $G_1$ (and hence of $G$) such that:
		\begin{enumerate}
		\item $(v,\gamma)\in \CE_2$ for all $v\in\CV_2$;
		\item for all $w\in\CW_2$, there exists $v\in\CE_2$ such that $(v,w)\in\CE_2$;
		\item $\lambda^{(4)}(G_1)\ll \lambda^{(3)}(G_2)$.
	\end{enumerate}
	Since $G_1$ is a $4$-maximal subgraph of $G$, we find that $\lambda^{(4)}(G) \leq \lambda^{(4)}(G_1)\ll \lambda^{(3)}(G_2)$. This reduces \eqref{eq:reduction of key prop} to showing that 
	\begin{equation}
		\label{eq:reduction of key prop - 2}
		\lambda^{(3)}(G_2) \ll (y\log x)^2 e^{-4z}.
	\end{equation}
	To prove this estimate, we shall use Proposition \ref{prop:key proposition for rationals}. 
	
	For each $v\in\CV_2\subseteq\CV_1$, we know that $(v,\gamma)\in \CE$, and thus $H(v/\gamma)\le x^3$. Hence, we may write $v=\rho\gamma$ with $H(\rho)\le x^3$. By the minimality of $\gamma$, we have that $\rho\ge1$.  In addition, for each $w\in\CW_2$, we know that there exists $v\in\CV_2$ such that $(v,w)\in\CE_2\subseteq\CE$. Hence, $H(v/w)\le x^3$. But we then find that $H(w/\gamma)\le H(w/v)\cdot H(v/\gamma)\le x^6$. Hence, we may write $w=\gamma\sigma$ with $H(\sigma)\le x^6$ and $\sigma\ge1$ (where we used the minimality of $\gamma$ once again). 
	
	By the above discussion, there exist sets $\CR,\CS\subset\{\alpha\in \Q_{\ge1}:H(\alpha)\le x^6\}$ such that $\CV_2=\gamma\CR=\{\gamma\rho:\rho\in\CR\}$ and $\CW_2=\gamma\CS$. In particular, 
	\[
	\lambda(\CV_2)=\frac{\lambda(\CR)}{\gamma}
	\quad\text{and}\quad 
	\lambda(\CW_2)=\frac{\lambda(\CS)}{\gamma}.
	\]
	Moreover, let $\CE_*=\{(\rho,\sigma)\in\CR\times\CS:(\gamma\rho,\gamma\sigma)\in\CE_2\}$, let $G_*=(\lambda,\CR,\CS,\CE_*)$, and let $\delta_*$ be the density of $G_*$. We then have that
	\[
	\lambda(\CE_2) =\frac{\lambda(\CE_*)}{\gamma^2}
	\quad\text{and}\quad \delta_*=\delta(G_2).
	\]
	We thus find that
	\begin{equation}
		\label{eq:reduction of key prop - 3}
	\lambda^{(3)}(G_2) =\frac{\lambda^{(3)}(G_*)}{\gamma^2} .
	\end{equation}
	Next, we show that we may apply Proposition \ref{prop:key proposition for rationals} to $G_*$. 
	
	Indeed, our assumption that $\alpha/\beta\notin\N$ for all distinct $\alpha,\beta\in\CB$ implies that both $\CR$ and $\CS$ are sets of primitive numerators. In addition, we may check that $[\gamma\rho,\gamma\sigma] = [\rho,\sigma]/\gamma$, and thus $\gamma y<[\rho,\sigma]\le 2\gamma y$ for every $(\rho,\sigma)\in\CE_*$. Lastly, we obviously have that $L(\frac{\gamma\rho}{\gamma\sigma};z)=L(\rho/\sigma;z)$. We may thus apply Proposition \ref{prop:key proposition for rationals} with parameters $y_{\text{Proposition \ref{prop:key proposition for rationals}}}=\gamma y$ and  $x_{\text{Proposition \ref{prop:key proposition for rationals}}}=x^6$ to prove that $\lambda^{(3)}(G_*)\ll (\gamma y\log x)^2e^{-4z}$. Together with \eqref{eq:reduction of key prop - 3}, this completes the proof of  \eqref{eq:reduction of key prop - 2}. We have thus established Proposition \ref{prop:key proposition}.	
\end{proof}

\section{GCD graphs}\label{sec:GCDgraphs}

In this section, we extend the theory of GCD graphs developed in \cite{DS,DS-quantitative} to include the possibility of having vertices in $\Q$. We will then state all the necessary results concerning these graphs in order to establish Proposition \ref{prop:key proposition for rationals}. 

\subsection{Definitions} We start with a series of definitions most of which are a simple adaptation of \cite[Section 6]{DS} and \cite[Section 8]{DS-quantitative}. The only important differences lie in Definitions \ref{dfn:structured-GCD-graph} and \ref{dfn:R(G) of structured} that incorporate ideas from \cite{GW,HSW}.

\begin{dfn}[$p$-adic valuation]
	Given a prime $p$, an integer $k$, and a rational number $\rho>0$, we write $p^k\|\rho$ or, equivalently, $\e_p(\rho)=k$, if we may write $\rho=p^ka/q$ with $p\nmid aq$. 
\end{dfn}

\begin{dfn}[Positive and negative parts of a function]
	Given a real-valued function $f$, we let $f^+=\max\{f,0\}$ and $f^-=\max\{-f,0\}$. 
\end{dfn}

\begin{dfn}[GCD graph]\label{dfn:GCDgraph}
	Let $G$ be a septuple $(\mu,\CV,\CW,\CE,\CP,f,g)$ such that:	
	\begin{enumerate}
		\item $(\mu,\CV,\CW,\CE)$ forms a weighted bipartite graph with $\CV,\CW\subset\Q_{>0}$;
		\item $\CP$ is a set of primes;
		\item $f$ and $g$ are functions from $\CP$ to $\Z$ such that for all $p\in\CP$ and all $(a/q,b/r)\in\CV\times\CW$ with $\gcd(a,q)=\gcd(b,r)=1$ we have:
		\begin{enumerate}
			\item $p^{f^+(p)}|a$ and $p^{g^+(p)}|b$;
			\item $p^{f^-(p)}|q$ and $p^{g^-(p)}|r$;
			\item if $(a/q,b/r)\in\CE$, then $p^{\min\{f^+(p),g^+(p)\}} \| \gcd(a,b)$ and $p^{\min\{f^-(p),g^-(p)\}} \| \gcd(q,r)$.
		\end{enumerate}
	\end{enumerate}			
	We then call $G$ a \emph{GCD graph} with \emph{multiplicative data} $(\CP,f,g)$. We will also refer to $\CP$ as the \emph{set of primes} of $G$. If $\CP=\emptyset$, we say that $G$ has \emph{trivial} set of primes and we view $f=f_\emptyset$ and $g=g_\emptyset$ as two copies of the empty function from $\emptyset$ to $\mathbb{Z}$.
	
	Lastly, we will say that $G$ is {\it non-trivial} if the corresponding weighted bipartite graph from (a) is non-trivial, that is to say, if $\CE\neq\emptyset$, which is equivalent to having $\mu(\CE)>0$.
\end{dfn}

\begin{dfn}[exact GCD graph]\label{exact gcd graph dfn}
Given a GCD graph $G=(\mu,\CV,\CW,\CE,\CP,f,g)$, we define the following notions:
\begin{enumerate}
\item $G$ is {\it exact} if $p^{f(p)}\|v$ and $p^{g(p)}\|w$ for all $(v,w)\in\CV\times\CW$ and $p\in\CP$;
\item $G$ is {\it numerator-exact} if $p^{f^+(p)}\|a$ and $p^{g^+(p)}\|b$ for all $p\in\CP$ and all $(a/q,b/r)\in\CV\times \CW$ in part (c)-(i) of Definition \ref{dfn:GCDgraph};
\item $G$ is {\it denominator-exact} if $p^{f^-(p)}\|q$ and $p^{g^-(p)}\|r$ for all $p\in\CP$ and all $(a/q,b/r)\in\CV\times \CW$ in part (c)-(ii) of Definition \ref{dfn:GCDgraph}.
\end{enumerate}
\end{dfn}

To each GCD graph $G$, and for each real number $\theta\ge1$, we associate a {\it $\theta$-quality} $q^{(\theta)}(G)$. This will be defined similarly to \cite{DS-quantitative}, but we omit the factors $(1-\one_{f(p)=g(p)\neq0}/p)^{-2}$ from it since this will allow us to state more precise results. We also omit the factors $(1-1/p^{\frac{\theta+2}{4}})^{-3}$ that were inserted for mere convenience.

\begin{dfn}[quality]\label{dfn:quality} 
	Let $\theta\ge1$ and let $G = (\mu,\mathcal{V},\mathcal{W},\mathcal{E},\mathcal{P},f,g)$ be a GCD graph. The \emph{$\theta$-quality} of $G$ is defined by
		\[
		q^{(\theta)}(G):= \mu^{(\theta)}(G) \prod_{p\in \mathcal{P}}
		p^{|f(p)-g(p)|}	
			= \delta(G)^\theta \mu(\CV)\mu(\CW) \prod_{p\in \mathcal{P}}
			p^{|f(p)-g(p)|}	. 
		\]
\end{dfn}

\begin{dfn}[GCD subgraph]\label{dfn:GCDsubgraph} 
	Consider two GCD graphs $G=(\mu,\CV,\CW,\CE,\CP,f,g)$ and $G'=(\mu',\CV',\CW',\CE',\CP',f',g')$. 
		\begin{enumerate}
			\item We say that $G'$ is a \emph{GCD subgraph} of $G$ if
	\[
\mu'=\mu,\quad \CV'\subseteq \CV,\quad \CW'\subseteq\CW,\quad \CE'\subseteq \CE,\quad \CP'\supseteq\CP, \quad f'\big|_{\CP}=f,\quad
	g'\big|_{\CP}=g.
	\]

	\item We say that $G'$ is a {\it non-trivial GCD subgraph} of $G$ if $\mu(\CE')>0$.
	
	\item We say that $G'$ is an {\it exact GCD subgraph} of $G$ if for every $p\in\CP'\setminus\CP$ and every $(v,w)\in\CV\times\CW$, we have $p^{f'(p)}\|v$ and $p^{g'(p)}\|w$. Similarly, we define the notions of numerator-exact and denominator-exact GCD subgraphs of $G$.
	\end{enumerate}
\end{dfn}

\begin{rem*}
	Note that an exact GCD subgraph $G'$ of a graph $G$ is not necessarily exact as a graph, since the definition only addresses the $p$-adic valuation of the vertices of $G'$ with respect to the ``new'' primes $p$ in $\CP'\setminus \CP$.  On the other hand, if we know that $G$ is an exact graph, then every exact GCD subgraph of it is an exact GCD graph itself.
	
	Similar remarks hold for the notion of numerator-exact GCD subgraphs and denominator-exact GCD subgraphs.
\end{rem*}

We have the following important special vertex sets and GCD graphs $G_{p^k,p^\ell}$ that are formed by restricting to elements with a given $p$-adic valuation. Note that each $G_{p^k,p^\ell}$ is an example of an exact GCD subgraph of $G$.

\begin{dfn}
	\label{def:special graphs}	
	Let $p$ be a prime number, and let $k,\ell\in\Z$. 
	\begin{enumerate}
		\item If $\CV\subset\Q_{>0}$, we set
		\[
		\CV_{p^k}=\{v\in\CV:p^k\|v\}.
		\]
		
		\item Let $G=(\CV,\CW,\CE)$ be a bipartite graph. We write for brevity
		\[
		\CE_{p^k,p^\ell}:=\CE(\CV_{p^k},\CW_{p^\ell}) .
		\]
		
		\item Let $G=(\mu,\CV,\CW,\CE,\CP,f,g)$ be a GCD graph such that $p\notin \CP$. We then define the septuple
		\[
		G_{p^k,p^\ell}=(\mu,\CV_{p^k},\CW_{p^\ell},\CE_{p^k,p^\ell} ,\CP\cup\{p\},f_{p^k},g_{p^\ell})
		\]
		where the functions $f_{p^k}$, $g_{p^\ell}$ are defined on $\CP\cup\{p\}$ by the relations $f_{p^k}\vert_\CP=f$, $g_{p^\ell}\vert_\CP=g$, 
		\[
		f_{p^k}(p)=k\quad\text{and}\quad g_{p^\ell}(p)=\ell.
		\]
	\end{enumerate}
\end{dfn}

\begin{dfn}\label{dfn:R(G)} 
	Let $G = (\mu,\mathcal{V},\mathcal{W},\mathcal{E},\mathcal{P},f,g)$ be a GCD graph.
	We let $\CR(G)$ be given by
	\[
	\CR(G):= \Big\{p\notin\CP: \exists \big(\tfrac{a}{q},\tfrac{b}{r}\big)\in\CE\ \mbox{with $\gcd(a,q)=\gcd(b,r)=1$ and $p|\gcd(a,b)\gcd(q,r)$}\Big\} .
	\]
	That is to say $\CR(G)$ is the set of primes occurring in a GCD which we haven't yet accounted for. 
	Given two parameters $\theta>2$ and $M\ge2$, we let $C=10^{13}M/(\theta-2)^3$ and we split $\CR(G)$ into two subsets:
	\begin{itemize}
		\item The set $\CR_{\theta,M}^\sharp(G)$ of all primes $p\in \CR(G)$ satisfying both of the following properties:
		\begin{itemize}[topsep=1em, itemsep=1em]
		\item there exists $k\in\Z$ such that\quad $\ds\frac{\mu(\CV_{p^k})}{\mu(\CV)}\ge 1-\frac{C}{p}\quad\text{and}\quad
		\frac{\mu(\CW_{p^k})}{\mu(\CW)} \ge 1-\frac{C}{p}$;
		\item $q(G_{p^i,p^j})<M\cdot q(G)$\quad for all $(i,j)\in\Z^2$ with $i\neq j$.
	\end{itemize}
	\item The set $\CR_{\theta,M}^\flat(G)\coloneqq \CR(G)\setminus \CR_{\theta,M}^\sharp(G)$.
	\end{itemize}
\end{dfn}

\begin{dfn}[Structured GCD graph]\label{dfn:structured-GCD-graph}\ 
	\begin{enumerate}
		\item  Let $G$ be a GCD graph with edge set $\CE$. We shall say that $G$ is a {\it structured GCD graph} if, for each $p\in\CR(G)$, there exists an integer $k_p$ such that 
	\begin{equation}
	\label{eq:possible valuations along edges} 
	\big(\e_p(v)-k_p,\e_p(w)-k_p\big)\in\big\{(-1,0),(0,-1),(0,0),(0,1),(1,0)\big\} \quad\mbox{for all $(v,w)\in\CE$}.
		\end{equation}
	\item 
	Let $G$ and $G'$ be two GCD graphs. We shall say that $G'$ is a {\it structured GCD subgraph} of $G$ if $G'$ is structured, and it is also a GCD subgraph of $G$.
		\end{enumerate}
\end{dfn}

	\begin{center}
		\begin{figure}
	\begin{tikzpicture}
		\node[circ](vp){ \quad \quad  \ \ $\CV_{p^{k_p}}$ \quad \ \ \quad };
		\node[circ, above of=vp, yshift=6em](vp+1){$\CV_{p^{k_p+1}}$};
		\node[circ, below of=vp, yshift=-6em](vp-1){$\CV_{p^{k_p-1}}$};
		\node[circ, right of=vp, xshift=10em](wp){\quad \quad \ \ $\CW_{p^{k_p}}$\quad \ \ \quad };
		\node[circ, above of=wp, yshift=6em](wp+1){$\CW_{p^{k_p+1}}$};
		\node[circ, below of=wp, yshift=-6em](wp-1){$\CW_{p^{k_p-1}}$};
		\path[draw] (vp) -- (wp);
		\path[draw] (vp+1) -- (wp);
		\path[draw] (vp-1) -- (wp);
		\path[draw] (vp) -- (wp+1);
		\path[draw] (vp) -- (wp-1);
	\end{tikzpicture}
	\caption{Edges in a structured GCD graph.}
	\end{figure}
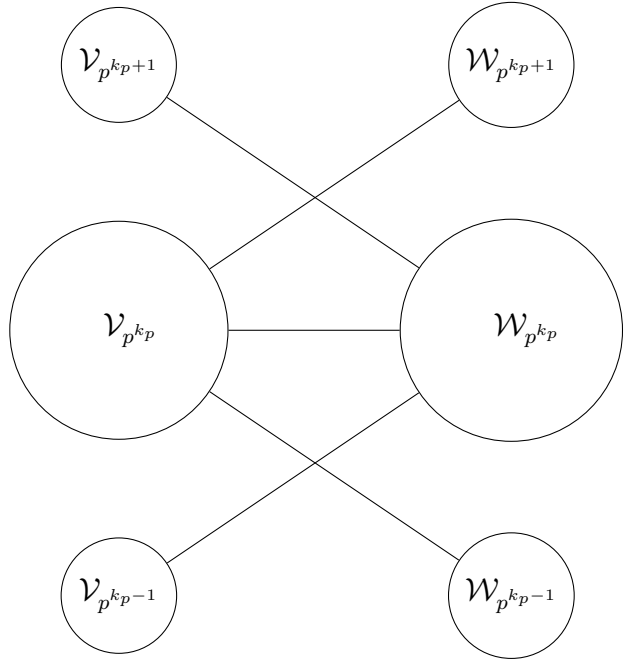
\end{center}

\begin{rem}
	\label{rem:structured-1}
	(a) The choice of $k_p$ might not be unique. For instance, if $(\e_p(v),\e_p(w))=(2,1)$ for all $(v,w)\in\CE$, then \eqref{eq:possible valuations along edges} holds for any choice of $k_p\in\{1,2\}$.
	
	\medskip
	
	(b) Every GCD subgraph of a structured GCD graph is itself a structured GCD graph.
	
	\medskip
	
	(c) Let $G$, $\CE$ and $k_p$ be as in Definition \ref{dfn:structured-GCD-graph}(a). Then we must have $k_p\neq0$. Indeed, by the definition of $\CR(G)$, there must exist some edge $(a/q,b/r)\in\CE$ with $\gcd(a,q)=\gcd(b,r)=1$ and $p|\gcd(a,b)\gcd(q,r)$. However, note that \eqref{eq:possible valuations along edges} implies that $p\nmid \gcd(a,b)\gcd(q,r)$ when $k_p=0$, which is a contradiction. This proves our claim that $k_p\neq0$. 
		In particular, for a structured GCD graph $G$,  each prime $p\in\CR(G)$ has one of the following properties: either
	\begin{equation}
		\label{eq:primes in R^+}	
	\mbox{$\e_p(v)\ge0$ and $\e_p(w) \ge0$ for all $(v,w)\in\CE$,}
	\end{equation}
	or
	\begin{equation}
		\label{eq:primes in R^-}	
		\mbox{$\e_p(v)\le0$ and $\e_p(w) \le 0$ for all $(v,w)\in\CE$.}	
	\end{equation}
	In addition, these properties are mutually exclusive, because we cannot have $\e_p(v)=\e_p(w)=0$ if $(v,w)\in\CE$ and $p\in\CR(G)$. 
\end{rem}

\begin{dfn}\label{dfn:R(G) of structured}
	Let $G=(\mu,\CV,\CW,\CE,\CP,f,g)$ be a structured GCD graph. We then define the sets 
	\[
	\CR_+(G)\coloneqq \{p\in\CR(G): \eqref{eq:primes in R^+}\ \text{holds}\}
	\quad\text{and}\quad 
	\CR_-(G)\coloneqq \{p\in\CR(G): \eqref{eq:primes in R^-}\ \text{holds}\}.
	\]
\end{dfn}

\begin{rem}
	\label{rem:structured}
Let $G$ be a structured GCD graph, $p\in\CR(G)$ and $k_p$ be as in Definition \ref{dfn:structured-GCD-graph}(a). In view of Remark \ref{rem:structured-1}(c), we have
\[
\CR_+(G)=\{p\in\CR(G): k_p>0\}
\quad\text{and}\quad 
\CR_-(G)=\{p\in\CR(G): k_p<0\},
\]
as well as
\[
\CR(G) = \CR_+(G)\sqcup \CR_-(G) .
\]
The reason we define the sets $\CR_+(G)$ and $\CR_-(G)$ in terms of conditions \eqref{eq:primes in R^+} and \eqref{eq:primes in R^-} is so that they do not depend on the choice of $k_p$, which might not be unique (cf.~Remark \ref{rem:structured-1}(a)).
\end{rem}

\subsection{Results on GCD graphs} \label{sec:results on GCD graphs}
Next, we state five key propositions that guarantee the existence of GCD subgraphs with nice properties. We will see in the next section how to deduce Proposition \ref{prop:key proposition for rationals}  (and hence Theorem \ref{thm:solutions}) from them. In their statement, and for the remainder of the paper, we implicitly fix a choice of parameters $\theta\in(2,2.01)$ and $M\ge2$ (in the proof of Proposition \ref{prop:key proposition for rationals}, we shall take $\theta=2.001$ and $M=e^4$), and we set for simplicity
\[
\mu(G)=\mu^{(\theta)}(G)\quad\text{and}\quad 
q(G)=q^{(\theta)}(G),
\]
as well as
\[
\CR^\sharp(G)=\CR_{\theta,M}^\sharp(G)
\quad\text{and}\quad 
\CR^\flat(G)=\CR_{\theta,M}^\flat(G).
\]
Moreover, we shall say that $G$ is maximal to mean that it is $\theta$-maximal as a weighted bipartite graph. 

In addition, it will be convenient to set
\[
\tau\coloneqq \theta-2 \in(0,1/100)
\]
and to introduce the quantities $C_1,\dots,C_8$ as follows:
\begin{equation}\label{eq:CDefs}
	\begin{split}
		&C_1
		=10^4/\tau,  \qquad 
		C_2 = 10MC_1^3,  
		\qquad  
		C_3 = 10^3 C_1^3 , 
		\qquad 
		C_4 = 10^{10}M^2 C_2^2, \\
		&C_5=\max\big\{C_3, (50\log C_4)^3\big\}, 
		\qquad  C_6 = \max\big\{C_4, 10^4 MC_2, C_2^{10/\tau}\big\}, \\
		&C_7=C_5^{C_6},\qquad C_8=100MC_2 .
	\end{split}
\end{equation}
Note that the constant $C$ in Definition \ref{dfn:R(G)} equals $C_2$.

\medskip

With the above notational conventions, we are ready to state our five key propositions about GCD graphs. The first two are simple modifications of Propositions 8.3 and 8.1 in \cite{DS-quantitative}, respectively. Similarly, the last two propositions will follow by adapting the proof of Proposition 8.2 in \cite{DS-quantitative}. On the other hand, there is no version of Proposition \ref{prop:StructureStep} in \cite{DS,DS-quantitative}. It should be noted though that the ideas leading to it are all contained in the proofs of Proposition 8.2 and Lemmas 8.4 and 8.5 in \cite{DS-quantitative}.

\begin{proposition}[Bounded quality loss for small primes]\label{prop:SmallPrimes}
	Let  $G$ be a non-trivial GCD graph with set of primes $\CP$. Then there is a non-trivial and exact GCD subgraph $G'$ of $G$ with set of primes $\CP'$ such that
	\[
	\mathcal{P}' \subseteq \mathcal{P} \cup \bg( \mathcal{R}(G) \cap \{ p\leqslant C_6 \}\bg),
	\quad
	\CR(G')\subseteq\{p> C_6\},
	\quad
	q(G')\ge q(G)/C_7.
	\]
\end{proposition}

\begin{proposition}[Iteration when $\CR^\flat(G)\ne\emptyset$]\label{prop:IterationStep1}
	Let $G$ be a non-trivial GCD graph with set of primes $\CP$ such that
	\[
	\CR(G)\subseteq\{p>C_6\}\quad\text{and}\quad \CR^\flat(G)\neq\emptyset.
	\] 
	Then there is an exact GCD subgraph $G'$ of $G$ with multiplicative data $(\CP',f',g')$ such that:
	\begin{enumerate}[label=(\alph*)]
		\item $G'$ is non-trivial and maximal as a weighted bipartite graph;
		\item $\CP\subsetneq\CP'\subseteq\CP\cup\CR(G)$;
		\item$\CR(G')\subsetneq \CR(G)$;
		\item $q(G')\geqslant M^U q(G)$ with $U=\#\{p\in\CP'\setminus \CP:f'(p)\neq g'(p)\}$.
	\end{enumerate}
\end{proposition}

\begin{proposition}[Structure when $\CR^\flat(G)=\emptyset$]\label{prop:StructureStep}
	Let $G$ be a non-trivial and maximal GCD graph such that
	\[
	\CR(G)\subseteq\{p> C_6\}
	\quad\text{and}\quad \CR^\flat(G)=\emptyset .
	\]
	In addition, for each prime $p$, let $a_{p,1},\dots,a_{p,n}$ be $n$ non-negative weights. Then there is a structured GCD subgraph $G'$ of $G$ with edge set $\CE'$ such that:
	\begin{enumerate}
		\item $G'$ is non-trivial and maximal as a weighted bipartite graph;
		\item $\sum_{p\in v/w,\, p\in\CR(G)} a_{p,i} \le C_8 n\sum_{p\in\CR(G)}\frac{a_{p,i}}{p}$\quad for all $i\in\{1,\dots,n\}$ and all $(v,w)\in\CE'$;
		\item $q(G')\geqslant q(G)/2$.
		\end{enumerate}
\end{proposition}

\begin{proposition}[Iteration when $G$ is structured and $\CR_-(G)\neq\emptyset$]\label{prop:IterationStep2-denominators}
	Let $G$ be a non-trivial and structured GCD graph with set of primes $\CP$ such that 
	\[
	\CR(G)\subseteq\{p> C_6\}
	\quad\text{and}\quad
	\CR_-(G)\neq\emptyset.
	\] 
	Then there is a numerator-exact GCD subgraph $G'$ of $G$ with multiplicative data $(\CP',f',g')$ such that:
	\begin{enumerate}
		\item  $G'$ is non-trivial and maximal as a weighted bipartite graph;
		\item $\CP\subsetneq\CP'\subseteq \CP\cup \CR_-(G)$,\quad $\CR_-(G')\subsetneq \CR_-(G)$,\quad $\CR_+(G')\subseteq \CR_+(G)$;
		\item $f'(p)\le 0$ and $g'(p)\le 0$ for all $p\in\CP'\setminus\CP$;
		\item $q(G')\geqslant q(G)\prod_{p\in\CP'\setminus\CP}  (1-\one_{f'(p)=g'(p)<0}/p)^2 (1-1/p^{1+\frac{\tau}{4}})$.
	\end{enumerate}
\end{proposition}

Finally, even though we do not need it in this paper, we also have the following iteration proposition that we state because it might be useful in future work.

\begin{proposition}[Iteration when $G$ is structured and $\CR_+(G)\neq\emptyset$]\label{prop:IterationStep2-numerators}
	Let $G$ be a non-trivial and structured GCD graph with multiplicative data $(\CP',f',g')$ such that 
	\[
	\CR(G)\subseteq\{p> C_6\}
	\quad\text{and}\quad
	\CR_+(G)\neq\emptyset.
	\] 
	Then there is a denominator-exact GCD subgraph $G'$ of $G$ with set of primes $\CP'$ such that:
	\begin{enumerate}
		\item  $G'$ is non-trivial and maximal as a weighted bipartite graph;
		\item $\CP\subsetneq\CP'\subseteq \CP\cup \CR_+(G)$,\quad $\CR_+(G')\subsetneq \CR_+(G)$,\quad $\CR_-(G')\subseteq \CR_-(G)$;
		\item $f'(p)\ge0$ and $g'(p)\ge0$ for all $p\in\CP'\setminus\CP$;
		\item $q(G')\geqslant q(G)\prod_{p\in\CP'\setminus\CP}  (1-\one_{f'(p)=g'(p)>0}/p)^2 (1-1/p^{1+\frac{\tau}{4}})$.
	\end{enumerate}
\end{proposition}

\section{Proof of Proposition \ref{prop:key proposition for rationals} assuming the results of Section \ref{sec:GCDgraphs}}\label{sec:proof-of-moments-bound}

Here, we show how to deduce Proposition \ref{prop:key proposition for rationals} from Propositions \ref{prop:SmallPrimes}-\ref{prop:IterationStep2-denominators}. Throughout, we take $\tau=0.001$, $\theta=2+\tau=2.001$, and $M=e^4$, and we adopt all notational conventions of Section \ref{sec:results on GCD graphs}. For instance, we will write $\lambda(G)$ and $q(G)$ instead of $\lambda^{(2.001)}(G)$ and $q^{(2.001)}(G)$, and we will say that $G$ is maximal if it is $2.001$-maximal as a weighted bipartite graph.

Let $G=(\lambda,\CR,\CS,\CE)$ with $\CR,\CS$ and $\CE$ as in the statement of Proposition \ref{prop:key proposition for rationals}. With a slight abuse of notation, we view $G$ as the GCD graph $(\lambda,\CR,\CS,\mathcal{E},\emptyset,f_\emptyset,g_\emptyset)$ with trivial multiplicative data. We then have that
\[
\lambda^{(3)}(G) \le \lambda^{(2.001)}(G)= q^{(2.001)}(G)=q(G) 
\]
by relation \eqref{Eq:RelationEdgeWeight} and by Definition \ref{dfn:quality}. Hence, it suffices to show that	
\begin{equation}
	\label{eq:key-prop3 rephrasing}
	q(G) \ll (y\log x)^2 e^{-4z}.
\end{equation}
We may assume that $G$ is non-trivial; otherwise $q(G)=0$, and thus \eqref{eq:key-prop3 rephrasing} holds trivially.

It is now time to employ the results of Section \ref{sec:GCDgraphs}. We first use Proposition \ref{prop:SmallPrimes}, and we then iterate Proposition \ref{prop:IterationStep1} till we arrive at a GCD subgraph $G_1 = (\lambda, \CV_1,\CW_1,\CE_1,\CP_1,f_1,g_1)$ of $G$ such that:
\begin{enumerate}[label=(\alph*)]
	\item $G_1$ is maximal as a weighted bipartite graph;
	\item $G_1$ is exact as a GCD graph;
	\item $\mathcal{R}^\flat(G_1) = \emptyset$;
	\item $q(G_1) \gg e^{4U} q(G)$ with $U=\#\{p\in\CP_1:p>C_6,\ f_1(p)\neq g_1(p)\}$. 
\end{enumerate}
Next, we apply Proposition \ref{prop:StructureStep} with $G_{\text{Proposition}\ \ref{prop:StructureStep}}=G_1$, with $n=1$, and with $a_{p,1}=\one_{p>z}/p$. Hence, we may locate a GCD subgraph 
$G_2 = (\lambda, \CV_2,\CW_2,\CE_2,\CP_2,f_2,g_2)$ of $G_1$ with the same multiplicative data $(\CP_2,f_2,g_2)=(\CP_1,f_1,g_1)$ and such that:
\begin{enumerate}[label=(\alph*)]
	\item $G_2$ is maximal as weighted bipartite graph; 
	\item $G_2$ is exact and structured as a GCD graph;
	\item $\sum_{p\in v/w,\ p\in\CR(G_1),\ p>z}\frac{1}{p} \le C_8 \sum_{p>z} \frac{1}{p^2}\ll \frac{1}{z}$ for all $(v,w)\in\CE_2$;
	\item $q(G_2) \ge q(G_1)/2$.
\end{enumerate}
Since $G_2$ is structured, all of its subgraphs are also structured. Thus, we may iterate Proposition \ref{prop:IterationStep2-denominators} 
till we arrive at a GCD subgraph $G_3 = (\lambda, \CV_3,\CW_3,\CE_3,\CP_3,f_3,g_3)$ of $G_2$ such that:
\begin{enumerate}[label=(\alph*)]
	\item $G_3$ is maximal as a weighted bipartite graph;
	\item $G_3$ is structured and numerator-exact as a GCD graph;
	\item $\mathcal{R}_-(G_3)=\emptyset$;
	\item $q(G_3) \ge q(G_2) \prod_{p\in\CP_3\setminus \CP_2}(1-\one_{f_3(p)=g_3(p)<0}/p)^2(1-1/p^{1+1/4000})$.
\end{enumerate}
In particular, since $G_3$ is a structured GCD graph, for each $p\in\CR(G_3)$ there exists an integer $k_p\in\Z$ such that
\begin{equation}
	\label{eq:possible valuations along edges encore}
	\big(\e_p(v)-k_p,\e_p(w)-k_p\big)\in\big\{(-1,0),(0,-1),(0,0),(0,1),(1,0)\big\} \quad\mbox{for all $(v,w)\in\CE_3$}.
\end{equation}
In addition, $k_p>0$ by Remark \ref{rem:structured} and the fact that $\CR_-(G_3)=\emptyset$. 

Let us define the integers
\[
d^\pm =  \prod_{p\in\CP_2}p^{f_3^\pm(p)},\quad e^\pm =  \prod_{p\in\CP_2}p^{g_3^\pm(p)},
\quad j^\pm =\gcd(d^\pm,e^\pm),
\]
and
\[
D^\pm =  \prod_{p\in\CP_3\setminus \CP_2}p^{f_3^\pm(p)},\quad E^\pm =  \prod_{p\in\CP_3\setminus\CP_2}p^{g_3^\pm(p)},
\quad J^\pm =\gcd(D^\pm,E^\pm),
\]
and
\[
N= \prod_{p\in\CR(G_3)} p^{k_p} ,\quad N^*=\prod_{p|N}p.
\]
The integer $N$ has no prime factor in $\CP_3$ by the definition of $\CR(G_3)$. Moreover, we have 
\[
D^+=E^+=J^+=1 
\]
because $f_3^+(p)=g_3^+(p)=0$ for all $p\in\CP_3\setminus\CP_2$ by property (c) of Proposition \ref{prop:IterationStep2-denominators}. In addition, note that
\begin{equation}
	\label{eq:quality loss factor}
	\prod_{p\in\CP_3}p^{|f_3(p)-g_3(p)|} = \frac{d^+D^+e^+E^+d^-D^-e^-E^-}{(j^+J^+j^-J^-)^2} =  \frac{d^+e^+d^-D^-e^-E^-}{(j^+j^-J^-)^2}  ,
\end{equation}
because $|\chi-\psi|=|\chi^+ - \psi^+| + |\chi^- - \psi^-|$ and $|\chi^\pm-\psi^\pm|=\chi^\pm+\psi^\pm-2\min\{\chi^\pm,\psi^\pm\}$ for all $\chi,\psi\in\R$.  

Given $(v,w)\in\CV_3\times \CW_3$, let
\[
A^\pm= A^\pm(v) = \prod_{\substack{p|N \\ \e_p(v)=k_p\pm1}}p
\quad\text{and}\quad 
B^\pm= B^\pm(w) = \prod_{\substack{p|N \\ \e_p(w)=k_p\pm 1}} p .
\]
Note that if $(v,w)\in\CE_3$, then \eqref{eq:possible valuations along edges encore} implies that the integers $A^+,A^-,B^+,B^-$ are mutually coprime and that
\begin{equation}
	\label{eq:valuation of v and w on primes in R(G_3)}
\e_p(v) = \e_p(NA^+/A^-)\quad\text{and}\quad \e_p(w) = \e_p(NB^+/B^-)\quad\text{for all}\ p\in\CR(G_3) .
\end{equation}
In particular, $\e_p(v),\e_p(w)\ge k_p-1\ge0$ for all $p\in\CR(G_3)$.

Now, given $(v,w)\in\CE_3$, let us write $v=a/q$ and $w=b/r$ with $\gcd(a,q)=\gcd(b,r)=1$. We then find that
\begin{equation}\label{Eq:abqr_factorization}
a=d^+\frac{NA^+}{A^-}\cdot a',\quad 
b=e^+\frac{NB^+}{B^-}\cdot b',\quad 
q=d^-D^-\cdot q',\quad
r=e^-E^- \cdot r'
\end{equation}
for some integers $a',b',r',q'$ such that $\gcd(a',q')=\gcd(b',r')=1$. In addition, observe that
\begin{equation}
	\label{eq:prime factors of a',q',b',r'}
	p|a'q'b'r'\quad\implies\quad p\notin \CR(G_3) ,
\end{equation}
by \eqref{eq:valuation of v and w on primes in R(G_3)} and the comment following it. Moreover, since all elements of $\CV_3\subseteq\CR$ and $\CW_3\subseteq\CS$ are rational numbers $\ge1$ of height $\le x$, we have that $q\le a\le x$ and $r\le b\le x$. Together with \eqref{Eq:abqr_factorization}, this implies
\begin{equation}
	\label{eq:height condition}
	\frac{NA^+}{A^-}\cdot a'\le x, \quad 
	\frac{NB^+}{B^-}\cdot b'\le x
		\quad\text{and}\quad 
		J^- \le x. 
\end{equation}
Furthermore, we have that:
\begin{enumerate}
	\item $A^-A^+B^-B^+|N^*$ because $A^-, A^+, B^-$ and $B^+$ are mutually coprime and square-free;
	\item $\gcd(a'q',N)=\gcd(b'r',N)=1$ by \eqref{eq:prime factors of a',q',b',r'};
	\item $\gcd(a',d^-D^-)=\gcd(b',e^-E^-)=1$ because $\gcd(a,q)=\gcd(b,r)=1$;
	\item $\gcd(a',d^+)=\gcd(b',e^+)=1$ because $G_2$ is exact;
	\item $\gcd(a,b)=j^+N/(A^-B^-)$ and $\gcd(q,r)=j^-J^-$ because: 
	\begin{itemize}
		\item the definition of $\CR(G_3)$ implies that $\gcd(a,b)$ and $\gcd(q,r)$ are solely composed of primes in $\CP_3\sqcup\CR(G_3)$;
		\item if $p\in\CR(G_3)$, then $\e_p(v)=\e_p(NA^+/A^-)$ and $\e_p(w)=\e_p(NB^+/B^-)$ by \eqref{eq:valuation of v and w on primes in R(G_3)};
		\item $\gcd(NA^+/A^-,NB^+/B^-)=N/(A^-B^-)$ by condition (a) above;
		\item if $p \in \CP_3$, then $p^{\min\{f_3^+(p),g_3^+(p)\}}\|\gcd(a,b)$ and $p^{\min\{f_3^-(p),g_3^-(p)\}}\|\gcd(q,r)$ by condition (c)-(iii) in Definition \ref{dfn:GCDgraph};
		\item if $p\in\CP_3$, then $\min\{f_3^\pm(p),g_3^\pm(p)\} = \e_p(j^\pm J^\pm)$ by the definition of $j^\pm$ and of $J^\pm$;
		\item $J^+=1$.
	\end{itemize} 
\end{enumerate}
Using property (e), equation \eqref{Eq:abqr_factorization} and Lemma \ref{lem:R for rational ratio}, we find that
\[
[v,w] = \frac{qr}{\gcd(a,b)\gcd(q,r)} = \frac{d^-D^-e^-E^-}{j^+j^-J^-N} \cdot q'A^-\cdot r'B^-
\]
whenever $(v,w)\in\CE_3$. But we also know that $[v,w]\asymp y$ for all $(v,w)\in\CE_3\subseteq\CE$. Hence, we conclude that
\begin{equation}
	\label{eq:bracket condition}
	q'A^-\cdot r'B^- \asymp y\cdot \frac{j^+j^-J^-N}{d^-D^-e^-E^-}  \quad\text{for all}\ (v,w)\in\CE_3. 
\end{equation}

Now, let $h$ denote the multiplicative function
\[
h(n)=\prod_{\substack{p|n \\ p>z}}e^{4z/p},
\]
and let us note that
\begin{equation}\label{Eq:hfunction_average}
	\sum_p \frac{|h(p)-1|}{p}=O(1)
\end{equation}
because $1\le h(p)\le e^{\one_{p>z}\cdot 4z/p}\le 1+\one_{p>z}\cdot e^4z/p$. We claim that 
\begin{equation}
	\label{eq:remove L}
	e^{4z}\ll e^{4U}h(A^+)h(A^-)h(B^+)h(B^-)h(a')h(q') h(b')h(r')
	\quad\text{for all}\ (v,w)\in\CE_3.
\end{equation}
If $z\le C_6$, this is clearly true because $h\ge1$. So assume that $z>C_6$. We claim that it suffices to show that
\begin{equation}
	\label{eq:remove L-2}
	\sum_{\substack{p\in v/w \\ p>z}} \frac{1}{p} \le \frac{U}{z} + O\bigg(\frac{1}{z}\bigg)+ \sum_{n\in\{A^+,A^-,B^+,B^-,a',q',b',r'\}} \sum_{\substack{p| n \\ p>z}}\frac{1}{p} 
	\quad\text{for all}\ (v,w)\in\CE_3.
\end{equation}
Indeed, we know that $\sum_{p\in v/w,\ p>z}\frac{1}{p}=L(v/w;z)>1$ for all $(v,w)\in\CE_2\subseteq\CE$. Hence, if \eqref{eq:remove L-2} is true, we may then multiply both sides by $4z$ and exponentiate them to deduce \eqref{eq:remove L}.

Let us now prove \eqref{eq:remove L-2}. If $p\in v/w$ and $p>z>C_6$, we have the following four possibilities:
\begin{itemize}
	\item $p\in \CP_1=\CP_2$. Since $G_1$ is exact, we find that $f_1(p)\neq g_1(p)$. This implies that 
	\[
	\sum_{p\in v/w,\ p\in\CP_2,\ p>z} \frac{1}{p}\le \frac{U}{z}.
	\]
	\item $p\in \CP_3\setminus \CP_2$. We then have $p\in\CR(G_2)\subseteq\CR(G_1)$, and thus property (c) of $G_2$ implies that
	\[
	\sum_{p\in v/w,\ p\in\CP_3\setminus \CP_2,\ p>z} \frac{1}{p} \ll \frac{1}{z}.
	\] 
	\item $p\in \CR(G_3)$. Using \eqref{eq:valuation of v and w on primes in R(G_3)}, we find that $p \in \frac{NA^+/A^-}{NB^+/B^-}=\frac{A^+B^-}{A^-B^+}$, and thus that $p|A^+A^-B^+B^-$. 
	\item $p\notin \CP_3\cup \CR(G_3)$. Using \eqref{Eq:abqr_factorization}, we must then have that $p|a'q'b'r'$. 
\end{itemize}
Putting together the above observations proves \eqref{eq:remove L-2}, and thus \eqref{eq:remove L}.

Let us now pick an edge $(v_0,w_0)=(\frac{a_0}{q_0},\frac{b_0}{r_0})\in\CE_3$ (it doesn't matter which one) and let us set 
\begin{equation}
	\label{eq:X and Y}
	X \coloneqq q_0'A_0^-
	\quad\text{and}\quad 
	Y \coloneqq r_0' B^-_0 ,
\end{equation}
where $A_0^- \coloneqq A^-(v_0)$ and $B_0^- \coloneqq B^-(w_0)$. By  \eqref{eq:bracket condition}, we have that
\begin{equation}
	\label{eq:X and Y 2}
	XY\asymp y  \cdot \frac{j^+j^-J^-N}{d^-D^-e^-E^-}.
\end{equation}
We then apply Lemma \ref{lem:common neighbor 1} with $G_{\text{Lemma}\ \ref{lem:common neighbor 1}}=(\lambda,\CV_3,\CW_3,\CE_3)$, $w_0$ as above, and $\theta=2.001$. This proves that there exists a GCD subgraph $G_4=(\lambda,\CV_4,\CW_4,\CE_4,\CP_3,f_3,g_3)$ of $G_3$ such that:
\begin{enumerate}
	\item $(v,w_0)\in \CE_4$ for all $v\in\CV_4$;
	\item for all $w\in\CW_4$, there exists $v\in\CV_4$ such that $(v,w)\in\CE_4$;
	\item $\lambda(G_3)\ll \lambda^{(1.001)}(G_4)$.
\end{enumerate}
In particular, condition (a) and relations \eqref{eq:bracket condition}, \eqref{eq:X and Y} and \eqref{eq:X and Y 2} imply that
\begin{equation}
	\label{eq:all of size X}
	q'A^-\asymp \frac{y}{r_0'B_0^-}\cdot \frac{j^+j^-J^-N}{d^-D^-e^-E^-}    \asymp X \quad\text{for all}\ v\in\CV_4.
\end{equation}
In addition, for every $w\in\CW_4$, condition (b) implies that the existence of some $v\in\CV_4$ such that $(v,w)\in\CE_4\subseteq\CE_3$. Together with  \eqref{eq:bracket condition}, \eqref{eq:all of size X} and \eqref{eq:X and Y 2}, this implies that
\begin{equation}
	\label{eq:all of size Y}
	r'B^- \asymp \frac{y}{q'A^-}\cdot \frac{j^+j^-J^-N}{d^-D^-e^-E^-}  \asymp \frac{y}{X}\cdot \frac{j^+j^-J^-N}{d^-D^-e^-E^-} \asymp Y \quad\text{for all}\ w\in\CW_4.
\end{equation}

We proceed to bound $q(G)$ in terms of $\lambda(\CE_4)$. By the properties of $G_1$ and $G_2$, we have
\[
	q(G) \ll e^{-4U} q(G_1) \ll e^{-4U} q(G_2).
\]
Next, note that if $f_3(p)=g_3(p)<0$ and $p\in\CP_3\setminus\CP_2$, then we must have $p|J^-$. Moreover, we know that $\prod_p (1-1/p^{1+1/4000})\gg1$. Consequently, the properties of $G_3$ and the definition of the quality $q(\cdot)$, imply that
\[
q(G)	\ll e^{-4U} q(G_3) \prod_{p|J^-}\bigg(1-\frac{1}{p}\bigg)^{-2} 
= e^{-4U} \lambda(G_3) \prod_{p\in\CP_3} p^{|f_3(p)-g_3(p)|}  \prod_{p|J^-}\bigg(1-\frac{1}{p}\bigg)^{-2} .
\]
Lastly, note that $\lambda(G_3)\ll \lambda^{(1.001)}(G_4)\le \lambda(\CE_4)$ by the properties of $G_4$ and \eqref{Eq:RelationEdgeWeight}. Together with  \eqref{eq:quality loss factor}, this yields the estimate
\begin{equation}
	\label{eq:quality ineq}
		q(G) \ll e^{-4U} \lambda(\CE_4)\cdot  \frac{d^+e^+d^-D^-e^-E^-}{(j^+j^-J^-)^2} \cdot \prod_{p|J^-}\bigg(1-\frac{1}{p}\bigg)^{-2} .
\end{equation}

Now, consider the modified weight function
\[
\mu(v) \coloneqq \lambda(v) h(A^+)h(A^-)h(a')h(q') .
\]
In virtue of \eqref{eq:remove L}, we have that $e^{4z} \lambda(v)\lambda(w) \ll e^{4U} \mu(v)\mu(w)$ for all $(v,w)\in\CE_3\subseteq\CE_4$, and thus $e^{-4U}\lambda(\CE_4)\ll e^{-4z} \mu(\CE_4)$. Together with \eqref{eq:quality ineq}, this reduces the proof of \eqref{eq:key-prop3 rephrasing}, and thus of Proposition \ref{prop:key proposition for rationals}, to showing that
\begin{equation}
	\label{eq:key-prop3 rephrasing 2}
	\mu(\CE_4) \ll (y\log x)^2 \cdot \frac{(j^+j^-J^-)^2} {d^+e^+d^-D^-e^-E^-} \cdot \prod_{p|J^-}\bigg(1-\frac{1}{p}\bigg)^2.
\end{equation}

To prove \eqref{eq:key-prop3 rephrasing 2}, note that if $(v,w)\in\CE_4\subseteq\CE_3$, then \eqref{Eq:abqr_factorization} and \eqref{eq:all of size X} imply that
\[
\begin{split}
	\mu(v) &=  \frac{d^-D^-}{d^+N} \cdot \frac{A^-}{A^+} \cdot \frac{q'}{a'} \cdot h(A^+)h(A^-)h(a')h(q')\\
	&\asymp\frac{d^-D^-}{d^+N} \cdot \frac{X}{A^+a'} \cdot h(A^+)h(A^-)h(a')h(q') .
\end{split}
\]
Similarly, from \eqref{Eq:abqr_factorization} and \eqref{eq:all of size Y} we deduce that
\[
\mu(w) \asymp  \frac{e^-E^-}{e^+N} \cdot \frac{Y}{B^+b'}  \cdot h(B^+)h(B^-)h(b')h(r') .
\]
Therefore, for all $(v, w) \in \mathcal{E}_4$ we have
\[
\mu(v)\mu(w) \asymp \frac{d^-D^-e^-E^-}{d^+e^+} \cdot \frac{XY}{N^2} \cdot \frac{h(A^+)h(A^-)h(a')h(q')}{A^+a'}\cdot \frac{h(B^+)h(B^-)h(b')h(r')}{B^+b'} .
\]
Moreover, we have
\[
\frac{d^-D^-e^-E^-}{d^+e^+} \cdot \frac{XY}{N^2} 
\asymp y^2\cdot \frac{(j^+j^-J^-)^2}{d^+d^-D^-e^+e^-E^-} \cdot  \frac{1}{XY} 
\]
by \eqref{eq:X and Y 2}. We thus conclude that
\[
\mu(\CE_4) \ll  y^2\cdot \frac{(j^+j^-J^-)^2}{d^+d^-D^-e^+e^-E^-} \cdot \frac{S_1}{X}\cdot \frac{S_2}{Y} ,
\]
where 
\[
S_1 \coloneqq  \mathop{\sum\sum\sum\sum}_{\substack{A^+,A^-,a',q' \in\N \\ A^-A^+|N^*,\  \frac{NA^+}{A^-}\cdot a' \le x,\ A^-q'\asymp X \\ \gcd(a'q',N)=\gcd(a',J^-)=1  \\  
		\frac{a'd^+NA^+/A^-}{q'd^-D^-} \in \CR }}
\frac{h(A^+)h(A^-)h(a')h(q')}{A^+a'} 
\]
and 
\[
S_2 \coloneqq   \mathop{\sum\sum\sum\sum}_{\substack{B^-,B^+,b',r' \in \N\\ B^-B^+|N^*,\ \frac{NB^+}{B^-} \cdot b' \le x,\ B^-r'\asymp X \\ \gcd(b'r',N)=\gcd(b',J^-)=1  \\  
		\frac{b'e^+NB^+/B^-}{r'e^-E^-} \in \CS }}
\frac{h(B^+)h(B^-)h(b')h(r')}{B^+b'} ,
\]
where we used \eqref{eq:height condition}. We have thus reduced the proof of \eqref{eq:key-prop3 rephrasing 2}, and hence of Proposition \ref{prop:key proposition for rationals}, to showing that
\begin{equation}
	\label{eq:key-prop3 rephrasing 4}
	S_1 \ll X\cdot (\log x) \prod_{p|J^-}\bigg(1-\frac{1}{p}\bigg)
	\quad\text{and}\quad
	S_2 \ll Y\cdot (\log x) \prod_{p|J^-}\bigg(1-\frac{1}{p}\bigg).
\end{equation}
We proceed to prove the bound for $S_1$; the sum $S_2$ is treated similarly.

We fix $A^\pm$ and $q'$ and apply Corollary \ref{cor:behrend} to bound the sum over $a'$. Notice that we do not know whether $N\le x^{O(1)}$. Hence, we will only use that $\gcd(a',\frac{NA^+}{A^-}J^-)=1$ because we do know that $NA^+/A^-$ and $J^-$ are both $\le x$ by \eqref{eq:height condition}. 
Therefore, for each fixed choice of $A^\pm$ and $q'$, Corollary \ref{cor:behrend} together with our assumption that $\CR$ is a set of primitive numerators and the estimate \eqref{Eq:hfunction_average} imply that
\[
\sum_{a'}\frac{h(a')}{a'} \ll \frac{\log x}{\sqrt{\log\log x} } \cdot \prod_{p| \frac{NA^+}{A^-}J^-} \bigg(1-\frac{1}{p}\bigg).
\]
Consequently, since $\gcd(N, J^{-})=1$, we obtain
\[
S_1 \ll \frac{\log x}{\sqrt{\log\log x}} \cdot  S_1' \cdot  \prod_{p|J^-}\bigg(1-\frac{1}{p}\bigg),
\]
where 
\[
S_1'\coloneqq 	\mathop{\sum\sum\sum}_{\substack{A^-A^+|N^*  \\  \frac{NA^+}{A^-} \le x ,\ A^-q'\asymp X \\ \gcd(q',\frac{NA^+}{A^-})=1 }}
\frac{h(A^+)h(A^-)h(q')}{A^+}  \prod_{   p| \frac{NA^+}{A^-} } \bigg(1-\frac{1}{p}\bigg).
\]
To establish \eqref{eq:key-prop3 rephrasing 4} for $S_1$, it remains to prove that
\begin{equation}
	\label{eq:key-prop3 rephrasing 5}
	S_1' \ll X\cdot \sqrt{\log\log x} .
\end{equation}

Consider the cut-off parameter
\[
Z\coloneqq \exp\left(\sqrt{\log\log x}\right)
\]
and note that
\begin{equation}
	\label{eq:controlled loss of Euler factors}
	\prod_{   \substack{ p| \frac{NA^+}{A^-} \\ p>Z}} \bigg(1-\frac{1}{p}\bigg)^{-1}\ll \prod_{Z<p\le \log x}\bigg(1-\frac{1}{p}\bigg)^{-1} \asymp \frac{\log\log x}{\log Z} = \sqrt{\log\log x} ,
\end{equation}
because an integer $\le x$ has $\ll\log x$ prime factors. 

Now, we split $S_1'=S_{1,1}'+S_{1,2}'$, where $S_{1,1}'$ is the subsum with $q'>Z$ and $S_{1,2}'$ is the remaining sum. To bound $S_{1,1}'$, let us fix $A^\pm$. Since $q'\asymp X/A^-$, we must have $Z\ll X/A^-$ for this subsum of $S_{1,1}'$ to be non-empty. Hence, if we use Lemma \ref{lem:mult-fnc-bound} with $f(n)=h(n)\cdot \one_{\gcd(n,NA^+/A^-)=1}$ and $x_{\text{Lemma}\ \ref{lem:mult-fnc-bound}}\asymp X/A^-$, we find that
\[
\sum_{\substack{q'\asymp \frac{X}{A^-} \\ \gcd(q',\frac{NA^+}{A^-})=1 }} h(q')
\ll \frac{X}{A^-} \exp\bigg(\sum_{p\le X/A^- } \frac{h(p)\cdot\one_{p\nmid NA^+/A^-} \ -1}{p}\bigg).
\]
We have $h(p) \cdot \one_{p\nmid NA^+/A^-} \ -1 \le |h(p)-1|-\one_{p|NA^+/A^-}$. Together with \eqref{Eq:hfunction_average} and our assumption that $Z\ll X/A^-$, this implies that
\[
\sum_{\substack{q'\asymp \frac{X}{A^-} \\ \gcd(q',\frac{NA^+}{A^-})=1 }} h(q')
\ll \frac{X}{A^-} \exp\bigg( - \sum_{\substack{p\le X/A^- \\ p|NA^+/A^-}} \frac{1}{p}\bigg) 
\ll \frac{X}{A^-} \prod_{  \substack{ p| \frac{NA^+}{A^-}   \\ p\le Z}} \bigg(1-\frac{1}{p}\bigg). 
\]
Combining this estimate with \eqref{eq:controlled loss of Euler factors}, we find that
\[
S'_{1,1} \ll X\sqrt{\log\log x} 
	\mathop{\sum\sum}_{A^-A^+|N^*}
	\frac{h(A^+)h(A^-)}{A^+A^-}  \prod_{   p| \frac{NA^+}{A^-} } \bigg(1-\frac{1}{p}\bigg)^2 
	\leq X\sqrt{\log\log x}  \cdot H(N),
\]
where $H$ is the multiplicative function defined by 
\[
H(N):= \sum_{mn|\prod_{p|N}p} \frac{h(m)h(n)}{mn}  \prod_{   p| \frac{N}{mn} } \bigg(1-\frac{1}{p}\bigg)^2.
\]
Since $m$ and $n$ run over square-free divisors of $N$, and we have $1\leq h(p)\leq e^{4}$ for all primes $p$, we get
\[
H(p^{\ell})=1+\frac{2(h(p)-1)}{p}+O\bigg(\frac{1}{p^2}\bigg),
\]
for all primes $p$ and all integers $\ell\geq 1$. 
Therefore, it follows from \eqref{Eq:hfunction_average} that $H(N)\ll 1$ uniformly in $N\geq 1$ and hence
\begin{equation}
	\label{eq:key-prop3 bound 1}
	S_{1,1}' \ll X\sqrt{\log\log x},
\end{equation}
as needed.

Lastly, let us bound $S_{1,2}'$. Here, we fix $A^+$ and $q'\le Z$, and sum over $A^-$ first. Note that
\[
\prod_{   p| \frac{NA^+}{A^-} } \bigg(1-\frac{1}{p}\bigg) 
\le  \prod_{   p| N } \bigg(1-\frac{1}{p}\bigg) \prod_{p|A^-} \bigg(1-\frac{1}{p}\bigg)^{-1} .
\]
In addition, 
\[
\sum_{A^-\asymp \frac{X}{q'}} h(A^-) \prod_{p|A^-} \bigg(1-\frac{1}{p}\bigg)^{-1} \ll \frac{X}{q'}
\]
by Lemma \ref{lem:mult-fnc-bound} applied with $f(n)=h(n)n/\phi(n)$ and by \eqref{Eq:hfunction_average}. Consequently, 
\[
S_{1,2}'\ll X	\cdot \frac{\phi(N)}{N}  \sum_{q'\le Z}\frac{h(q')}{q'} \sum_{A^+|N} 
\frac{h(A^+)}{A^+}  .
\]
The sum over $q'$ is $\ll \log Z=\sqrt{\log\log x}$ by Lemma \ref{lem:mult-fnc-bound} and the sum over $A^+$ is $\ll N/\phi(N)$ by \eqref{Eq:hfunction_average}. This proves that 
\begin{equation}
	\label{eq:key-prop3 bound 2}
	S_{1,2}' \ll X\cdot \sqrt{\log\log x} ,
\end{equation}
as needed.

Combining \eqref{eq:key-prop3 bound 1} and \eqref{eq:key-prop3 bound 2} completes the proof of \eqref{eq:key-prop3 rephrasing 5}, and thus of Proposition \ref{prop:key proposition for rationals}. This, in turn, completes the proof of Theorem \ref{thm:solutions} modulo proving the results of Section \ref{sec:GCDgraphs}. \qed

\section{Lemmas on GCD graphs and deduction of Propositions \ref{prop:SmallPrimes} and \ref{prop:IterationStep1}}\label{sec:Prep}

\begin{lemma}[Quality variation for special GCD subgraphs]\label{lem:InducedGraphs}
	Let $G=(\mu,\CV,\CW,\CE,\CP,f,g)$ be a non-trivial GCD graph, let $p\in\CR(G)$ and let $k,\ell\in\Z$. If $\mu(\CV_{p^k}),\mu(\CW_{p^\ell})>0$, then
	\[
	\frac{q(G_{p^k,p^\ell})}{q(G)}
	=\bigg(\frac{\mu(\CE_{p^k,p^\ell})}{\mu(\CE)}\bigg)^{2+\tau}
	\bigg(\frac{\mu(\CV)}{\mu(\CV_{p^k})}\bigg)^{1+\tau}
	\bigg(\frac{\mu(\CW)}{\mu(\CW_{p^\ell})}\bigg)^{1+\tau}
	p^{|k-\ell|} .
	\]
\end{lemma}
\begin{proof}
	This follows directly from the definitions. There is a small difference compared to Lemma 11.1 in \cite{DS-quantitative} due to the modified definition of $q(G)$.
\end{proof}

\begin{lemma}[Few edges between unbalanced sets, I]\label{lem:UnbalancedSetEdges1}
	Let	$G=(\mu,\CV,\CW,\CE,\CP,f,g)$ be a non-trivial GCD graph. Let $p\in\CR(G)$, $r\in\Z_{\geqslant1}$ and $k\in\mathbb{Z}$ be such that $p^r> C_4$ and\footnote{If $p\ge C_2$, the last hypothesis is vacuous.}
	\[
	\frac{\mu(\CW_{p^k})}{\mu(\CW)}\geqslant 1-\frac{C_2}{p}\ ;
	\]
	If we set $\CL_{k,r}=\{\ell\in\Z:|\ell-k|\geqslant r+1\}$, then one of the following holds:
	\begin{enumerate}[label=(\alph*)]
		\item There exists $\ell\in\CL_{k,r}$ such that $q(G_{p^k,\,p^\ell})>M\cdot q(G)$.
		\item $\sum_{\ell\in\CL_{k,r}}  \mu(\CE_{p^k,\,p^\ell})\leqslant\mu(\CE)/(4p^{1+\tau/4})$.
	\end{enumerate}
\end{lemma}

\begin{proof} 
This result is Lemma 11.3 in \cite{DS-quantitative}, with the only differences being the following ones: (i) $k$ and $\ell$ are allowed to lie in $\Z$ and not just in $\Z_{\ge0}$ due to our more general definition of GCD graphs that allows rational vertices; (ii) the conclusion in (a) is slightly stronger than in \cite[Lemma 11.3]{DS-quantitative} due to the modified definition of $q(\cdot)$ used here, with the omission of the factors $(1-\one_{f(p)=g(p)\neq0}/p)^{-2}(1-1/p^{1+\tau/4})^{-3}$. Note however that these factors are not used in the proof of \cite[Lemma 11.3]{DS-quantitative}; they are only exploited in \cite[Lemma 15.1]{DS-quantitative}. In conclusion, the argument of \cite{DS-quantitative} goes through to our more general context with minimal changes.
\end{proof}

We also need the symmetric version of Lemma \ref{lem:UnbalancedSetEdges1}:

\begin{lemma}[Few edges between unbalanced sets, II]\label{lem:UnbalancedSetEdges2}
	Let $G=(\mu,\CV,\CW,\CE,\CP,f,g)$ be a non-trivial GCD graph. Let $p\in\CR(G)$,  $r\in\mathbb{Z}_{\geqslant 1}$ and $\ell\in\Z$ be such that $p^r>C_4$ and
	\[
	\frac{\mu(\CV_{p^\ell})}{\mu(\CV)}\geqslant 1-\frac{C_2}{p} .
	\]
	If we set $\CK_{\ell,r}=\{k\in\Z:|\ell-k|\geqslant r+1\}$, then one of the following holds:
	\begin{enumerate}[label=(\alph*)]
		\item There exists $k\in\CK_{\ell,r}$ such that $q(G_{p^k,\,p^\ell})>M\cdot q(G)$.
		\item $\sum_{k\in\CK_{\ell,r}}  \mu(\CE_{p^k,\,p^\ell})\leqslant\mu(\CE)/(4p^{1+\tau/4})$.
	\end{enumerate}
\end{lemma}

\begin{lemma}[Quality increment unless a prime power divides almost all]\label{lem:MainLem}
	Consider a non-trivial GCD graph $G=(\mu,\CV,\CW,\CE,\CP,f,g)$ and a prime $p\in\CR(G)$ with $p>C_2$. Then one of the following holds:
	\begin{enumerate}
		\item There is a non-trivial, maximal and exact GCD subgraph $G'$ of $G$ with multiplicative data $(\CP',f',g')$ such that
		\[	\CP'=\CP\cup\{p\},\quad \CR(G')\subseteq \CR(G)\setminus\{p\},\quad q(G') \geqslant M^{\one_{f'(p)\neq g'(p)}} \cdot q(G).
		\]
		\item There exists $k\in\Z$ such that
		\[
		\frac{\mu(\CV_{p^k})}{\mu(\CV)}\geqslant 1-\frac{C_2}{p}
		\quad\text{and}\quad
		\frac{\mu(\CW_{p^k})}{\mu(\CW)}\geqslant 1-\frac{C_2}{p}.
		\]
	\end{enumerate}
\end{lemma}

\begin{proof}
	This result is Lemma 13.2 in \cite{DS-quantitative}, with the only differences being that we use the modified definition of $q(\cdot)$ in part (a), and that in part (b) we allow $k$ to lie in $\Z$ instead of $\Z_{\ge0}$. The argument of \cite{DS-quantitative} goes through with minimal changes.
\end{proof}

\begin{proof}
	[Deduction of Proposition \ref{prop:IterationStep1}]
	Using the above lemma, we readily establish Proposition \ref{prop:IterationStep1}. For more details, see the deduction of \cite[Proposition 8.1]{DS-quantitative} from \cite[Lemma 13.2]{DS-quantitative}. 
\end{proof}

\begin{lemma}[Small quality loss or a prime power divides a positive proportion]\label{lem:SmallPrime}
	Consider a non-trivial GCD graph $G=(\mu,\CV,\CW,\CE,\CP,f,g)$ and a prime $p\in\CR(G)$. Then one of the following holds:
	\begin{enumerate}[label=(\alph*)]
		\item There is a non-trivial and exact GCD subgraph $G'$ of $G$ with set of primes $\CP'$ such that
		\[
		\CP'=\CP\cup\{p\},\quad
		\CR(G')\subseteq\CR(G)\setminus\{p\},\quad 
		q(G')\ge q(G)/C_3 .
		\]
		\item There is some $k\in\Z$ such that
		\[
		\frac{\mu(\CV_{p^k})}{\mu(\CV)}\geqslant \frac{9}{10}
		\quad\text{and}\quad
		\frac{\mu(\CW_{p^k})}{\mu(\CW)}\geqslant \frac{9}{10}.
		\]
	\end{enumerate}
\end{lemma}

\begin{proof}	This result is Lemma 14.1 in \cite{DS-quantitative}, with the only differences being that we use the modified definition of $q(\cdot)$ in part (a), and that in part (b) we allow $k$ to lie in $\Z$ instead of $\Z_{\ge0}$. The argument of \cite{DS-quantitative} goes through with minimal changes. Note that, even though Lemma 14.1 in \cite{DS-quantitative} does not specify that $G'$ is an exact subgraph of $G$, this fact follows readily from the proof, because we end up taking $G'=G_{p^k,p^\ell}$ for some $k,\ell\in\Z$. 
\end{proof}

\begin{lemma}[Adding small primes to $\CP$]\label{lem:SmallIteration}
	Let $G=(\mu,\CV,\CW,\CE,\CP,f,g)$ be a non-trivial GCD graph and let $p\in\CR(G)$ be a prime. Then there is a non-trivial and exact GCD subgraph $G'$ of $G$ with set of primes $\CP'$  such that 
	\[
	\CP'=\CP\cup\{p\},
	\quad \CR(G')\subseteq\CR(G)\setminus\{p\},
	\quad 
	q(G')\ge q(G)/C_5.
	\]
\end{lemma}

\begin{proof}
	This result is Lemma 14.2 in \cite{DS-quantitative} adapted to our more general definition of GCD graphs with rational vertices. The argument of \cite{DS-quantitative} goes through to this more general context with minimal changes, using Lemmas \ref{lem:MainLem} and \ref{lem:SmallPrime} above. Note that, even though Lemma 14.2 in \cite{DS-quantitative} does not specify that $G'$ is an exact subgraph of $G$, this fact follows readily from the proof, because we end up taking $G'=G^{(1)}_{p^k,p^\ell}$ for some $k,\ell\in\Z$, where $G^{(1)}$ is a maximal GCD subgraph of $G$.
\end{proof}

\begin{proof}[Deduction of Proposition \ref{prop:SmallPrimes}]

	Using the above lemma, we readily establish Proposition \ref{prop:SmallPrimes}. For more details, see the deduction of \cite[Proposition 8.3]{DS-quantitative} from \cite[Lemma 14.2]{DS-quantitative}. 
\end{proof}

\section{Proof of Proposition \ref{prop:StructureStep}}\label{sec:structure}

As in the statement of Proposition \ref{prop:StructureStep}, let $G=(\mu,\CV,\CW,\CE,\CP,f,g)$ be a non-trivial and maximal GCD graph such that 
\[
\CR(G)\subseteq\{p> C_6\}
\quad\text{and}\quad 
\CR^\flat(G)=\emptyset .
\]
If $\CR(G)=\emptyset$, then $G$ is structured for trivial reasons, so that Proposition \ref{prop:StructureStep} follows by taking $G'=G$. Let us thus assume that $\CR(G)\neq\emptyset$.

Now, fix for the moment a prime $p\in\CR(G)$. Since $\CR^\flat(G)=\emptyset$, we must have that $p\in\CR^\sharp(G)$, and thus there exists some integer $k_p$ such that
\begin{equation}
	\label{eq:almost all vertices have valuation k_p}
\frac{\mu(\CV_{p^{k_p}})}{\mu(\CV)}\geqslant1-\frac{C_2}{p}
\quad\text{and}\quad
\frac{\mu(\CW_{p^{k_p}})}{\mu(\CW)}\geqslant 1-\frac{C_2}{p}.
\end{equation}
These inequalities, together with the maximality of $G$ and Lemma \ref{lem:SmallSetEdges}, applied with $\theta=2+\tau$, $\eta=C_2/p$, $\CA=\CV\setminus\CV_{p^{k_p}}$ and $\CB=\CW\setminus\CW_{p^{k_p}}$, imply that
\begin{equation}\label{eq:NoBadEdges1}
	\mu\Big(\CE\big(\CV\setminus\CV_{p^{k_p}} , \CW\setminus \CW_{p^{k_p}}  \big)\Big) 
	\leqslant \bigg(\frac{C_2}{p}\bigg)^{1+\frac{\tau}{2.01}} \mu(\CE) \le \frac{\mu(\CE)}{2 p^{1+\frac{\tau}{4}}},
\end{equation}
where we used that $p>C_6\ge C_2^{10/\tau}$ (cf.~\eqref{eq:CDefs}) and that $\tau\in(0,1/100)$.

Moreover, using again the fact that $p\in\CR^\sharp(G)$, we find that
\begin{equation}\label{eq:no quality increase for asymmetric primes}
	q(G_{p^i,p^j}) 	<M\cdot q(G)
	\quad\text{for all}\ (i,j)\in\Z^2\ \text{with}\ i\neq j.
\end{equation}
In particular, if we set
\begin{align*}
\CV^*_p \coloneqq \bigcup_{i=k_p-1}^{k_p+1}\CV_{p^i}
\quad\text{and}\quad 
\CW^*_p \coloneqq \bigcup_{j=k_p-1}^{k_p+1}\CW_{p^j},
\end{align*}
then Lemmas \ref{lem:UnbalancedSetEdges1} and \ref{lem:UnbalancedSetEdges2} applied with $r=1$ imply that
\begin{equation}\label{eq:NoBadEdges2}
	\mu\Big(\CE\big(\CV\setminus \CV^*_p, \CW_{p^{k_p}} \big)\Big)
	= \sum_{i\in\Z\,:\, |i-k_p|\geqslant2} \mu\big(\CE_{p^i,p^{k_p}}  \big) \leqslant \frac{\mu(\CE)}{4p^{1+\frac{\tau}{4}}},
\end{equation}
and
\begin{equation}\label{eq:NoBadEdges3}
	\mu\Big(\CE\big(\CV_{p^{k_p}} ,   \CW\setminus \CW^*_p   \big)\Big)
	= \sum_{j\in\Z\,:\, |j-k_p|\geqslant2} \mu\big(\CE_{p^{k_p},p^j} \big) 
		\leqslant \frac{\mu(\CE)}{4p^{1+\frac{\tau}{4}}}.
\end{equation}
Hence, if we let
\[
\CE^*_p \coloneqq \CE\big( \CV_{p^{k_p}}, \CW^*_p \big)\cup \CE\big(\CV^*_p, \CW_{p^{k_p}}\big),
\]
then \eqref{eq:NoBadEdges1}, \eqref{eq:NoBadEdges2} and \eqref{eq:NoBadEdges3} imply that
\begin{equation}
	\label{eq:NoBadEdges4}
\mu(\CE\setminus \CE^*_p)\le \frac{\mu(\CE)}{p^{1+\frac{\tau}{4}}} \qquad\text{for all}\ p\in\CR(G).
\end{equation}

Now, let us set
\[
\CE^*\coloneqq \bigcap_{p\in\CR(G)} \CE^*_p.
\]
Using \eqref{eq:NoBadEdges4} and our assumption that $\CR(G)\subseteq\{p>C_6\}$, we find that 
\begin{equation}
	\label{eq:tiny number of edges left out}
\mu\big(\CE\setminus \CE^*\big)\le \sum_{p\in\CR(G)}\mu(\CE\setminus \CE^*_p)
\le \sum_{p>C_6}  \frac{\mu(\CE)}{p^{1+\frac{\tau}{4}}}  \le \frac{10\mu(\CE)}{\tau C_6^{\tau/4}}\le \frac{\mu(\CE)}{100},
\end{equation}
since $\sum_{n>x} n^{-1-c}\le x^{-1-c}+\int_x^\infty y^{-1-c}\dee y =\frac{c+1}{c} x^{-1-c}$ for all $c>0$ and all $x\ge1$, and we have assumed that $C_6\ge C_2^{10/\tau}\ge(10^5/\tau)^{10/\tau}$. Hence, if we let $G^*=(\CV,\CW,\CE^*,\CP,f,g)$, then 
\begin{equation}
	\label{G^*/G}
	\frac{q(G^*)}{q(G)}=\bigg(\frac{\mu(\CE^*)}{\mu(\CE)}\bigg)^{2+\tau} \ge 0.99^{2+\tau} >0.98
\end{equation}
for $\tau\le 0.01$. We may easily check that $G^*$ is a structured GCD graph and a GCD subgraph of $G$.

\medskip

Next, we show that condition (b) of Proposition \ref{prop:StructureStep} holds for some appropriate subgraph of $G^*$ by adapting the argument leading to \cite[Corollary 11.8 and Lemmas 8.4-8.5]{DS-quantitative}. To this end, fix an index $m\in\{1,\dots,n\}$ and recall the weights $a_{p,m}\ge0$ from the statement of Proposition \ref{prop:StructureStep}. Given $\rho\in\Q_{>0}$, let us define $A_m(\rho)=\sum_{p\in\rho,\, p\in\CR(G)}a_{p,m}$. We then have that
\begin{equation}
	\label{eq:anatomy on average}
\sum_{(v,w)\in\CE^*} \mu(v)\mu(w) A_m(v/w) 
= \sum_{p\in\CR(G)} a_{p,m} \cdot \mu\Big(\big\{(v,w)\in \CE^*  : p\in v/w\big\} \Big).
\end{equation}
Note that if $p\in\CR(G)$ is a prime of $v/w$, then we must have that $\e_p(v)\neq \e_p(w)$. Since we also know that $(v,w)\in \CE^*\subseteq \CE^*_p$, we conclude that $(v,w)$ lies in $\CE_{p^i,p^j}$ with $(i,j)\in\{(k_p,k_p+1),(k_p,k_p-1),(k_p-1,k_p),(k_p+1,k_p)\}$. For all these choices of $(i,j)$, we may use relations \eqref{eq:almost all vertices have valuation k_p} and \eqref{eq:no quality increase for asymmetric primes} in conjunction with Lemma \ref{lem:InducedGraphs} to find that
\[
M>\frac{q(G_{p^i,p^j})}{q(G)} = \bigg(\frac{\mu(\CE_{p^i,p^j})}{\mu(\CE)}\bigg)^{2+\tau}
\bigg(\frac{\mu(\CV)}{\mu(\CV_{p^i})}\bigg)^{1+\tau}
\bigg(\frac{\mu(\CW)}{\mu(\CW_{p^j})}\bigg)^{1+\tau}
p
\ge \bigg(\frac{\mu(\CE_{p^i,p^j})}{\mu(\CE)}\bigg)^{2+\tau}
\frac{p^{2+\tau}}{C_2^{1+\tau}} ,
\]
whence
\[
\frac{\mu(\CE_{p^i,p^j})}{\mu(\CE)}< \frac{\big(MC_2^{1+\tau}\big)^{\frac{1}{2+\tau}}}{p} < \frac{MC_2}{p} = \frac{C_8}{100p}.
\]
Therefore, 
\[
\mu\Big(\big\{(v,w)\in \CE^*: p\in v/w\big\} \Big)
	< 4\cdot \frac{C_8}{100p} \cdot \mu(\CE) < \frac{C_8}{20p}\cdot \mu(\CE^*),
\]
in virtue of \eqref{eq:tiny number of edges left out}. Inserting this bound into \eqref{eq:anatomy on average}, we conclude that
\begin{equation}
	\label{eq:anatomy on average 2}
\sum_{(v,w)\in \CE^*} \mu(v)\mu(w) A_m(v/w) < \frac{C_8}{20} \sum_{p\in\CR(G)} \frac{a_{p,m}}{p} \cdot \mu(\CE^*) 
\quad\text{for}\ m=1,2,\dots,n.
\end{equation}

Now, let
\[
\CE^{**}=\bigg\{(v,w)\in \CE^*  : A_m(v/w) \le C_8n \sum_{p\in\CR(G)} \frac{a_{p,m}}{p} \quad\text{for}\ m=1,2,\dots,n\bigg\}.
\]
Using the union bound, Markov's inequality and \eqref{eq:anatomy on average 2}, we find that $\mu( \CE^*\setminus\CE^{**})\le \frac{1}{20}\cdot\mu(\CE^*)$. Hence, if we let $G^{**}=(\CV,\CW,\CE^{**},\CP,f,g)$, then 
\[
	\frac{q(G^{**})}{q(G^*)}=\bigg(\frac{\mu(\CE^{**})}{\mu(\CE^*)}\bigg)^{2+\tau} \ge \bigg(\frac{19}{20}\bigg)^{2+\tau} >0.9. 
\]
Together with \eqref{G^*/G}, this implies that $q(G^{**})\ge q(G)/2$. Clearly, $G^{**}$ is a structured GCD graph and we also have that $A_m(v/w)\le C_8 n\sum_{p\in\CR(G)}\frac{a_{p,m}}{p}$ for all $(v,w)$ in its edge set $\CE^{**}$ and for all $m\in\{1,\dots,n\}$. To complete the proof, we take $G'$ to be a maximal GCD subgraph of $G^{**}$.\qed

\section{Proof of Propositions \ref{prop:IterationStep2-denominators} and \ref{prop:IterationStep2-numerators}}\label{sec:IterationStep2-denominators}

\begin{lemma}[Small quality loss in a structured graph]\label{lem:MainLem2} 
	Consider a maximal, non-trivial and structured GCD graph $G=(\mu,\CV,\CW,\CE,\CP,f,g)$ and a prime $p\in\CR(G)$ with $p>C_6$, and let $k_p\in\Z$ be such that 
		\begin{equation}
		\label{eq:possible valuations along edges encore^2} 
		\big(\e_p(v),\e_p(w)\big)\in\big\{(k_p-1,k_p),(k_p,k_p-1),(k_p,k_p),(k_p,k_p+1),(k_p+1,k_p)\big\}
	\end{equation}
	for all $(v,w)\in\CE$. 	Then there is a non-trivial and maximal GCD subgraph $G'$ of $G$ with multiplicative data $(\CP',f',g')$ such that: 
	\begin{enumerate}
		\item $\CP'=\CP\cup\{p\}$;		
		\item $\CR(G')\subseteq\CR(G)\setminus\{p\}$;
		\item $f'(p),g'(p)\in\{k_p-1,k_p,k_p+1\}$;
		\item if $k_p>0$, then $G'$ is a denominator-exact subgraph of $G$, whereas if $k_p<0$, then $G'$ is a numerator-exact subgraph of $G$;
		\item $q(G')\geqslant q(G)  \big(1-\one_{f'(p)=g'(p)=k_p}/p\big)^2 \big(1-1/p^{1+\frac{\tau}{4}} \big)$.
	\end{enumerate}
\end{lemma}

\begin{proof} 
	Before we begin, let us note that any non-trivial and exact subgraph $G'$ of $G$ automatically satisfies property (c). Indeed, since $G'$ is non-trivial, it has at least one edge, which must satisfy \eqref{eq:possible valuations along edges encore^2}. Since $G'$ is also an exact subgraph of $G$, we conclude that it must satisfy (c).
	
	In addition, note that any $G'$ satisfying (e) is non-trivial. Indeed, we have $q(G)>0$ because $G$ is non-trivial. In virtue of (e), we must also have that $q(G')>0$. Thus, $G'$ is non-trivial.

	\medskip
	
Let us now proceed to the main part of the proof. We shall separate various cases.

\medskip
 
{\it Case 1: $p\in\CR^\flat(G)$.} We must then be in one of following two subcases:

\medskip

{\it Case 1a: there exist distinct integers $i,j$ such that $q(G_{p^i,p^j})\ge M\cdot q(G)$.} We then simply take $G'$ to be a maximal GCD subgraph of $G_{p^i,p^j}$. This completes the proof, since $G'$ is an exact subgraph of $G$ and it satisfies $q(G')\ge M \cdot q(G)>q(G)$.

\medskip

{\it Case 1b: for every $k\in\Z$, we have $\min\{\mu(\CV_{p^k}),\mu(\CW_{p^k})\}<1-C_2/p$.} We then apply Lemma \ref{lem:MainLem} (we must be in case (a) of the lemma by the above assumption) to locate a subgraph $G'$ of $G$ satisfying all desired properties.

\medskip

{\it Case 2: $p\in\CR^\sharp(G)$.} Hence, there exists $k\in\Z$ such that
\begin{equation}
	\label{eq:almost all p^k}
\frac{\mu(\CV_{p^{k}})}{\mu(\CV)}\geqslant1-\frac{C_2}{p}
\quad\text{and}\quad
\frac{\mu(\CW_{p^{k}})}{\mu(\CW)}\geqslant 1-\frac{C_2}{p}.
\end{equation}
To proceed further, we separate into four cases.

\medskip

{\it Case 2a: $k=k_p-1$.} Let 
\[
	\CE^*  \coloneqq 
	\Big\{(v,w)\in\CE: 
	\big(\e_p(v),\e_p(w)\big)\in\big\{(k-1,k),(k,k-1),(k,k),(k,k+1),(k+1,k)\big\} \Big\} .
\]
On the one hand, using \eqref{eq:possible valuations along edges encore^2} and our assumption that $k=k_p-1$, we have that
\begin{equation}
	\label{eq:E^* reduction of choices}
\CE^*  = \Big\{(v,w)\in\CE: \big(\e_p(v),\e_p(w)\big)\in\{(k,k+1),(k+1,k)\big\} \Big\}.
\end{equation}
On the other hand, using \eqref{eq:almost all p^k} and arguing as in the proof of Proposition \ref{prop:StructureStep} in Section \ref{sec:structure}, with $k$ playing the role of $k_p$, we find that 
$\mu(\CE\setminus \CE^*)\le \mu(\CE)/p^{1+\frac{\tau}{4}}$, and thus $\mu(\CE^*)\ge\mu(\CE)/2$. Together with \eqref{eq:E^* reduction of choices}, this implies that either $\mu(\CE_{p^k,p^{k+1}})\ge \mu(\CE)/4$ or $\mu(\CE_{p^{k+1},p^k})\ge \mu(\CE)/4$. In the former case, Lemma \ref{lem:InducedGraphs} yields that
$q(G_{p^k,p^{k+1}})/q(G)\ge 4^{-2-\tau} p\ge 1$. In the latter case, we find similarly that $q(G_{p^{k+1},p^k})\ge q(G)$. We then take $G'$ to be a maximal GCD subgraph of $G_{p^k,p^{k+1}}$ or of $G_{p^{k+1},p^k}$, according to the subcase we are in. In particular, $G'$ is an exact GCD subgraph of $G$. We may then check easily that $G'$ satisfies all required properties. 

\medskip

{\it Case 2b:} $k=k_p+1$. If $\CE^*$ is defined as above, we have $\CE^*  = \{(v,w)\in\CE: (\e_p(v),\e_p(w)\big)\in\{(k-1,k),(k,k-1)\} \}$. Hence, we may argue similarly to Case 2a to complete the proof.

\medskip 

{\it Case 2c:} $k=k_p>0$. Let $G^+=(\mu,\CV^+,\CW^+,\CE^+,\CP\cup\{p\},f_{p^k},g_{p^k})$, where:
	\[
	\CV^+=\CV_{p^k}\cup\CV_{p^{k+1}},
	\quad
	\CW^+=\CW_{p^k}\cup\CW_{p^{k+1}},
	\quad
	\CE^+=\CE\big(\CV^+,\CW^+\big),
	\]
	and $f_{p^k}$ and $g_{p^k}$ are as in Definition \ref{def:special graphs}. It is easy to check that $G^+$ is a GCD subgraph of $G$ satisfying properties (a)-(d)\footnote{$G^+$ is a denominator-exact subgraph of $G$ because if $(a/q,b/r)\in\CV^+\times\CW^+$ with $\gcd(a,q)=\gcd(b,r)=1$, then $\e_p(a/q),\e_p(b/r)\ge k>0$, and thus $\e_p(q)=0=f_{p^k}^-(p)$ and $\e_p(r)=0=g_{p^k}^-(p)$.}. Moreover, both $G_{p^k,p^{k-1}}$ and $G_{p^{k-1},p^k}$ are GCD subgraphs of $G$ satisfying properties (a)-(d). Furthermore, if we follow the proof of Lemma 15.1 in \cite{DS-quantitative} (see Case 2 there), we may show that at least one of the graphs $G^+$, $G_{p^k,p^{k-1}}$ and $G_{p^{k-1},p^k}$ also satisfies property (e). Let us call $G^*$ this GCD subgraph of $G$. To complete the proof of the lemma, we take $G'$ to be a maximal subgraph of $G^*$.
	
\medskip

{\it Case 2d:} $k=k_p<0$. Let $G^-=(\mu,\CV^-,\CW^-,\CE^-,\CP\cup\{p\},f_{p^k},g_{p^k})$, where:
	\[
\CV^-=\CV_{p^{k-1}}\cup\CV_{p^{k}},
\quad
\CW^-=\CW_{p^{k-1}}\cup\CW_{p^k},
\quad
\CE^-=\CE\big(\CV^-,\CW^-\big).
\]
This is a GCD subgraph of $G$ satisfying properties (a)-(d). To complete the proof, we argue similarly to Case 2c, selecting $G'$ to be a maximal subgraph of one of the graphs $G^-$, $G_{p^k,p^{k+1}}$ and $G_{p^{k+1},p^k}$. 

\end{proof}

\begin{proof}[Proof of Proposition \ref{prop:IterationStep2-denominators}] 	This follows almost immediately from Lemma \ref{lem:MainLem2}. Our assumptions that $G$ is structured, that $\CR(G)\subseteq\{p> C_6\}$ and that $\CR_-(G)\neq\emptyset$ imply that there is a prime $p > C_6$ lying in $\CR(G)$ with $k_p<0$ (see Remark \ref{rem:structured}). Thus we can apply Lemma \ref{lem:MainLem2} with this choice of $p$ and complete the proof.
\end{proof}

\begin{proof}[Proof of Proposition \ref{prop:IterationStep2-numerators}] As above, this follows almost immediately from Lemma \ref{lem:MainLem2}.
\end{proof}

\section*{Acknowledgments}

We would like to thank Andrew Granville and Manuel Hauke for useful discussions on Erd\H os's integer dilation approximation problem. In particular, the name of the problem is due to the former.

DK is supported by the Courtois Chair II in fundamental research, by the Natural Sciences and Engineering Research Council of Canada (RGPIN-2018-05699 and RGPIN-2024-05850) and by the Fonds de recherche du Qu\'ebec - Nature et technologies (2022-PR-300951 and 2025-PR-345672).

YL is supported by a junior chair of the Institut Universitaire de France, and was supported by a Simons-CRM visiting professorship while an essential part of this work was carried. In particular, YL would like to thank the Centre de Recherches Math\'ematiques for its excellent working conditions.

JDL is supported by an NSF MSPRF fellowship, and would like to thank the Institut \'Elie Cartan de Lorraine and the Centre de Recherches Math\'ematiques for their hospitality.

\end{document}